\documentclass[10pt,twoside,a4paper]{amsart}

\usepackage[english]{babel}
\usepackage[utf8]{inputenc}

\usepackage{amsfonts,amsthm,amssymb,amsmath,amsthm,latexsym,wasysym,stmaryrd,mathrsfs,dsfont,txfonts}
\usepackage{accents,multirow,rotating,subfigure,graphicx,bigstrut,lscape,multicol,fancyhdr,enumerate,accents,mathtools,exscale}

\usepackage{pdftricks}
\usepackage{color}
\usepackage[pdftex,all]{xy}

\usepackage{hyperref}
\usepackage{appendix}
\usepackage[justification=centering]{caption}

\graphicspath{{Figures/}}

\theoremstyle{theorem}
\newtheorem {theo}{Theorem}[section]
\newtheorem*{theo*}{Theorem}
\newtheorem {lemme}[theo]{Lemma}
\newtheorem*{lemme*}{Lemma}
\newtheorem {prop}[theo]{Proposition}
\newtheorem*{prop*}{Proposition}
\newtheorem {cor}[theo]{Corollary}
\newtheorem*{cor*}{Corollary}

\newtheorem*{cor_proof*}{Corollary (of the proof)}

\newtheorem*{conjecture*}{Conjecture}
\theoremstyle{definition}
\newtheorem {defi}[theo]{Definition}
\newtheorem*{defi*}{Definition}
\newtheorem {nota}[theo]{Notation}
\newtheorem*{nota*}{Notation}
\theoremstyle{remark}
\newtheorem {remarque}[theo]{Remark}
\newtheorem*{remarque*}{Remark}

\newtheorem*{warning*}{Warning}

\newtheorem*{remarques*}{Remarks}
\newtheorem {warnings}[theo]{Warnings}
\newtheorem*{warnings*}{Warnings}

\newtheorem*{convention*}{Convention}
\newtheorem {exemple}[theo]{Example}
\newtheorem*{exemple*}{Example}

\newtheorem*{exemples*}{Examples}

\newtheorem*{question*}{Question}

\newtheorem*{questions*}{Questions}

\newtheorem*{fact*}{Fact}
\newtheorem*{acknowledgments}{Acknowledgments}

\def\N{{\mathds N}}

\def\R{{\mathds R}}

\def\Z{{\mathds Z}}

\def\2Z{{\fract{\Z}/{2\Z}}}

\def\e{\varepsilon}
\def\p{\partial}

\def\Id{{\textnormal{Id}}}

\def\Ker{{\textnormal{Ker}}}

\def\UnN{\llbracket1,n\rrbracket}

\def\ie{{\it i.e. }}

\def\pcup{\operatornamewithlimits{\cup}\limits}
\def\psqcup{\operatornamewithlimits{\sqcup}\limits}

\def\psum{\operatornamewithlimits{\sum}\limits}

\def\fract#1/#2{\hbox{\leavevmode
  \kern.1em \raise .25ex \hbox{\the\scriptfont0 $#1$}\kern-.1em }\big/
  {\hbox{\kern-.15em \lower .5ex \hbox{\the\scriptfont0 $#2$}} }}

\def\fractt#1/#2{\hbox{\leavevmode
  \kern.1em \raise .25ex \hbox{\the\scriptfont0 $#1$}\kern-.1em
}\lower .2ex\hbox{\Big/}
  {\hbox{\kern-.15em \lower .8ex \hbox{\the\scriptfont0 $#2$}} }}

\def\subfract#1/#2{\hbox{\leavevmode
  \kern.1em \raise .25ex \hbox{\the\scriptfont0 \scriptsize $#1$}\kern-.1em }/
  {\hbox{\kern-.15em \lower .5ex \hbox{\the\scriptfont0 \scriptsize $#2$}} }}

\newcommand{\noi}{\noindent}
\newcommand{\disp}{\displaystyle}
\newcommand{\saut}{\vspace{\baselineskip}}
\newcommand{\dessin}[2]{
  \vcenter{\hbox{\includegraphics[height=#1]{#2.pdf}}}}
\newcommand{\dessinH}[2]{
  \vcenter{\hbox{\includegraphics[width=#1]{#2.pdf}}}}
\newcommand{\func}[3]{
  #1 \colon #2 \longrightarrow #3}

\def\hp{{\textrm{h}}}
\def\sv{\textrm{v}}
\def\sa{\textrm{a}}

\def\Ru{\textrm{R1}}
\def\Rd{\textrm{R2}}
\def\Rt{\textrm{R3}}

\def\TC{\textrm{TC}}
\def\OC{\textrm{OC}}
\def\UC{\textrm{UC}}

\def\SA{\textrm{SA}}

\def\Cd{{\textrm{C}^2}}

\def\Cdd{{\textrm{C}^2_2}}
\def\Ct{{\textrm{C}^3}}
\def\Ctu{{\textrm{C}^3_1}}
\def\Ctd{{\textrm{C}^3_2}}
\def\Ctt{{\textrm{C}^3_3}}

\def\GD{\textrm{GD}}

\def\GDn{\GD_n}

\def\GP{\textrm{GP}}

\def\GPn{\GP_n}

\def\rT{\textrm{rT}}

\def\rTn{\rT_n}
\def\rTHn{\rTn^{\hp}}

\def\rL{\textrm{rL}}

\def\rLn{\rL_n}
\def\rLHn{\rLn^{\hp}}

\def\rP{\textrm{rP}}

\def\rPn{\rP_n}
\def\rPHn{\rPn^{\hp}}

\def\wSL{\textrm{wSL}}

\def\wSLn{\wSL_n}
\def\wSLHn{\wSLn^{v}}

\def\wP{\textrm{wP}}

\def\wPn{\wP_n}
\def\wPHn{\wPn^{v}}

\def\wGD{\textrm{wGD}}
\def\wGDH{\wGD^{\sa}}

\def\wGDn{\wGD_n}
\def\wGDHn{\wGDn^{\sa}}

\def\wGP{\textrm{wGP}}

\def\wGPn{\wGP_n}
\def\wGPHn{\wGPn^{\sa}}

\def\vSL{\textrm{vSL}}

\def\vSLn{\vSL_n}

\def\vGD{\textrm{GD}}
\def\vGDH{\vGD^{\sa}}

\def\vGDn{\vGD_n}
\def\vGDHn{\vGDn^{\sa}}

\def\Tube{\textrm{Tube}}

\def\F{\textrm{F}}
\def\Fn{{\F_n}}
\def\RF{\textrm{RF}}
\def\RFn{{\RF_n}}
\def\Aut{\textrm{Aut}}
\def\AutC{\Aut_{\textrm{C}}}

\def\fDA{{\varphi_{\textrm{G}\rightarrow \textrm{A}}}}

\def\fAD{{\varphi_{\textrm{A}\rightarrow \textrm{G}}}}

\def\hG{{G_{\widehat{I}}}}
\def\tB{{B}}

\newcommand{\Ast}{\mbox{\larger[-3]$\ast$}}
\newcommand{\overstar}[1]{\accentset{\Ast}{#1}}
\newcommand{\loverstar}[2]{\makebox[0pt][l]{$\accentset{\Ast}{\phantom{#2}}$}#1}
\def\oB{\overline{B}}
\def\om{\overline{\mu}}

\def\Out{\textrm{Out}}
\def\In{\textrm{In}}
\def\Over{\textrm{Over}}
\def\Cross{\textrm{Cross}}

\begin{document}

\title{Homotopy classification of ribbon tubes and welded string links}
\author[B. Audoux]{Benjamin Audoux}
         \address{Aix Marseille Universit\'e, I2M, UMR 7373, 13453 Marseille, France}
         \email{benjamin.audoux@univ-amu.fr}
\author[P. Bellingeri]{Paolo Bellingeri}
         \address{Universit\'e de Caen, LMNO, 14032 Caen, France}
         \email{paolo.bellingeri@unicaen.fr}
\author[J.B. Meilhan]{Jean-Baptiste Meilhan}
         \address{Universit\'e Grenoble Alpes, IF, 38000 Grenoble, France}
         \email{jean-baptiste.meilhan@ujf-grenoble.fr}
\author[E. Wagner]{Emmanuel Wagner}
         \address{IMB UMR5584, CNRS, Univ. Bourgogne Franche-Comt\'e, F-21000 Dijon, France}
         \email{emmanuel.wagner@u-bourgogne.fr}
\date{\today}
\begin{abstract}
Ribbon 2-knotted objects are locally flat embeddings of surfaces in 4-space which bound immersed 3-manifolds with only ribbon singularities. 
They appear as topological realizations of welded knotted objects, which is a natural quotient of virtual knot theory.
In this paper we consider ribbon tubes and ribbon torus-links, which are natural analogues of string links and links, respectively.
We show how ribbon tubes naturally act on the reduced free group, and how this action classifies ribbon tubes up to link-homotopy, 
that is when allowing each component to cross itself.
At the combinatorial level, this provides a classification of welded string links up to self-virtualization.
This generalizes a result of Habegger and Lin on usual string links, and the above-mentioned action on the reduced free group 
can be refined to a general ``virtual extension'' of Milnor invariants.
As an application, we obtain a classification of ribbon torus-links up to link-homotopy.
\end{abstract}

\maketitle

\begin{center}
  \normalsize \em {Dedicated to El\'eonore, Lise, Helena and Silo\'e.}
\end{center}

\saut

\section*{Introduction}
\label{sec:intro}

The theory of $2$-knots, i.e. locally flat embeddings of the $2$-sphere in $4$-space, takes its origins in the mid-twenties from the work of Artin \cite{artin2}. 
However, the systematic study of these objects only really began in the early sixties, notably through the work of Kervaire, Fox and Milnor \cite{FM,KM,kervaire}, 
but also in a series of papers from Kansai area, Japan, partially referenced below.\footnote{See Suzuki's comprehensive survey \cite{suzuki} 
for a much more complete bibliography on the subject.  } 
From this early stage, the class of ribbon $2$-knots was given a particular attention.
Roughly speaking, a $2$-knot is ribbon if it bounds a locally flat immersed $3$-ball whose singular set is a finite number of ribbon disks.
Introduced by T. Yajima \cite{yajima} under the name of simply knotted $2$-spheres, they were extensively studied by T. Yanagawa in \cite{Yana,yanagawa2,yanagawa3}.
Ribbon $2$-knots admit natural generalizations to ribbon $2$-knotted objects, such as links and tangles.
One particularly nice feature of these objects is that they admit a diagrammatic representation
which allows an explicit presentation of the associated knot groups, via a Wirtinger-type algorithm.

This diagrammatic representation also allows to use ribbon $2$-knotted objects as a (partial) topological realization of welded knot theory.
Welded knots are a natural quotient of virtual knots, by the so-called Over Commute relation, which is one of the two forbidden moves in virtual knot theory.
What makes this Over Commute relation natural is that the virtual knot group, and hence any virtual knot invariant derived from it, factors through it.
These welded knotted objects first appeared in a work of Fenn-Rimanyi-Rourke in the more algebraic context of braids \cite{FRR}.
Although virtual knot theory can be realized topologically as the theory of knots in thickened surfaces modulo handle stabilization \cite{CKS,Kuperberg},
this is no longer true for the welded quotient.
However T. Yajima \cite{yajima}, showed that inflating classical diagrams defines a map, called the Tube map, which sends knots onto ribbon torus-knots,
an analogue of ribbon $2$-knots involving embedded tori. This map has been generalized by S. Satoh \cite{Citare} to a surjective map from the welded diagrammatics. 
But this map fails to be one-to-one, and its kernel is not fully understood yet. 
Ribbon knotted objects in dimension $4$ are therefore closely related to both 3-dimensional topology and welded theory.

\saut

In the present paper, we consider two classes of ribbon knotted objects with several components: {\it ribbon tubes} and {\it ribbon torus-links}.
They are the natural analogues in the ribbon context of string links and links, respectively. 
As such, the latter ones can be obtained from the former by a natural braid-type closure operation.

We give in Theorem \ref{prop:classif} the classification of ribbon torus-links up to link-homotopy, that is, up to homotopies in which different components 
remain distinct. 
This theorem should be compared, on one hand, with the link-homotopy classification of links in $3$-space by Habegger and Lin \cite{HL}, 
and on the other hand, with the result of Bartels and Teichner proving the triviality of $2$-links up to link-homotopy \cite{BT}. 
The dichotomy between these two results is striking, and Theorem \ref{prop:classif} is closer in spirit to the one of Habegger and Lin; 
in particular, ribbon torus-links are far from being always link-homotopically trivial. 
Some examples are given in Section \ref{sec:desexemplesdeoufdanstaface}. 

The problem of link-homotopy classification in higher dimension, intiated in \cite{FeRo,RoMa}, has developed in various direction,  
for instance with the construction of homotopy invariants (see e.g. \cite{KoLink,Li,Kirk,RoMa}), 
or towards the relationship with link concordance (e.g. \cite{Coch}). 
Vanishing results of homotopy invariants for embeddings \cite{Coch,RoMa} have led, on one hand, to consider immersions, and on the other hand, to the result of 
Bartels and Teichner for spheres \cite{BT}. 
This also naturally connects to the general study of embedded surfaces in $4$-space, which is a well-developed subject, see e.g. \cite{Casa}. 
Our result suggests that embedded tori in four space form a very particular but interesting case of study, 
which appears as another natural generalization of $1$-dimensional links (as embedded $2$-spheres do).

The proof of the above-mentioned Theorem \ref{prop:classif} follows closely the work of Habegger and Lin \cite{HL2}. 
As such, it involves naturally ribbon tubes and their classification up to link-homotopy, that we now state.

\begin{theo*}(Theorem \ref{th:Iso})
\label{TheoIntro}
There is a group isomorphism between ribbon tubes up to link-homotopy and the group of basis-conjugating automorphisms of the reduced free group.
\end{theo*}

This theorem calls for several comments, which we develop in the next three paragraphs.

Firstly, this theorem says that ribbon tubes up to link-homotopy form a group, in fact the quotient of the welded pure braid group by self-virtualization. 
Actually, our main results on ribbon tubes are obtained, via the Tube map, as consequences of similar statements for welded string links; 
see Theorems \ref{th:wSLh=wPh} and \ref{th:wSLh=AutC}. 
In this context, the welded diagrammatics gives a faithful description of ribbon tubes, and our results can be thus seen as applications of virtual knot theory 
to the concrete study of topological objects. 

Secondly, the classifying invariant underlying the theorem is a $4$-dimensional version of Milnor numbers. 
Through the Tube map, we obtain a general and topologically grounded extension of Milnor invariants to virtual/welded objects, 
which should be compared with several previous virtual extensions of Milnor invariants proposed in \cite{DK,kotorii,KP}. 

Thirdly, since usual string links sit naturally in the welded string link monoid, 
the above theorem is a generalization of the link-homotopy classification of string links of Habegger and Lin \cite{HL}.
In particular our $4$-dimensional Milnor invariants have the natural feature that they coincide with the classical ones for usual string links 
through the Tube map. However we emphasize here that our proof of Theorem \ref{th:Iso} is completely independent from the one of \cite{HL}.

\saut

Our link-homotopy classification results for ribbon torus-links and ribbon tubes lead to two natural questions. 
The first one is that of the classification in higher dimension (for instance, for $3$-dimensional ribbon tori in $5$-space); 
the algebraic counterpart of Theorem \ref{th:Iso} above suggests that the statement would remain true for codimension two tori in higher dimension. 
The second question addresses the general case of tori in four space, by removing the ribbon assumption; 
this is a natural question which seems to us worth studying. 

\saut

The paper is organized as follows.
We begin by setting some notation in Section \ref{sec:general-setting}.
Section \ref{sec:topology} is devoted to the topological aspects of this paper. We introduce ribbon tubes, broken surface diagrams and link-homotopy in our context, and provide various results on these notions. 
Section \ref{sec:welded-diagrams} focuses on the diagrammatic aspects of the paper; it addresses welded string links and pure braids and their connections with ribbon tubes and configuration spaces.
The main tool for proving most of the results of the present paper is the theory of Gauss diagrams, which is reviewed in details in Section \ref{sec:gauss-diagrams}. 
The main proofs of the paper are given there.
Finally, in Section \ref{sec:milnor}, we define Milnor invariants for ribbon tubes, and show how they provide a natural and general extension of Milnor invariant to virtual knot theory. 

\begin{acknowledgments}
This work was initiated by an inspiring series of lectures given by Dror Bar-Natan at a workshop organized in Caen in June 2012. 
The authors are grateful to him for introducing his work on usual/virtual/welded knotted objects, and for countless fruitful discussions.  
We also warmly thank Akira Yasuhara, Ester Dalvit and Arnaud Mortier for stimulating comments and conversations during the preparation of this paper. 
Thanks are also due to the referee for his/her comments, which helped improving the paper. 
This work is supported by the French ANR research project ``VasKho'' ANR-11-JS01-00201. 
\end{acknowledgments}


\section{General settings and notation}
\label{sec:general-setting}
Unless otherwise specified, we set $n$ to be a non negative integer,
once for all.

\saut

We begin with some topological notation and setting.

Let $I$ be the unit closed interval. 
Let $\UnN$ be the set of integer between $1$ and $n$.
We fix $n$ distinct points $\{p_i\}_{i\in\UnN}$ in $I$. 
For every $i\in\UnN$, we choose a disk $D_i$ in the interior of the
2--ball $B^2=I\times I$ which contains the point $p_i$ in its interior, seen in
$\big\{\frac{1}{2}\big\}\times I$.
We furthermore require that the disks $D_i$, for $i\in\UnN$, are pairwise disjoint.
We denote by $C_i:=\p D_i$ the oriented boundary of $D_i$. 
We consider the 3--ball $B^3:=B^2\times I$ and the 4--ball
$B^4=B^3\times I$. For $m$ a positive integer and for every
submanifold $X\subset B^m\cong B^{m-1}\times I$, we set the notation
\begin{itemize}
\item $\p_0X=X\cap \big(B^{m-1}\times\{0\}\big)$, the ``upper boundary'' of $X$;
\item $\p_1X=X\cap
  \big(B^{m-1}\times\{1\}\big)$, the ``lower boundary'' of $X$;
\item $\p_*X=\p
  X\setminus(\p_0X\sqcup\p_1X)$, the ``lateral boundary'' of $X$; 
\item $\overstar{X}=X\setminus\big(\p_*X\cup\p(\p_0X)\cup\p(\p_1X)\big)$.
\end{itemize}
\noi By a tubular neighborhood of $X$, we will mean an open set $N$
such that $N\cap\mathring{B}^m$
is a tubular neighborhood of $\mathring{X}$ in $\mathring{B}^m$ and
$\p_\e N$ is a tubular neighborhood of $\p_\e X$ in $\p_\e B^m$ for both
$\e=0$ and 1.

In the following, an immersion $Y\subset X$ shall be called \emph{locally
  flat} if and only if
it is locally homeomorphic to a linear subspace $\R^k$ in $\R^m$ for
some positive integers $k\leq m$, except on $\p X$ and/or $\p Y$,
where one of the
$\R$--summand should be replaced by $\R_+$. An intersection $Y_1\cap Y_2\subset X$
shall be called \emph{flatly transverse} if and only if it is locally homeomorphic
to the intersection of two linear subspaces $\R^{k_1}$ and $\R^{k_2}$
in $\R^m$ for
some positive integers $k_1,k_2\leq m$, except on $\p X$, $\p Y_1$
and/or $\p Y_2$, where one of the
$\R$--summand should be replaced by $\R_+$.

\saut

Throughout this paper, and for various types of objects, diagrammatical or topological, we will consider local moves. 
A local move is a transformation that changes the object only inside a ball of the appropriate dimension. 
By convention, we will represent only the ball where the move occurs, and the reader should keep in mind that there is a
non represented part, which is identical for each side of the move.

\saut

We also define some algebraic notation which will be useful throughout
the paper.
 Let $G$ be a group and $a,b\in G$ two of its elements.
  We denote by
  \begin{itemize}
  \item $a^b:=b^{-1}ab$ the conjugate of $a$ by $b$;
  \item $[a;b]:=a^{-1}b^{-1}ab$ the commutator of $a$ and
    $b$;
  \item $\Gamma_k G$, for $k\in\mathbb{N}^*$, the $k^\textrm{th}$ term of the
    lower central serie of $G$ inductively defined by
    $\Gamma_{k+1} G:=\big[G;\Gamma_k G\big]$ and $\Gamma_{1} G=G$;
  \end{itemize}
    and, if $G$ is normally generated by elements
    $g_1,\ldots,g_p$, we further denote by
  \begin{itemize}
  \item
    $RG:=\fract{G}/{\big\{[g_i;g_i^g]\ \big|\ i\in\llbracket1,p\rrbracket,g\in G\big\}}$ the \emph{reduced version of $G$}, which is the smallest quotient where each generator
    commutes with all its conjugates;
  \item $\AutC(G):=\big\{f\in\Aut(G)\ \big|\ \forall
    i\in\llbracket1,p\rrbracket,\exists g\in G,f(g_i)=g_i^g\big\}$, the group of \emph{conjugating automorphisms} of $G$.
  \end{itemize}
Moreover, the free group on $n$ generators is denoted by $\Fn$.
Unless otherwise specified, generators of $\Fn$ will be denoted by
$x_1,\ldots,x_n$.
By abuse of notation, same notation will be kept for their images
under quotients of $\Fn$.
More generally, names of maps and of elements will be kept unchanged when
considering quotients.


\section{Ribbon tubes}
\label{sec:topology}
This section is devoted to the topological counterpart of the paper.
We first define the considered objects, namely ribbon tubes, and then classify them up to link-homotopy, in terms of automorphisms of the reduced free group.

\subsection{Ribbon tubes and their homology}
\label{sec:ribbon-2stringlinks}

In this section, we identify $B^2\hookrightarrow B^3$ with $B^2\times\big\{\frac{1}{2}\big\}$, so that the disks
$\{D_i\}_{i\in\UnN}$ are canonically seen in the interior of $B^3$.

\subsubsection{Definitions}
\begin{defi}\label{def:RibbonTube}
A \emph{ribbon tube} is a locally flat
embedding $T=\psqcup_{i\in\UnN} A_i$ of $n$ disjoint copies of the oriented annulus $S^1\times I$ in
$\loverstar{B^4}{B}$ such that
\begin{itemize}
\item $\p A_i=C_i\times\{0,1\}$ for all $i\in\UnN$ and the
  orientation induced by $A_i$ on $\p A_i$ coincides with that of $C_i$;
\item there exist $n$ locally flat immersed
  $3$--balls $\pcup_{i\in\UnN}B_i$ such that
  \begin{itemize}
  \item $\p_* B_i=\mathring{A}_i$ for all $i\in\UnN$;
  \item $\p_{\e}B_i=D_i\times\{\e\}$ for all $i\in\UnN$ and $\e\in\{0,1\}$;
  \item the singular set of $\pcup_{i=1}^nB_i$ is a disjoint union of so-called
    \emph{ribbon singularities}, \ie flatly transverse disks whose
    preimages are two disks, one in $\pcup_{i=1}^n\mathring{B}_i$ and
    the other with interior in $\pcup_{i=1}^n\mathring{B}_i$,  and
    with boundary essentially embedded in $\psqcup_{i=1}^n\p_*B_i=\psqcup_{i=1}^n\mathring{A}_i$.
  \end{itemize}
\end{itemize}
\end{defi}

We denote by $\rTn$ the set of ribbon tubes up to isotopy fixing the boundary circles. 
It is naturally endowed with a monoidal structure by the stacking product 
$T\bullet T':=T\pcup_{\p_1T=\p_0T'}T'$, and reparametrization, and with unit element the 
\emph{trivial ribbon tube} $\mathbf{1}_n:=\psqcup_{i\in\UnN}C_i\times I$. 

Note that this notion of ribbon singularity is a $4$--dimensional analogue of the classical notion of ribbon singularity introduced by R. Fox in \cite{Fox}.
Similar ribbon knotted objects were studied, for instance, in \cite{Yaji}, \cite{Yana} and \cite{KS}, and a survey can be found in \cite{suzuki}.

Note  that the orientation of the $I$--factor in $S^1\times I$ induces an order on essential curves embedded in a tube component, since they are simultaneously homotopic to $S^1\times \{t\}$ for some 
$t\in I$; we will refer to this order as the ``co--orientation order''.
There is also a second, independent, surface orientation on each tube, which we will not use.  
The reader is referred to Section 3.4 of \cite{WKO2} for a more detailed discussion. 

\begin{remarque}\label{remarquefermeture}
 There are two natural ways to close a ribbon tube $T\in\rTn$ into a closed (ribbon) knotted surface in $4$--space. 
First, by gluing the disks $\psqcup_{i\in\UnN}D_i\times\{0,1\}$ which bound $\partial_0 T$ and $\partial_1 T$, 
and gluing a $4$--ball along the boundary of $B^4$, one obtains an $n$-component ribbon $2$--link \cite{yajima},  
which we shall call the \emph{disk-closure} of $T$. 
Second, by gluing a copy of the trivial ribbon tube $\mathbf{1}_n$ along $T$, 
identifying the pair $(B^3\times \{0\},\partial_0 T)$ with $(B^3\times \{1\},\partial_1 \mathbf{1}_n)$ 
and $(B^3\times \{1\},\partial_1 T)$ with $(B^3\times \{0\},\partial_0 \mathbf{1}_n)$, 
and taking a standard embedding of the resulting $S^3\times S^1$ in $S^4$, 
one obtains an $n$-component ribbon torus--link \cite{Citare},  
which we shall call the \emph{tube-closure} of $T$. 
This is a higher dimensional analogue of the usual braid closure operation.
\end{remarque}

Let us mention here a particular portion of ribbon tube, called \emph{wen} in the literature, that may appear in general, 
and which we shall encounter in the rest of this paper.  
Consider an oriented euclidian circle in three-space, and an additional dimension given by time. 
While time is running, let the circle make a half-turn; this path in $4$--space is a wen.
One can also think of a wen as an embedding in $4$--space of a Klein bottle cut along a meridional circle.
There are several topological types of wens, but it was shown in  \cite{KS} that there are all isotopic in $4$--space, so that we can speak of a wen unambiguously.
Note that a wen is a surface embedded in $4$--space with two boundary components, which do not have the same orientation as is required in the previous definition. 
Hence a wen is not an element of $\rT_1$, but the square of a wen is. 
It turns out that the square of a wen is isotopic to the identity (see \cite{KS}), a fact which will be used in some proofs of this section. 
For a more detailed treatment of wens, see for instance \cite{KS}, \cite[Sec. 2.5.4]{WKO1} and \cite[Sec. 4.5]{WKO2}. 

\begin{defi}\label{def:MonotoneRibbonSLn}
  An element of $\rTn$ is said to be \emph{monotone} if it has a
  representative which is flatly transverse to the lamination
  $\pcup_{t\in I}B^3\times\{t\}$ of $B^4$.\\
  We denote by $\rPn$ the subset of $\rTn$ whose elements are monotone.
\end{defi}
\begin{prop}\label{prop:Abitbol}
  The set $\rPn$ is a group for the stacking product.
\end{prop}
\begin{proof}
  The product of two monotone elements is obviously monotone, and an
  inverse $T^{-1}$ for a monotone ribbon tube $T$ is given
  by $T^{-1}\cap \big(B^3\times\{t\}\big):=T\cap \big(B^3\times\{1-t\}\big)$ for each $t\in I$. 
\end{proof}

\begin{remarque}\label{rem:TBW}
It is worth noting here that two monotone ribbon tubes which are
equivalent in $\rTn$ are always related by a monotone isotopy, i.e. by
an isotopy moving only through monotone objects.
This is shown in Remark \ref{rem:star}.
As a consequence, $\rPn$ is equal to the group defined as the quotient
of monotone ribbon tubes by monotone isotopies.
\end{remarque} 
\begin{remarque}\label{remarquemoisie}
  The monotony condition enables $I$ to be considered as a time
  parameter and the flat transversality forces $T\cap \big(B^3\times\{t\}\big)$
  to be $n$ disjoint circles for all $t\in I$.
  A monotone ribbon tube can hence be seen as an element of the fundamental group of the
configuration space of $n$ circles  in $3$--space, that we will denote by $PR_n$ according to \cite{BH}.
But, since the orientation of $\p T$ is prescribed by the one of
  $\psqcup_{i\in\UnN}C_i$, these elements are actually in the kernel of the
  map $PR_n\to \Z_2^n$ constructed in the proof of Proposition 2.2 from
  \cite{BH}. The group $\rPn$ is therefore isomorphic to the fundamental group of the
configuration space of $n$ circles  in $3$--space lying in parallel planes, denoted by  $PUR_n$ in \cite{BH}.
  This fact can be reinterpreted as seeing monotone ribbon tubes as motions
  of horizontal rings intermingled with wens. Such wens can be pushed above and since the starting
  and final orientations for a given circle match, there is an even
  number of wens and they cancel pairwise.
\end{remarque}

\subsubsection{Homology groups}
\label{sec:Homology}

Let $T$ be a ribbon tube with tube components $\psqcup_{i\in\UnN}A_i$.

Since $T$ is locally flat in $B^4$, there is a unique way, up to
isotopy, to consider, for all $i\in\UnN$, disjoint tubular
neighborhoods $N(A_i)\cong D^2\times S^1\times I$ for $A_i$, with $A_i=\{0\}\times S^1 \times I\subset N(A_i)$.
We denote by $N(T):=\psqcup_{i\in\UnN}N(A_i)$ a reunion of such
tubular neigborhoods and by $W=B^4\setminus \overstar{N}(T)$ the complement of its
interior in $B^4$.

As a direct application of the Mayer--Vietoris exact sequence, we obtain:
\begin{prop}\label{prop:TubeHomology}
  The homology groups of $W$ are $H_0(W)=\Z$,
$$\textrm{$H_1(W)=\Z^n=\Z\big\langle c_i\ \big|\ i\in\UnN\big\rangle$,
  $H_2(W)=\Z^n=\Z\big\langle \tau_i\ \big|\ i\in\UnN\big\rangle$, $H_3(W)=\Z$} $$ and
  $H_k(W)=0$ for $k\geq4$, 
 where $c_i$ is the
  homology class in $H_1(W)$ of $\partial D^2\times \{s\}\times
  \{t\}\subset \p N(A_i)$ for any $(s,t)\in S^1\times I$, and 
  $\tau_i$ of $T$ is the
  homology class in $H_2(W)$ of $\partial D^2\times S^1\times
  \{t\}\subset \p N(A_i)$ for any $t\in I$.
\end{prop}

\subsection{Broken surface diagrams}
\label{sec:broken-surfaces}
Links in $3$--space can be described using diagrams, which are their generic projection onto a $2$--dimensional plane with extra decoration encoding the $3$--dimensional information. 
Similarly, it turns out that ribbon knotted objects, which are surfaces in $4$--space, can be described using their generic projection onto a $3$--space; this leads to the following notion of broken surface diagram.

\begin{defi}
A \emph{broken surface diagram} is a locally flat immersion $S$ of $n$ oriented annuli $\psqcup_{i\in\UnN}A_i$ in $\loverstar{B^3}{B}$ such that
\begin{itemize}
\item $\p A_i=C_i\times\{0,1\}$ for all $i\in\UnN$ and the
  orientation induced by $A_i$ on $\p A_i$ coincides with that of $C_i$;
\item the set $\Sigma(S)$ of connected components of singular points in $S$ consists of flatly
  transverse circles in $\pcup_{i=1}^n\mathring{A}_i$, called \emph{singular circles}.
\end{itemize}
Moreover, for each element of $\Sigma(S)$, a local ordering is given on the two circle preimages. 
By convention, this ordering is specified on pictures by erasing a small neighborhood in $\pcup_{i=1}^n\mathring{A}_i$ of the lowest preimage
 (see Figure \ref{fig:SingCir}).  Note that this is the same convention which is used for usual knot diagrams. 
\end{defi}

\begin{figure}[!h]
\[
\xymatrix@R=-1.2cm{
&\dessin{3cm}{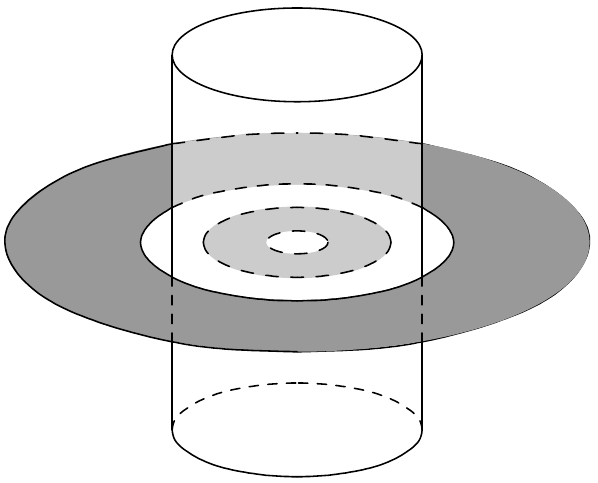}\hspace{.5cm}:\hspace{-.8cm}&
\textrm{\begin{tabular}{c}dark\\preimage\end{tabular}}
<\textrm{\begin{tabular}{c}light\\preimage\end{tabular}}
\\
\dessin{3cm}{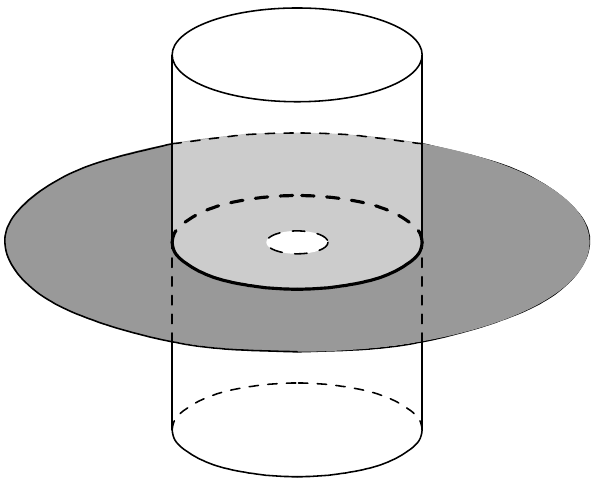}\ar[ur]\ar[dr]&&\\
&\dessin{3cm}{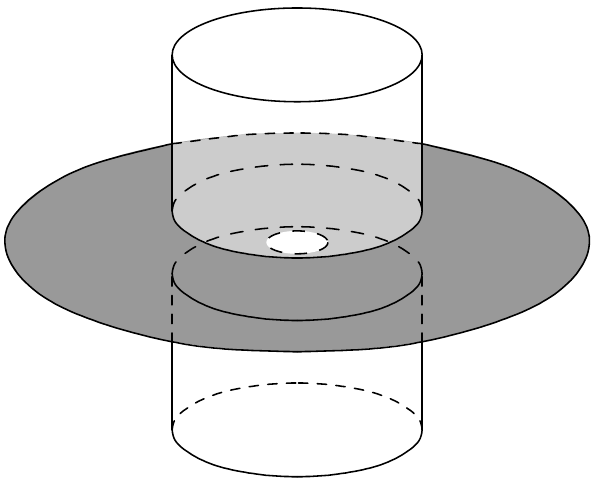}\hspace{.5cm}:\hspace{-.8cm}&
\textrm{\begin{tabular}{c}light\\preimage\end{tabular}}
<\textrm{\begin{tabular}{c}dark\\preimage\end{tabular}}
}
\]
\caption{Local pictures for a singular circle in a broken surface
  diagram}
\label{fig:SingCir}
\end{figure}

\begin{figure}
  \[
  \hspace{-.5cm}
  \begin{array}{ccccc}
    \dessin{4cm}{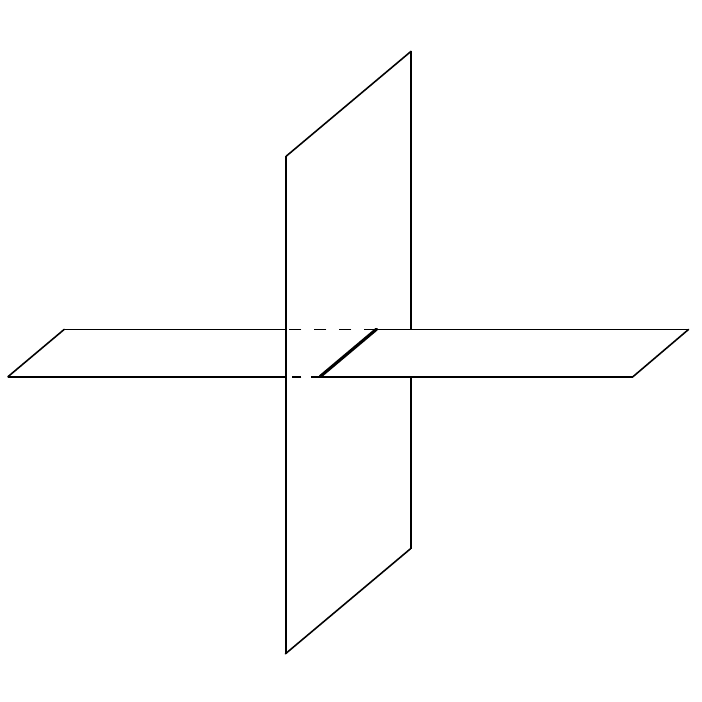}&\leadsto& \dessin{4cm}{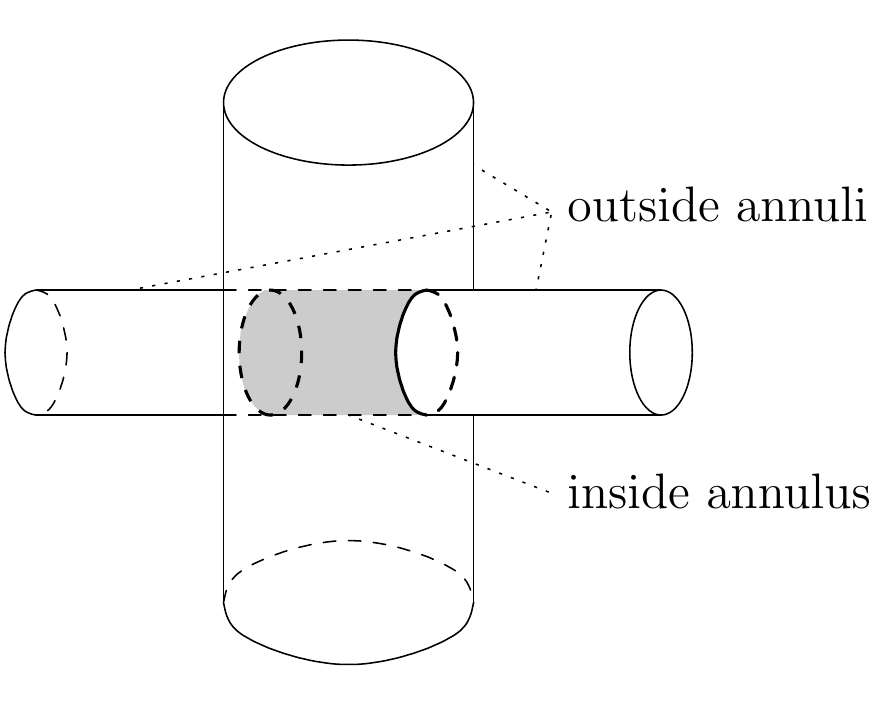} &\leadsto&\dessin{4cm}{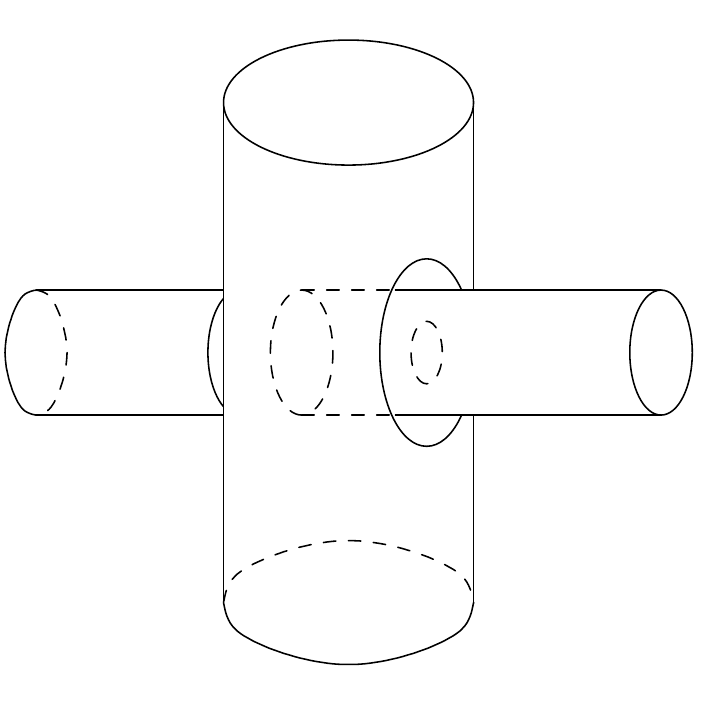}\\
    \textrm{ribbon band} && \textrm{symmetric immersion} \hspace{.7cm}&& \textrm{symmetric broken surface diagram}
  \end{array}
\]
  \caption{From ribbon bands to symmetric broken surface diagrams}
  \label{fig:RibbonBands}
\end{figure}

We shall actually consider only a special case of broken surface
diagrams, namely symmetric ones, that we now introduce, following essentially Yajima \cite{Yaji}.

\begin{defi}
  A \emph{ribbon band} is a locally flat immersion $F:=\psqcup_{i\in\UnN}F_i$ of $n$ disks $I\times I$ in $\loverstar{B^3}{B}$ such that
  \begin{itemize}
  \item $\p_\e F_i=I\times\{\e\}$ is a diameter for $D_i\subset B^2\times\{\e\}$ for each $\e\in\{0,1\}$ and each $i\in\UnN$;
  \item the singular set of $\pcup_{i=1}^n F_i$ is a disjoint union of 2--dimensional ribbon singularities, \ie flatly transverse intervals whose preimages are two intervals, one in $\pcup_{i=1}^n\mathring{F}_i$ and
    the other with interior in $\pcup_{i=1}^n\mathring{F}_i$  and
    with boundary intersecting distinct connected components of $\p_*F$.
  \end{itemize}
A \emph{symmetric immersion}  is a locally flat immersion of $n$ oriented annuli obtained by taking the boundary of 
a thickening of a ribbon band. 
\end{defi}
See the left-hand side and middle of Figure \ref{fig:RibbonBands}, respectively, for an example of ribbon band and symmetric immersion. 
\begin{remarque}\label{rk:inside-outside}
  A symmetric immersion is naturally endowed with a notion of ``inside'', which is simply the 
  thickening of the corresponding ribbon band. 

As pictured in Figure \ref{fig:RibbonBands}, singularities in a symmetric immersion $H$ are singular circles with exactly one essential preimage in $H$. These essential preimages cut $H$ into smaller annuli. Moreover, the singular circles come in pairs, such that the associated essential preimages delimit an annulus which is entirely contained inside $H$. Such annuli are called \emph{inside annuli}, and those which are not are called \emph{outside annuli}. See Figure \ref{fig:RibbonBands}.   
Note that every non essential preimage of a singularity belongs to an outside annulus. 
\end{remarque}

\begin{defi}
  A broken surface diagram $S$ is said to be \emph{symmetric} if it is a symmetric immersion such that,  
  for each inside annulus, the local orderings between the essential and non essential preimages of its two boundary circles are different.  
  See the right-hand side of  Figure \ref{fig:RibbonBands} for an illustration.  
\end{defi}

Let $T$ be a ribbon tube, and suppose that there is a projection $B^4\to B^3$ such that the image of $T$ in $B^3$ has singular locus a union of flatly transverse circles. 
For each double point, the preimages are naturally ordered by their positions on the projection rays, as in Figure \ref{fig:SingCir}, so that the projection naturally yields a broken surface diagram. 
This suggests that broken surface diagrams can be thought of as $3$--dimensional representations of ribbon tubes. 
This is indeed the case, as stated in the next result, which is essentially due to Yanagawa. 
\begin{lemme} \cite{Yana}
Any ribbon tube admits, up to isotopy, a projection to $B^3$ which is a broken surface diagram.
Conversely any broken surface diagram is the projection of a unique 
ribbon tube. 
\end{lemme}
\begin{proof}
Yanagawa proves that  for locally flat embeddings of $2$ spheres in
$\mathbb{R}^4$, there is an equivalence between the property of being
ribbon (property $R(4)$ in \cite{Yana}) and the property of admitting
a projection onto an immersion whose only singular points are
transverse double points (property $R(3)$ in \cite{Yana}). The proof
of the equivalence passes through the equivalence with a third notion, 
which is the fusion of a trivial $2$--link (property $F$ in
\cite{Yana}). Lemma 4.3 in \cite{Yana} proves that $R(4)$ and $F$ are
equivalent. Corollary 3.3 in \cite{Yana} states that $F$ implies $R(3)$ and
Lemma 3.4 states the
converse implication. All the arguments are local and
apply to the case of ribbon tubes as well.
\end{proof}

More specifically, we have the following.
\begin{lemme} \cite{Yaji,KS}\label{lem:symmetric}
Any ribbon tube can be represented by a symmetric broken surface diagram.
\end{lemme}
\begin{proof}
 By the previous lemma, any ribbon tube $T$ can be represented by a
 broken surface diagram $S$. Now we prove that $S$ can be transformed into a symmetric broken
 surface diagram which still represents the same ribbon tube.
 For this, we consider the disk--closure $S$ of $T$ as defined in Remark
 \ref{remarquefermeture}.
 We obtain a ribbon diagram for an $n$--component ribbon $2$-link.
 Now by Theorem 5.2 in \cite{KS}, this diagram is equivalent to a diagram
 obtained by closing a symmetric broken surface diagram.
 This equivalence is generated by local moves, which we may assume to
 avoid a neighborhood of the closing disks \emph{except}
 possibly for a finite number of moves which consist in discarding some wens
 accross closure disks; we do not perform these latter moves, and leave such wens
 near the boundary circles instead.
 We obtain in this way a new broken surface diagram which describes a ribbon tube isotopic to $T$,
 and it only remains to get rid of all residual wens.
 Following again \cite{KS}, all wens can be pushed down
 near the bottom circles $\pcup_{i=1}^n C_i\times \{0\}$.
 The orientations of the two boundary
 circles of a wen are opposite, and the orientation on $\p T$ induced
 by $T$ must agree with that of $\psqcup_{i\in\UnN}C_i$,
 so there are an even number of wens near each bottom circle, and
 they cancel pairwise. There is thus no more wen and the resulting
 broken surface diagram for $T$ is symmetric.
\end{proof}

\begin{remarque} 
  Since we are only interested in broken surface diagrams which represent ribbon tubes, the definition given in this paper is
  watered down compared to what is commonly used in the literature. 
  Let us recall that in general, a generic projection of an embedded surface in $4$--space onto a $3$--space has three types of singularities: 
  double points, triple points and branching points (see, for instance, \cite{Casa}, \cite{Yana}).  
  In our context, only the first type does occur.
\end{remarque}

\begin{remarque}
  If two symmetric broken surface diagrams
  differ by one of the ``Reidemeister'' (RI, RII, RIII) or ``virtual'' (V) moves shown in
  Figure \ref{fig:BS_ReidMoves}, then the associated ribbon tubes are
  isotopic. 
\begin{figure}[!h]
\[
\textrm{RI:}\xymatrix{\dessin{3.3cm}{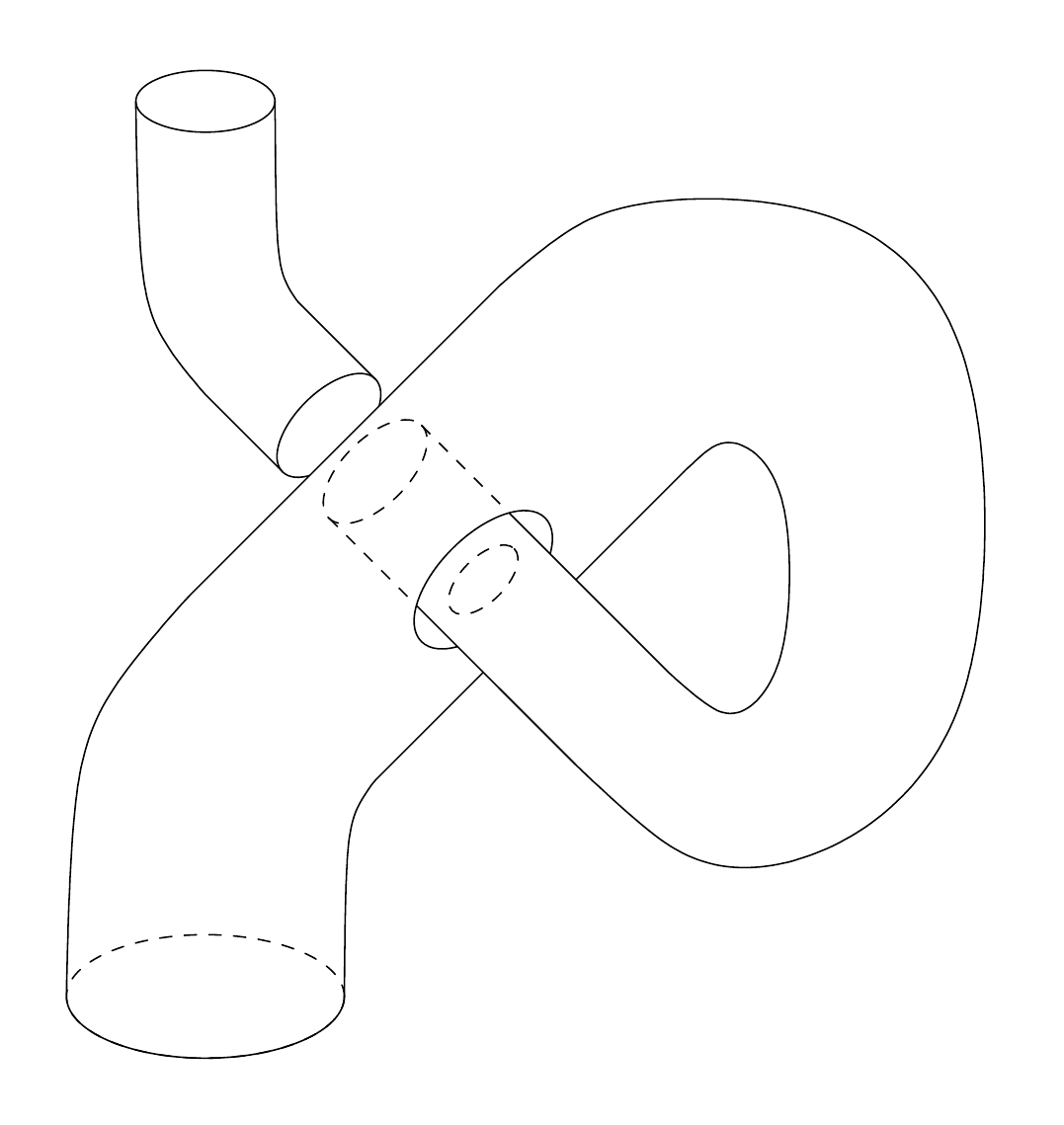}\ar@{^<-_>}[r]&\dessin{3.3cm}{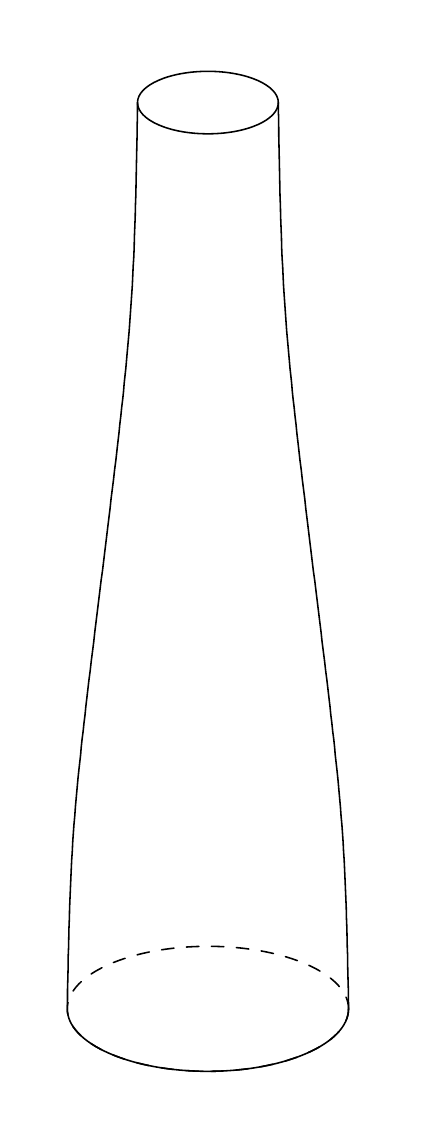}}
\hspace{1cm}
\textrm{RII:}\xymatrix{\dessin{3.3cm}{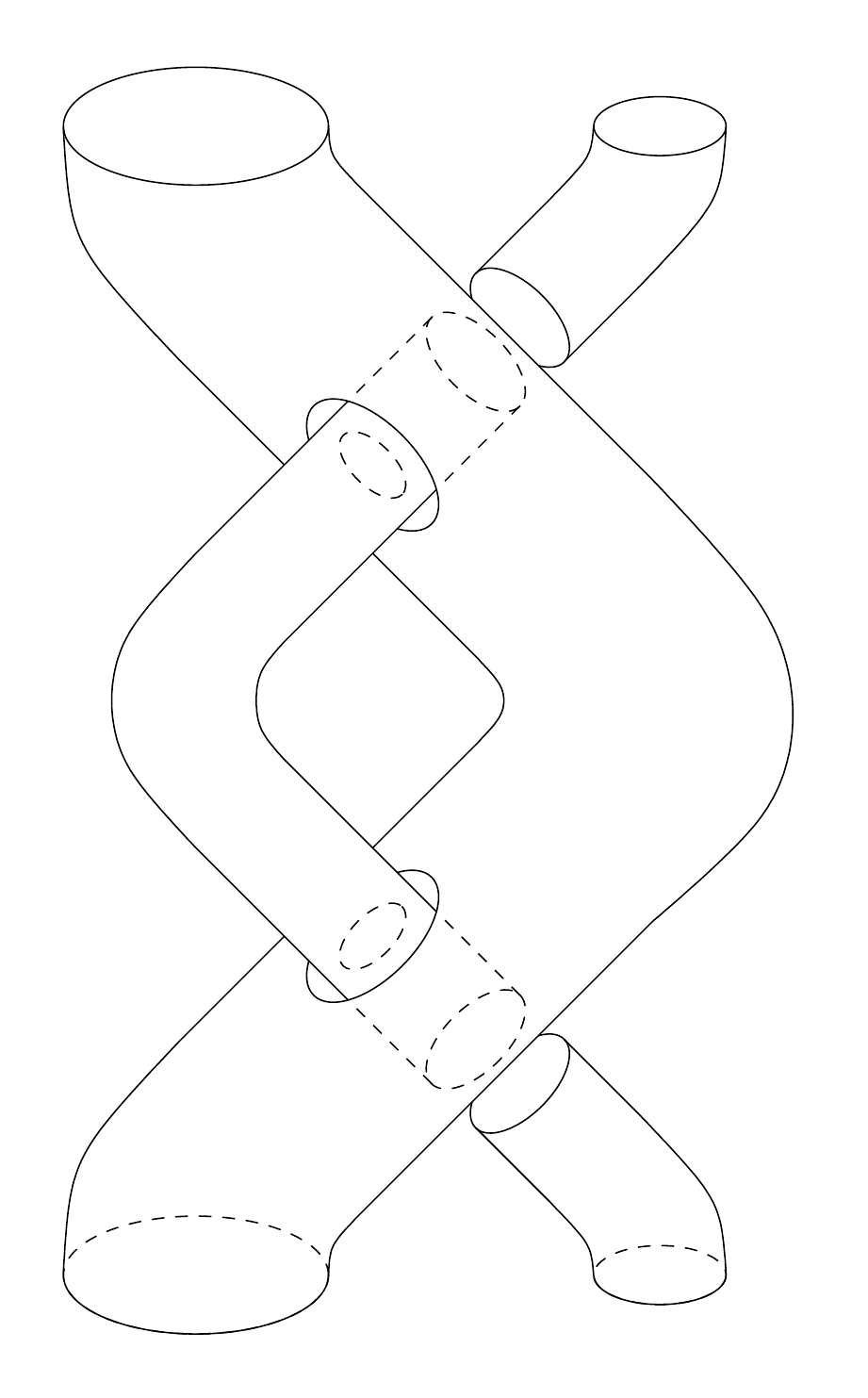}\ar@{^<-_>}[r]&\dessin{3.3cm}{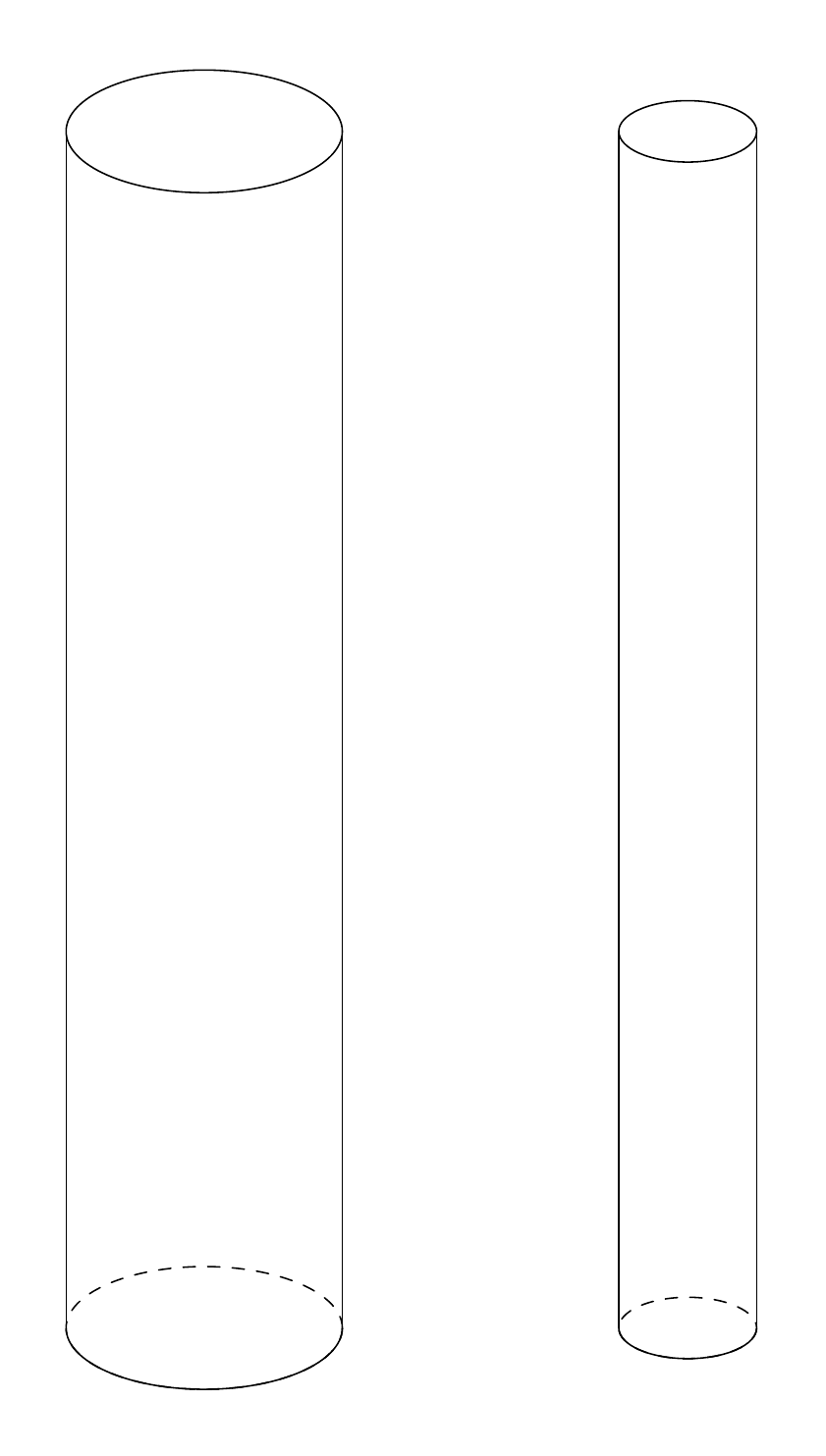}}
\]
\[
\textrm{RIII:}\xymatrix{\dessin{3.3cm}{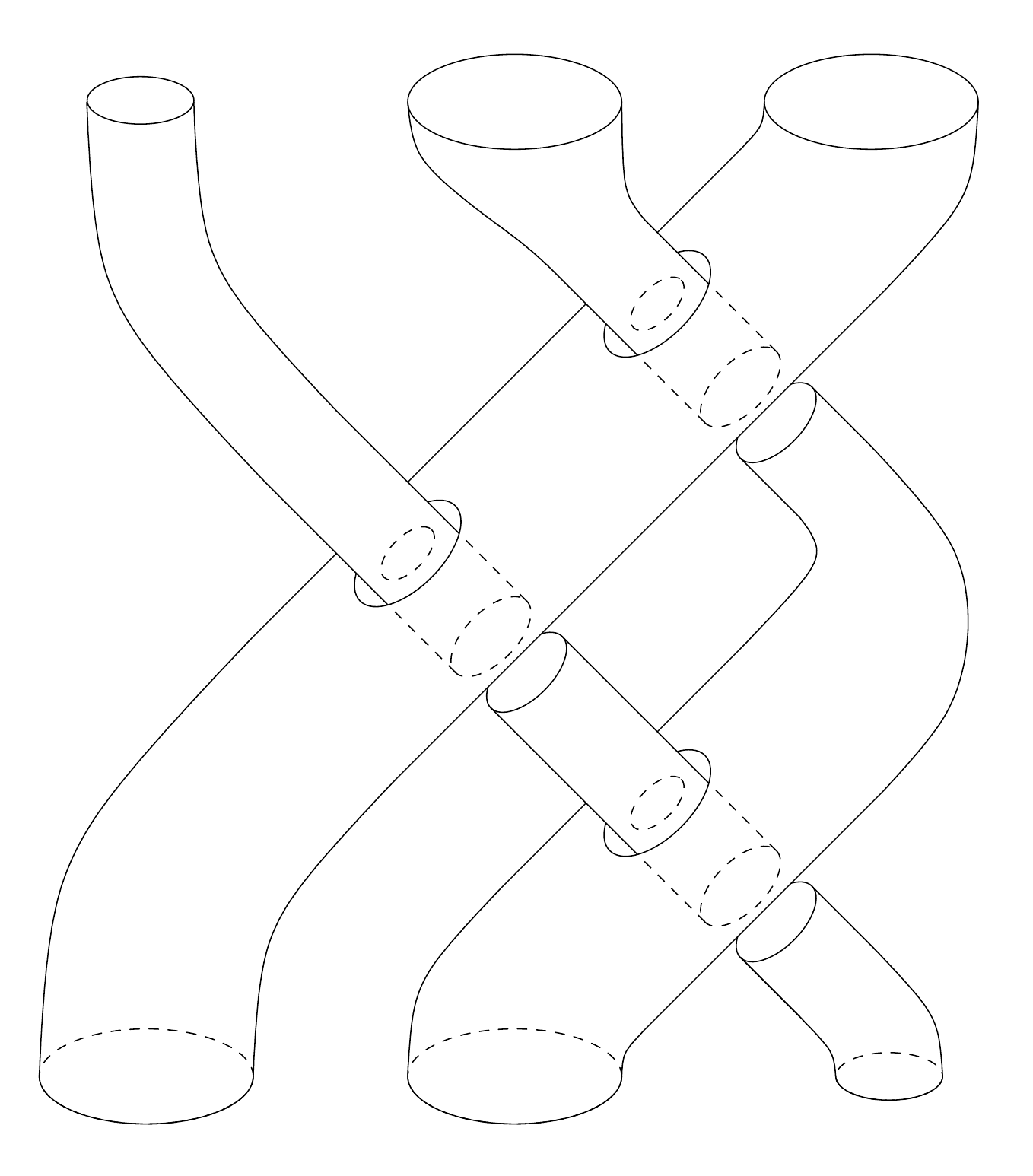}\ar@{^<-_>}[r]&\dessin{3.3cm}{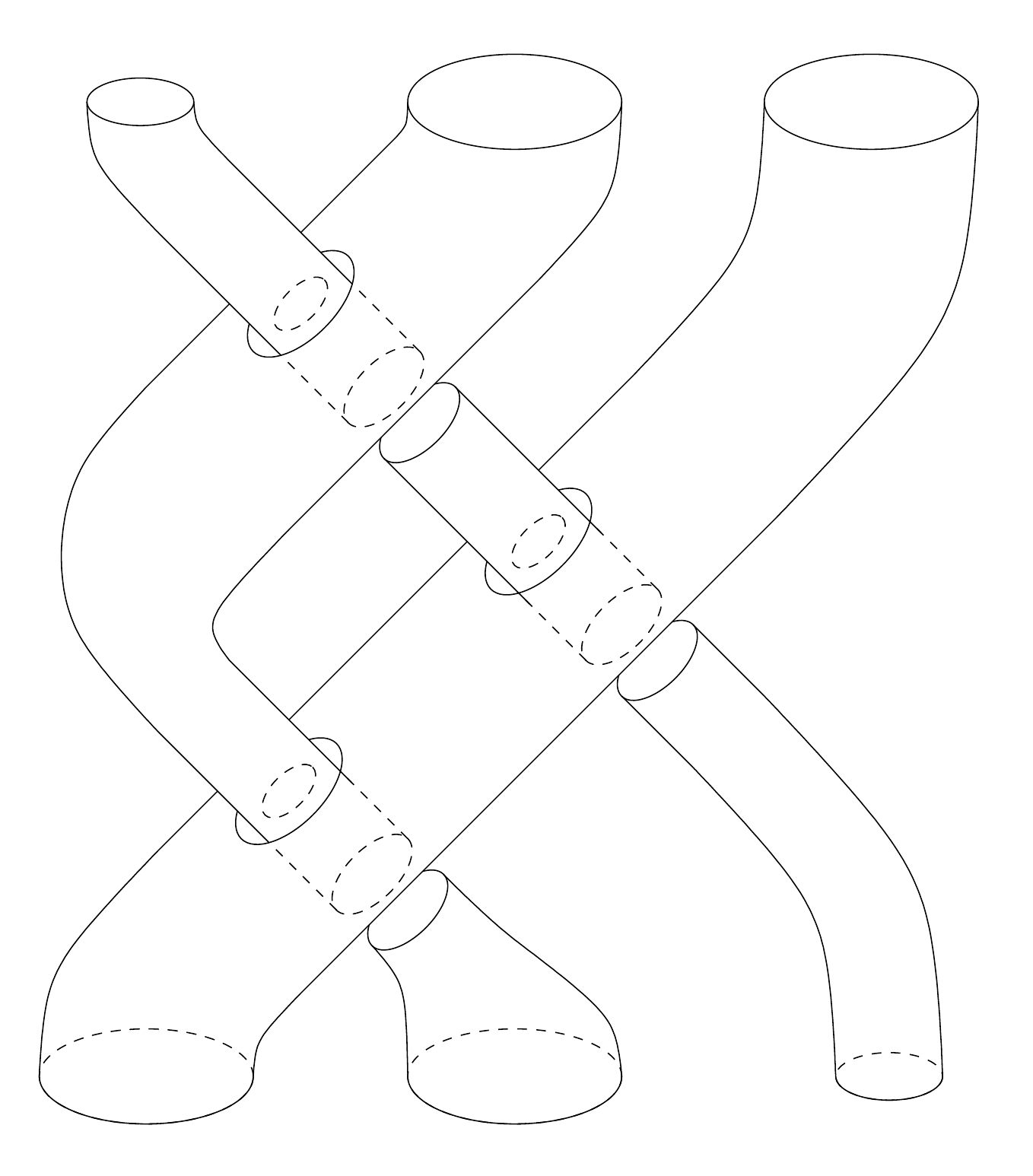}}
\hspace{1cm}
\textrm{V:}\xymatrix{\dessin{3.3cm}{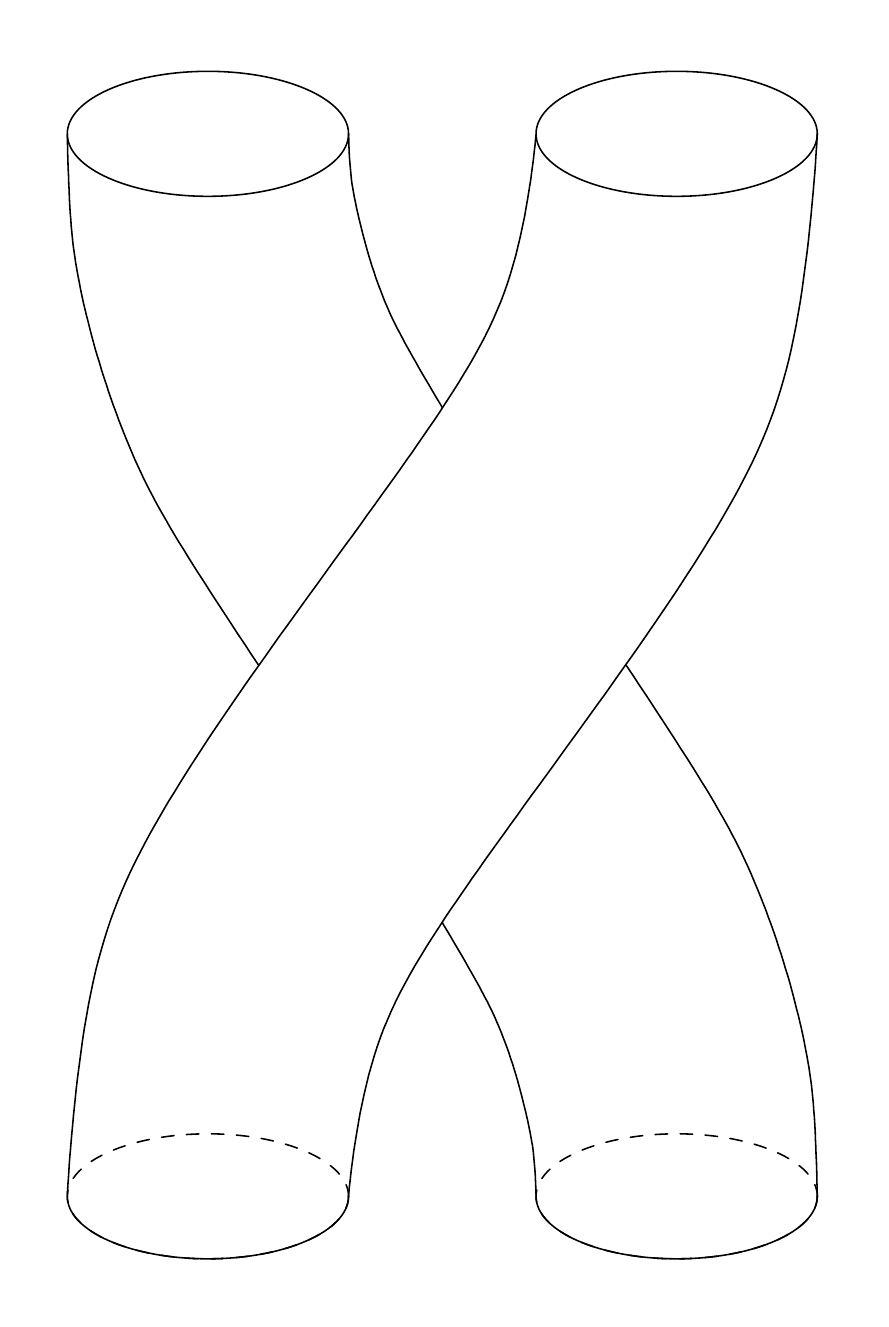}\ar@{^<-_>}[r]&\dessin{3.3cm}{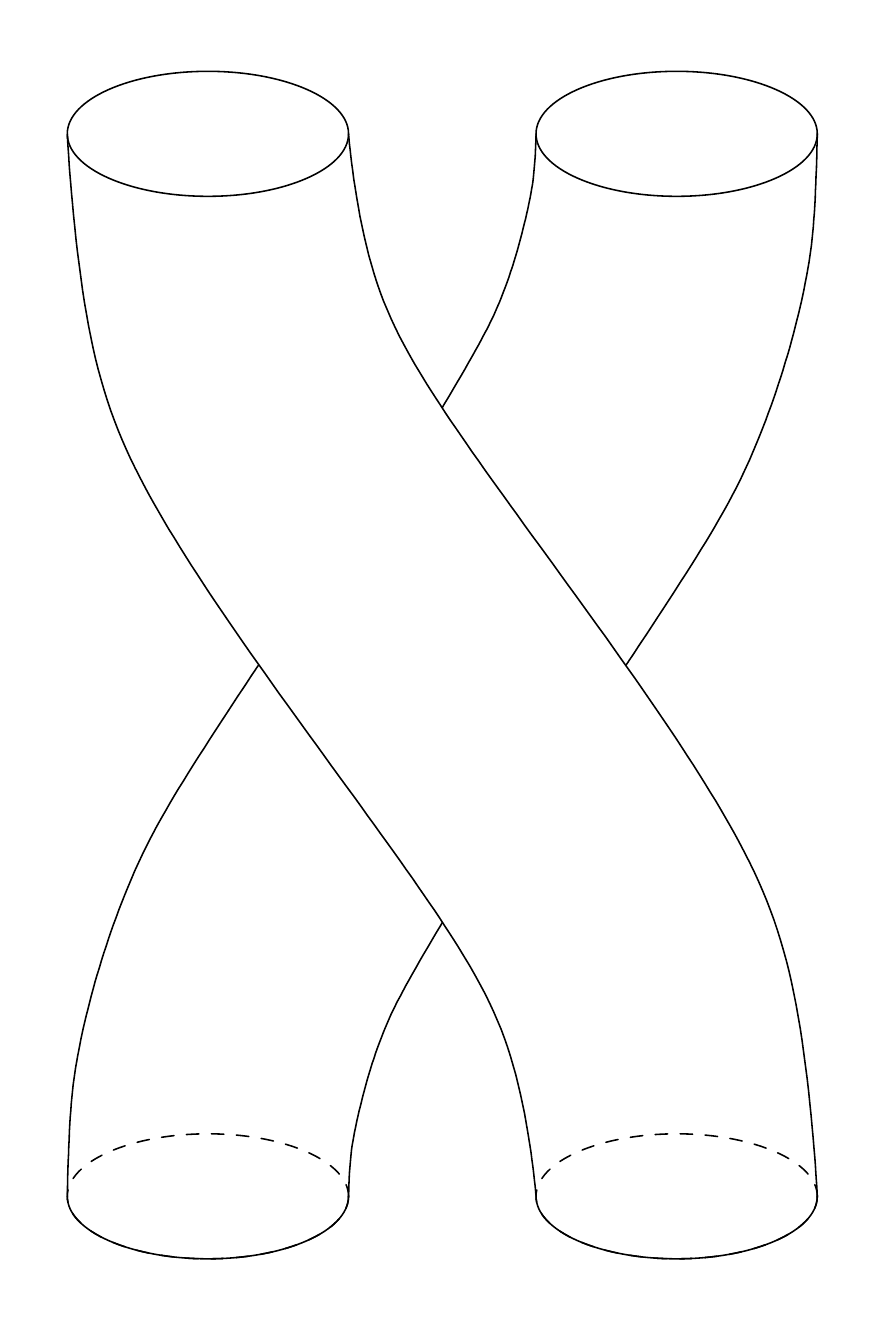}}
\]
  \caption{Reidemeister and virtual moves for broken diagrams}
  \label{fig:BS_ReidMoves}
\end{figure}
  Indeed, both sides of virtual and Reidemeister RII \& RIII moves can be locally modelized as
  occurring in the projection of monotone ribbon tubes. Then, one can
  use the approach mentioned in Remark \ref{remarquemoisie} to prove the statement. 
  Reidemeister move RI is more tricky but can, for example, be checked using Roseman moves defined in \cite{roseman}.
  Another approach can also be found in \cite{Jess}.
  Nevertheless, it is still unknown whether the correspondence between
  ribbon tubes and symmetric broken surface diagrams up to Reidemeister and virtual moves is
  one-to-one; see  e.g.  \cite[Sec. 3.1.1]{WKO1} or \cite[Sec. 3.1]{WKO2}. 
  This discussion is the key to the relation between ribbon tubes and welded string links defined in Section \ref{sec:welded-diagrams}.
  However, it follows from Corollary \ref{cor:TubeInj} that this correspondence is one-to-one up to the link-homotopy relation which is defined in the next section.
  \label{rk:Reidbrok}
\end{remarque}

\subsubsection{Fundamental group}\label{sec:fun-da-mental}

Let $T$ be a ribbon tube with tube components
$\psqcup_{i\in\UnN}A_i$ and consider $N(T):=\psqcup_{i\in\UnN}N(A_i)$ and
$W=B^4\setminus \overstar{N}(T)$ as in Section \ref{sec:Homology}.
We also consider a global parametrization $(x,y,z,t)$ of
$B^4$, which is compatible with $B^4\cong B^3\times I\cong B^2\times
I\times I$ near $\p_0B^4$ and $\p_1B^4$, and such
that the projection along $z$ maps $T$ onto a symmetric broken
surface diagram $S$.
We also fix a base
point $e:=(x_0,y_0,z_0,t_0)$ with $z_0$ greater than the
highest $z$--value taken on $N(T)$.
\begin{nota}
  We set $\pi_1(T):=\pi_1(W)$ with base point $e$.
\end{nota}

\paragraph{\it Meridians and longitudes}

For every point $a:=(x_a,y_a,z_a,t_a)\in T$, we define
$m_a\in\pi_1(T)$, the \emph{meridian around $a$}, 
as $\tau^{-1}_a\gamma_a\tau_a$ where
\begin{itemize}
\item $\tau_a$ is the straight path
from $e$ to $\widetilde{a}:=(x_a,y_a,z_0,t_a)$;
\item $\gamma_a$ is the loop in $W$, unique up to isotopy, based at
$\widetilde{a}$ and which enlaces positively $T$ around $a$.
\end{itemize}

\begin{defi}
  For each $i\in\UnN$ and $\e\in\{0,1\}$, we denote by $m_i^\e$ the meridian
  in $\pi_1(T)$ defined as $m_{a_i^\e}$ for any $a_i^\e\in
  C_i\times\{\e\}$.
  If $\e=0$, we call it a \emph{bottom meridian}, and if $\e=1$, we
  call it a \emph{top meridian}.
\end{defi}
See the left-hand side of Figure \ref{fig:longitudemeridien} for an example of a bottom meridian.

Note that, for any $\e\in\{0,1\}$ and any choice of $a_i^{\varepsilon}$, the fundamental group of $\p_\e W$
based at $(x_0,y_0,z_0,\e)$ can be identified with the free group $\Fn=\big\langle
m^\e_i\ \big|\ i\in\UnN\big\rangle$.

\begin{figure}[h!]
  \[
  \begin{array}{ccc}
    \dessin{5.5cm}{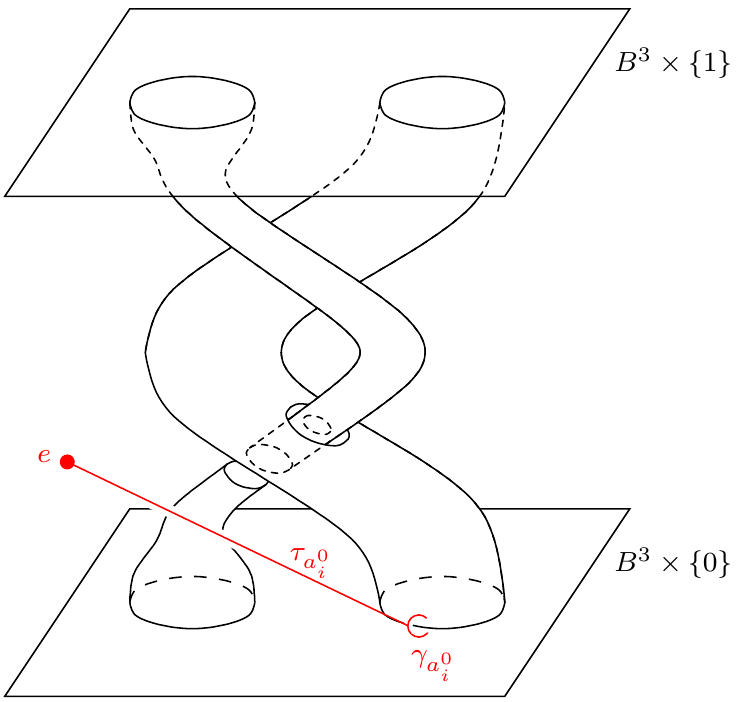}&\hspace{1cm}&\dessin{5.5cm}{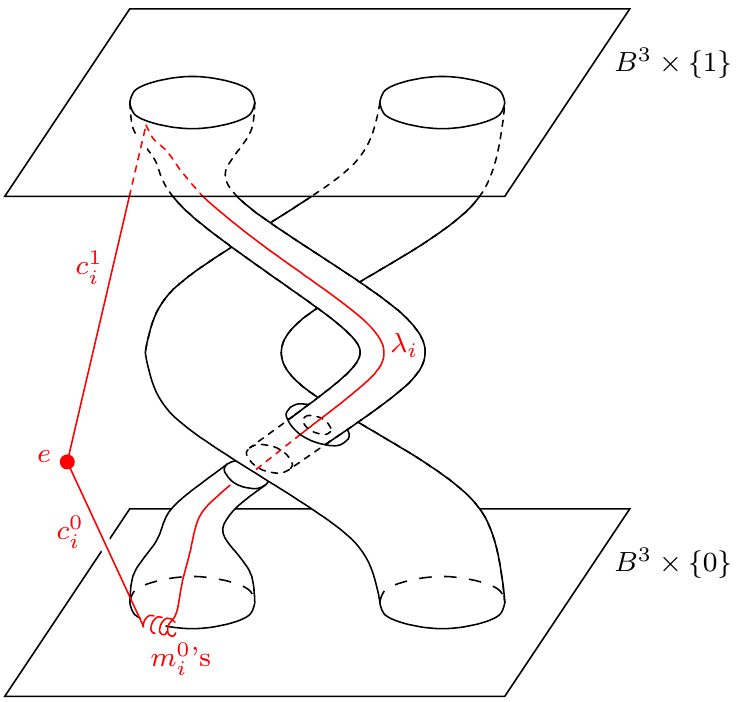}\\[3cm]
   \textrm{meridian $m_i^0$}&&\textrm{longitude $\lambda_i$}
  \end{array}
  \]
  \caption{Examples of meridians and longitude}
 \label{fig:longitudemeridien}
\end{figure}

\saut

Now, we define the notion of longitude for $T$ as follows. 
First, we fix two points $e_i^0\in\p_0N(A_i)$ and $e_i^1\in\p_1N(A_i)$ on each extremity of the boundary of the tubular neighborhood of $A_i$. 
A longitude for $A_i$ is defined as the isotopy class of an arc on $\p N(A_i)$ running from $e^0_i$ to $e^1_i$.  
Since $N(A_i)$ is homeomorphic to $D^2\times S^1\times I$, we note that
\[
\p N(A_i)=\big(S^1\times S^1 \times I\big)\cup \big(D^2 \times S^1
\times\{0\}\big)\cup\big(D^2 \times S^1 \times\{1\}\big),
\]
\noi so that the choice of a longitude for $A_i$ is {\it a priori} specified by two coordinates, one for each of the two $S^1$--factors in $S^1 \times S^1\times I$.
On one hand, the first $S^1$--factor is generated by the meridian $m_i^0$, so that the first coordinate is given by the linking number with the tube component $A_i$.
It can be easily checked, on the other hand, that two choices of longitude for $A_i$ which only differ by their coordinate in the second $S^1$--factor are
actually isotopic in $W$. 
\begin{defi}
 For each $i\in\UnN$, an \emph{$i^\textrm{th}$ longitude of $T$} is the isotopy class of an arc on $\p N(A_i)$, running from $e^0_i$ to
$e^1_i$, and closed into a loop with an arc $c_i^0\cup c_i^1$ defined as follows. 
For $\e\in\{0,1\}$, we denote by $\widetilde{e}_i^\e$ the point above
$e_i^\e$ with $z$--cordinate $z_0$; then $c_i^\e$ is the broken line between $e$, $\widetilde{e}_i^\e$ and $e_i^\e$. 
\end{defi}
This definition is illustrated on the right-hand side of Figure \ref{fig:longitudemeridien}. 
\begin{remarque}\label{rk:preflong}
As explained above, any two choices of an $i^\textrm{th}$ longitude differ by a power of $m_i^0$, 
which is detected by the linking number with the tube component $A_i$.
In particular, there is thus a \emph{preferred $i^\textrm{th}$ longitude}, which is defined as having linking number zero with $A_i$. 
We shall not make use of this fact here, but in Section \ref{sec:milnor} at the end of this paper. 
\end{remarque}

\paragraph{\it Wirtinger presentation}

In the following, we give a presentation for $\pi_1(T)$ in terms of broken surface diagrams.

Let $S$ be a symmetric broken surface diagram representing $T$.  
According to the notation set in Remark \ref{rk:inside-outside}, we denote by
$\Out(S)$ the set of outside annuli of $S$ and by $\In(S)$ the set of inside
annuli.

For each inside annulus $\beta\in\In(S)$, we define
\begin{itemize}
\item $\alpha^0_\beta\in\Out(S)$, the outside annulus which contains
  $\p \beta$;
\item $C^+_\beta,C^-_\beta$, the connected components of
  $\beta\cap\alpha^0_\beta$ which, respectively,  are closer to $\p_1
  S$ and to $\p_0S$, according to the co--orientation order (defined after Definition \ref{def:RibbonTube}); 
\item $\alpha^+_\beta,\alpha^-_\beta\in\Out(S)$, the outside annuli
  which have $C^+_\beta$ and $C^-_\beta$ as
  boundary components, respectively;
\item $\e_\beta=1$ if, according to the local
  ordering, the preimage of $C^-_\beta$ in $\beta$ is higher than the
  preimage in $\alpha^0_\beta$, and $\e_\beta=-1$ otherwise.
\end{itemize}
See Figure \ref{fig:BS_Signs} for an illustration. 
\begin{figure}[!h]
\[
\xymatrix@!0@R=2.2cm@C=7cm{
&\dessin{4cm}{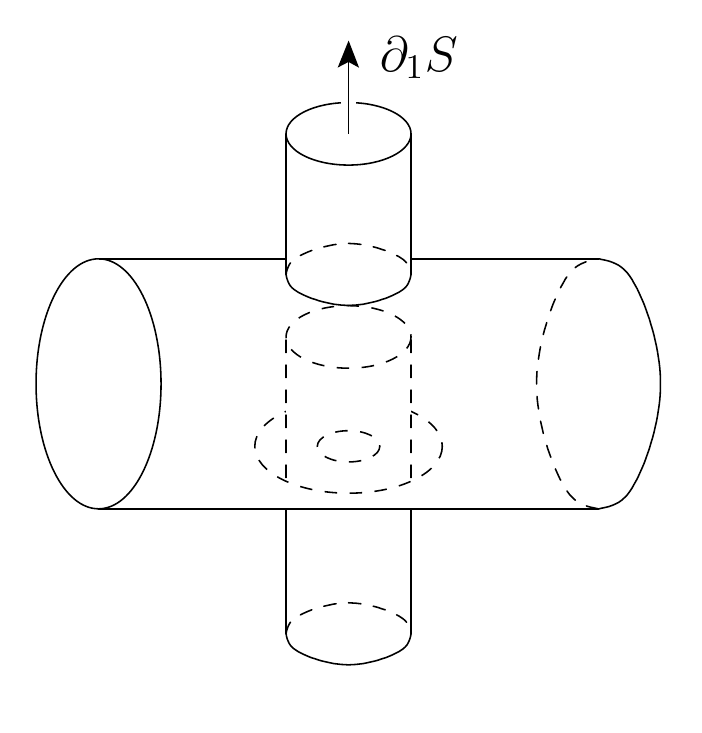} \leadsto \e_\beta=+1\\
\dessin{4cm}{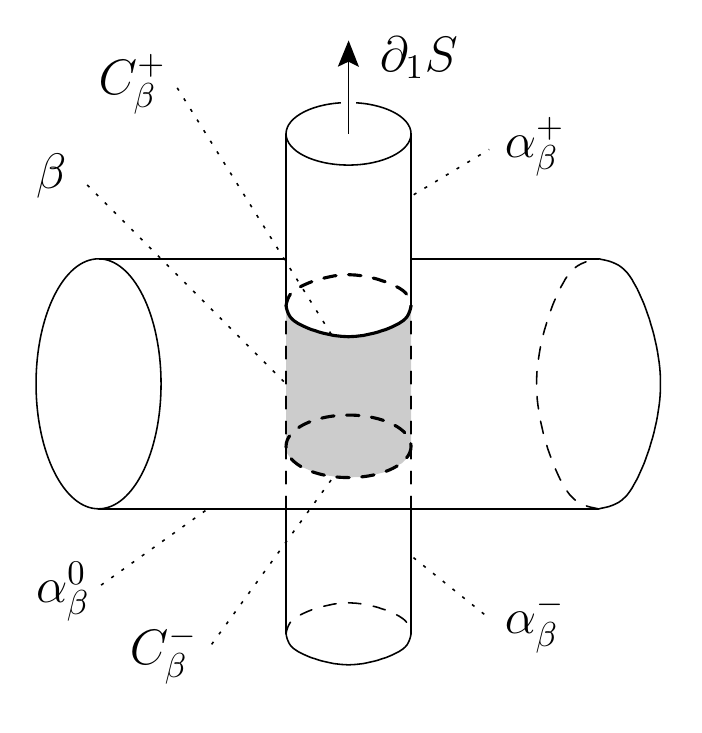} \ar[ru]\ar[rd]\\
&\dessin{4cm}{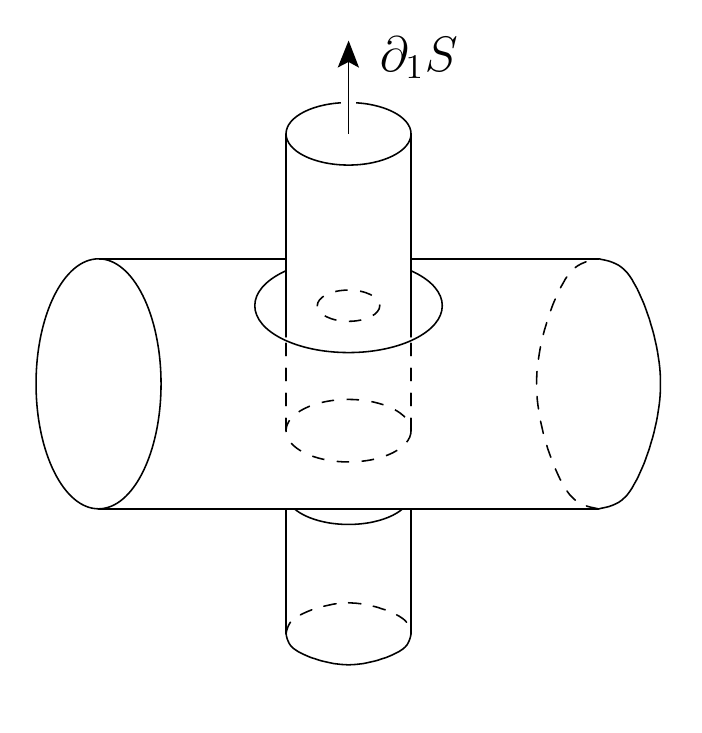} \leadsto \e_\beta=-1
}
\]
  \caption{Signs associated to inside annuli}
  \label{fig:BS_Signs}
\end{figure}

\begin{prop}\cite{yajima,Yana}\label{prop:WirtingerProjection} Let $T$ be a ribbon tube and $S$ any broken surface representing $T$, then 
  \[
\pi_1(T)\cong\big\langle\Out(S)\ \big|\
  \alpha^+_\beta=(\alpha^-_\beta)^{(\alpha^0_\beta)^{\e_\beta}}\textrm{
  for all }\beta\in\In(S)\big\rangle.
\]
In this isomorphism, $\alpha\in\Out(S)$ is sent to $m_a$, where $a$ is
any point on $\alpha$ close to $\p \alpha$.
\end{prop}
\begin{proof}
  By considering, in $B^4$, the union of the projection rays from
  $T\subset B^4$ to $S\subset B^3=\p_0 B^4$, the result follows from
  standard techniques (see {e.g.} the proof of Theorem 3.4 in \cite{BZ}).
\end{proof}
\begin{exemple}
 Let $H$ be the ribbon tube represented in Figure \ref{fig:longitudemeridien}. We have 
   \[
\pi_1(H)\cong\big\langle m_1^0, m_1^1, m_2^0 \ \big|\
  m_1^1=(m_1^0)^{(m_2^0)^{-1}}\big\rangle.
\]
\end{exemple}

\begin{cor}\label{cor:NormGen}
  The group $\pi_1(T)$ is generated by elements $\{m_a\}_{a\in T}$,
  and moreover, if $a\in A_i$ for some $i\in\UnN$, then $m_a$ is a
  conjugate of $m_i^\e$ for both $\e=0$ or 1. 
\end{cor}

\paragraph{\it Reduced fundamental group}
\label{sec:RedFundGr}

In this section, we define and describe a reduced notion of
fundamental group for $T$.
Indeed, Corollary \ref{cor:NormGen} states that $\pi_1(T)$ is normally
generated by meridians
  $m_1^\e,\cdots,m_n^\e$ for either $\e=0$ or 1.
  Moreover, since top meridians are also conjugates of the bottom
  meridians and {\it vice versa}, we can define the following without ambiguity:
\begin{defi}
  The \emph{reduced fundamental group of $T$} is defined as
  $R\pi_1(T)$, the reduced version of $\pi_1(T)$ seen as normally
  generated by either bottom meridians or top meridians.\\
  For convenience , we also denote $R\pi_1(\p_\e W)$ by $R\pi_1(\p_\e
  T)$, for $\e\in\{0,1\}$.
\end{defi}
\noindent (Recall that the reduced version of a group was defined at the end of Section \ref{sec:general-setting}.)

It is a consequence of the description of $H_*(W)$ given in
Proposition \ref{prop:TubeHomology} that, for $\e\in\{0,1\}$, the inclusion $\iota_\e:\p_\e
W\hookrightarrow W$ induces isomorphisms at the $H_1$ and $H_2$ levels.
Stallings theorem, \ie Theorem 5.1 in \cite{stallings}, then implies that
\[
(\iota_\e)_k:\fract{\pi_1(\p_\e W)}/{\Gamma_k \pi_1(\p_\e W)}
\xrightarrow[]{\hspace{.3cm} \simeq\hspace{.3cm} }\fract{\pi_1(T)}/{\Gamma_k \pi_1(T)}
\]
\noi are isomorphisms for every $k\in\N^*$.
But $\pi_1(\p_\e W)$ is the free group $\Fn$ generated by meridians
$m_1^\e,\cdots,m_n^\e$. It follows from Habegger-Lin's Lemma 1.3 in \cite{HL} that for $k\geq n$,
$R\left(\fract{\Fn}/{\Gamma_k\Fn}\right)\cong \RFn$.
As a consequence:

\begin{prop}\label{prop:ReducedIso} 
The inclusions $\iota_0$ and $\iota_1$ induce isomorphisms 
  $$\RFn\cong R\pi_1(\p_0 T)\xrightarrow[\iota^*_0]{\hspace{.3cm}\simeq \hspace{.3cm}}R\pi_1(T)\xleftarrow[\iota^*_1]{\hspace{.3cm}\simeq \hspace{.3cm}}R\pi_1(\p_1 T)\cong\RFn. $$
\end{prop}

Using the isomorphisms of Proposition \ref{prop:ReducedIso}, we define, for
every ribbon tube $T$, a map $\varphi_T:\RFn\to\RFn$ by
$\varphi_T:={\iota^*_0}^{-1}\circ\iota^*_1$. This can be seen as
reading the top meridians as products of the bottom ones.
It is straightforwardly checked that $\varphi_{T\bullet T'}=\varphi_T\circ\varphi_{T'}$ and this action on $\RFn$ is obviously invariant under
isotopies of ribbon tubes.

It follows from Corollary \ref{cor:NormGen} that:
\begin{prop}
  For every ribbon tube $T$, $\varphi_T$ is an element of $\AutC(\RFn)$, the group of conjugating automorphisms. 
  More precisely, the action of $T\in\rTn$ on $\RFn$ is given by conjugation of each $x_i$, for
  $i\in\UnN$, by the image through $\iota^*_0$ of an $i^\textrm{th}$ longitude of $T$.
\end{prop}

Note that, in the reduced free group, this conjugation does not depend
on the choice of an $i^\textrm{th}$ longitude. 

\begin{remarque}\label{rem:star}
  If the ribbon tube $T$ is  monotone, then the inclusions $\iota_\varepsilon$, 
  for $\varepsilon\in\{0,1\}$, actually induce  isomorphisms
  $\pi_1(\partial_\varepsilon W)=\Fn \cong \pi_1(T)$,
  so that $T$ defines an action in $Aut_C(\Fn)$ which is left
  invariant by any (possibly non-monotone) isotopy.
  Theorem 2.6 of \cite{WKO1} shows that this 
  induces an isomorphism between the group of monotone ribbon tubes up to monotone isotopies and $\AutC(\Fn)$, 
  since this group is isomorphic to the group $PUR_n$ defined in \cite{BH} (as noticed in Remark \ref{remarquemoisie}).  
  It follows that if two monotone ribbon tubes are  
  equivalent in $\rTn$, then they induce the same action on $\Fn$ and
  hence are isotopic through a monotone isotopy. Put differently, we have that the group $PUR_n$ injects in $\rTn$. 
\end{remarque}

\subsection{Classification of ribbon tubes up to link-homotopy.}

We now give a classification result for ribbon tubes, up to link-homotopy, in terms of their induced action on the reduced free group. 

\subsubsection{Link-homotopy}
\label{sec:LinkHomotopy}

\begin{defi}\label{def:SingRibbonTube}
A \emph{singular ribbon tube} is a locally flat
immersion $T$ of $n$ annuli $\psqcup_{i\in\UnN}A_i$ in
$\loverstar{B^4}{B}$ such that
\begin{itemize}
\item $\p A_i=C_i\times\{0,1\}$ for all $i\in\UnN$ and the
  orientation induced by $A_i$ on $\p A_i$ coincides with that of $C_i$;
\item the singular set of $T$ is a single flatly
    transverse circle, called \emph{singular loop}, whose
    preimages are two circles embedded in $\pcup_{i=1}^n\mathring{A}_i$, an essential and a non essential one.
\item there exist $n$ locally flat immersed $3$--balls $\pcup_{i\in\UnN}B_i$ such that
  \begin{itemize}
  \item $\p_* B_i=\mathring{A}_i$ and $\p_{\e}B_i=D_i\times\{\e\}$ for all $i\in\UnN$ and $\e\in\{0,1\}$;
  \item the singular set of  $\pcup_{i=1}^nB_i$ is a disjoint union of flatly
    transverse disks, all of them being ribbon singularities but one, whose
    preimages are two disks bounded by the preimages of the singular loop, one in $\pcup_{i=1}^n\p_*B_i$ and
    the other with interior in $\pcup_{i=1}^n\mathring{B}_i$.
  \end{itemize}
\end{itemize}
We say that a singular ribbon tube is \emph{self-singular} if and only if both preimages of the singular loop belong to the same tube component.
\end{defi}

\begin{defi}
  Two ribbon tubes $T_1$ and $T_2$ are said to be \emph{link-homotopic}
  if and only if there is a $1$--parameter family of regular and self-singular ribbon tubes from $T_1$ to
  $T_2$ passing through a finite number of self-singular ribbon
  tubes.\\
  We denote by $\rTHn$ the quotient
  of $\rTn$ by the link-homotopy equivalence, which is
  compatible with the monoidal structure of $\rTn$.
  Furthermore, we denote by $\rPHn$ the image of $\rPn$ in $\rTHn$.
\end{defi}

\begin{prop}\cite{KS}\label{def:CircleChange}
  The link-homotopy equivalence is generated by \emph{self-circle crossing changes}, 
  which are the operations in $B^4$ induced by the local move shown in Figure \ref{fig:BS_LocalMove},  
  which switches the local ordering on the preimages of a given singular circle, where it is required that both
  preimages are on the same tube component. 
\end{prop}
\begin{figure}[!h]
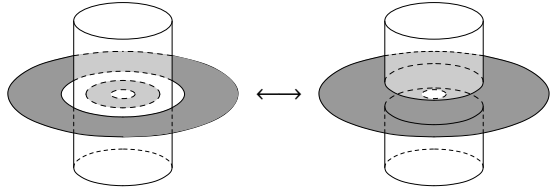

 \[
\dessin{2.5cm}{SingCir_1}\ \longleftrightarrow\ \dessin{2.5cm}{SingCir_2}
\] 
  \caption{A circle crossing change at the level of broken surface diagrams}
  \label{fig:BS_LocalMove}
\end{figure}

Note that a circle crossing change can be seen as a local move among symmetric broken surface diagrams. 
Indeed, although applying a circle crossing change yields a surface diagram which is no longer symmetric 
(see the middle of Figure \ref{fig:SBS_LM}), the resulting inside annulus correspond to a piece of tube passing entirely above or below another piece of tube. 
There is thus no obstruction in $B^4$ for pushing these two pieces of tube apart, so that their projections
don't meet anymore (see the right-hand side of Figure \ref{fig:SBS_LM}).
\begin{figure}[!h]
 \[
\xymatrix@R=-1.2cm{
&\dessin{2.5cm}{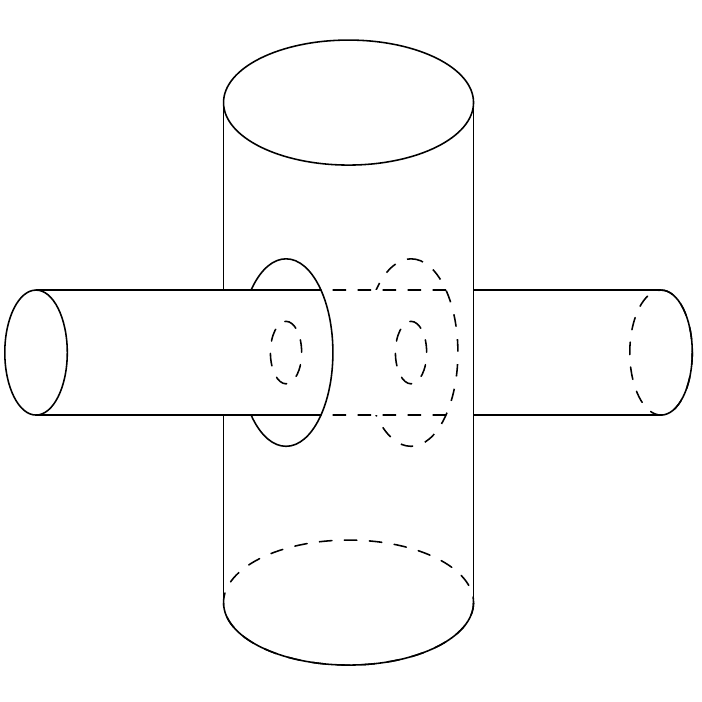}\ar@{}[rd]|{\vcenter{\hbox{\rotatebox{330}{$\simeq$}}}}&\\
\dessin{2.5cm}{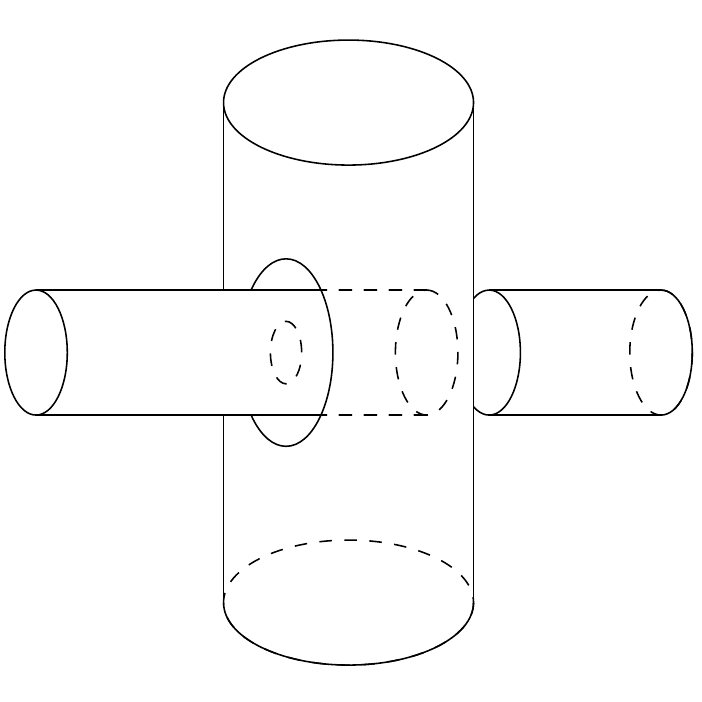} \ar@{<->}[ur] \ar@{<->}[dr] &&\dessin{2.5cm}{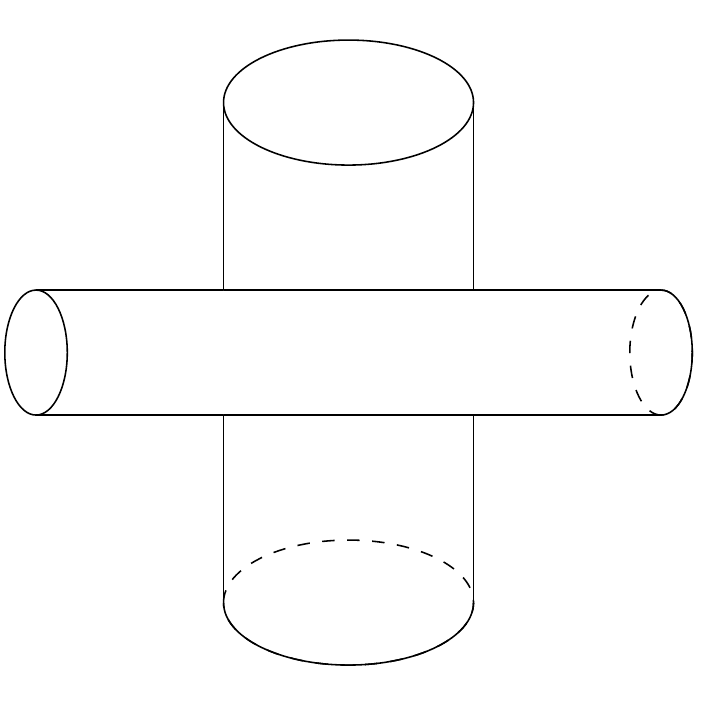}\\
&\dessin{2.5cm}{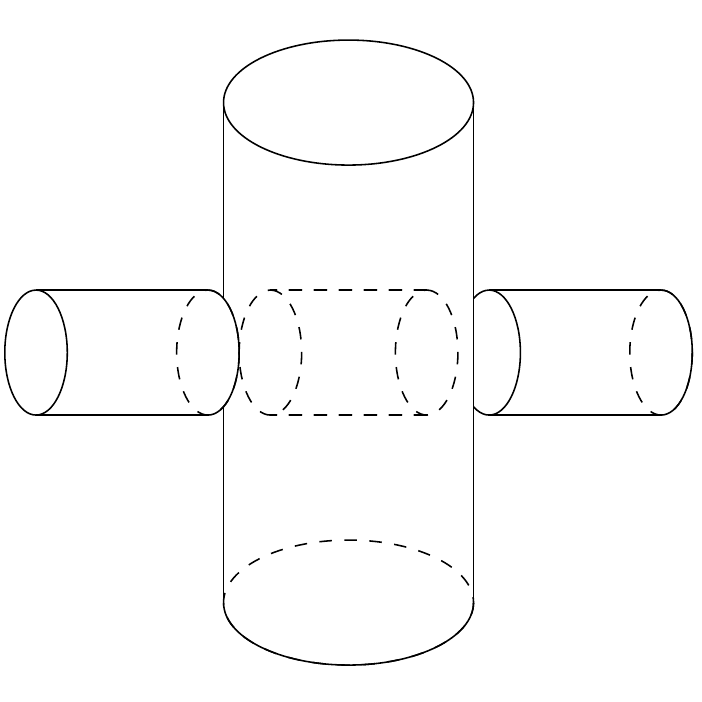}\ar@{}[ru]|{\vcenter{\hbox{\rotatebox{30}{$\simeq$}}}}& 
}
\]
  \caption{A circle crossing change at the symmetric broken surface diagram level}
  \label{fig:SBS_LM}
\end{figure}

We now state one of the main results of this paper. 
\begin{theo}\label{thm:tube_braids}
  Every ribbon tube is link-homotopic to a monotone ribbon tube.
\end{theo}
\begin{proof}
  Using the surjectivity of the Tube map defined in Section \ref{ref:CorrTTleMonde}, this is a direct consequence of Theorem \ref{th:wGD=wGP}.
\end{proof}

\begin{remarque}
Theorem \ref{thm:tube_braids} can be regarded as a higher-dimensional analogue of 
Habegger-Lin's result, stating that any string link is link-homotopic to a pure braid \cite{HL}. 
\end{remarque}

\begin{cor}
  The set $\rTHn$ is a group for the stacking product.
\end{cor}

\subsubsection{Actions on the reduced free group and link-homotopy}
\label{sec:Action}

In Section \ref{sec:RedFundGr}, a conjugating automorphism $\varphi_T$ was associated to any ribbon tube $T$. It turns out that this automorphism $\varphi_T$ is invariant under link-homotopy.
\begin{prop}
  If $T_0$ and $T_1$ are two link-homotopic ribbon tubes, then $\varphi_{T_0}=\varphi_{T_1}$.
\end{prop}
\begin{proof}
  This is a consequence of Lemma \ref{lem:wtf} and the surjectivity of the Tube map defined in Section \ref{ref:CorrTTleMonde}.
However, this result can be given a more topological proof
that we will sketch here.

 It is sufficient to prove the proposition in the case of a
  link-homotopy $H$ passing through a unique singular ribbon tube. We
  denote then by $A$ the annulus which contains the unique singular loop and by $\delta_0$ the disk in $A$ which is bounded by this singular loop.  
By abuse of notation, we will also denote by $A$ the annulus in
  $T_0$, in $T_1$ and in any ribbon tube in-between, which are the
  deformations of $A$ by $H$.
  Moreover, we can assume that the resolution of $\delta_0$ in $T_0$ creates a ribbon singularity, as on the left-hand side of Figure 7, whereas it doesn't in $T_1$, as on the right-hand side of the figure.
 
Following the proof of Lemma 1.6 in \cite{HL},\footnote{ As it was already the case in \cite{HL}, Stallings theorem cannot be
  used here since $H_2(W_H)$ has an extra summand.
  Habegger and Lin retraction argument couldn't either be used in our context.}
we consider the
complement $W_H$ of a tubular neighborhood of $H$ in $B^5\cong B^4\times I$,
where $I$ represents the 1--parameter of the link-homotopy.
The inclusions $\iota_0:\p_0 W_H\hookrightarrow W_H$ and $\iota_1:\p_1
W_H\hookrightarrow W_H$ induce the following commutative diagram:

\begin{equation}
\vcenter{\hbox{\xymatrix{
R\pi_1(\p_0T_0) \ar[r]^(.38)\simeq \ar@{}[d]|{\rotatebox{90}{$=$}}&R\pi_1(\p_0
W)=R\pi_1(T_0)\ar[d]&R\pi_1(\p_1T_0) \ar[l]_(.38)\simeq \ar@{}[d]|{\rotatebox{90}{$=$}}\\
\RFn \ar[r] &
R\pi(H) & \RFn\ar[l]\\
R\pi_1(\p_0T_1) \ar[r]^(.38)\simeq \ar@{}[u]|{\rotatebox{90}{$=$}}&R\pi_1(\p_1 W)=R\pi_1(T_1)\ar[u]&R\pi_1(\p_1T_1) \ar[l]_(.38)\simeq  \ar@{}[u]|{\rotatebox{90}{$=$}}
}}},
\label{diag:TopInv}
\end{equation}
\noindent where $R\pi_1(H):=\fract{\pi_1(W_H)}/{\Omega}$ with $\Omega$ 
the normal subgroup generated by all commutators $[m;m^g]$, $m$
being any bottom meridian of $T_0$ and $g$ any element of
$\pi_1(W_H)$.
Note that, since top and bottom meridians of $T_0$ and $T_1$ are equal in
$\pi_1(W_H)$, and since top meridians are conjugates of the bottom
ones, both vertical maps to $R\pi(H)$ in (\ref{diag:TopInv}) are well defined. 

Let $B_*\subset \mathring{B}^4$ be a $4$--ball so that
\begin{itemize}
\item $\oB_*:=B_*\times I\subset B^5$ contains $\delta_0$;
\item $H$ is trivial outside $\oB_*$;
\item $\p\oB_*\cap H$ is the disjoint union of 4 thickened circles
  $C_1\times I$, $C_2\times I$, $C_3\times I$ and $C_4\times I$ where $C_1$, $C_2$, $C_3$ and $C_4$ are four essential curves in $A$, 
  numbered according to the co--orientation order. 
\end{itemize}
Up to symmetry, we can assume that $\delta_0$ is totally embedded in $\alpha_1$,
the annulus in $A$ which is cobounded by $C_1$ and $C_2$,
rather than in $\alpha_2$, the annulus in $A$ which is cobounded by $C_3$ and $C_4$.

Now, we denote by $Z:=W_H\setminus\oB_*$ the complement of $H$ outside
$B_*$ and by $m_i$, for $i\in\{1,2,3,4\}$, the meridian in
$\pi_1(W_H)$ which enlaces positively $C_i\times I$.

We claim that:
\begin{enumerate}
\item $\pi_1(\p_0
  W_H)=\fract{\pi_1(Z)}/{\big\{m_2=m_1;m_4=m_3^{m_1^{\pm1}}\big\}}$;
\item $\pi_1(\p_1
  W_H)=\fract{\pi_1(Z)}/{\big\{m_2=m_1;m_4=m_3\big\}}$;
\item\label{cond:trois} $\pi_1(W_H)=\fract{\pi_1(Z)}/{\big\{m_2=m_1;m_4=m_3;m_1m_3=m_3m_4\big\}}$.
\end{enumerate}

Using the Seifert--Van Kampen theorem and the Wirtinger presentation for the ribbon tube $\p_0 H\cap B_*$ and $\p_1
H\cap B_*$, the first two assertions are rather clear since filling $Z$ with $\big(\p_0 H\cap B_*\big)\times I$ 
or $\big(\p_1 H\cap B_*\big)\times I$ gives a thickening of, respectively, $\p_0 W_H$ and $\p_1 W_H$.

Now we focus our attention on $W_H\cap\oB_*$.
Using standard techniques, borrowed {\it e.g.} from the proof of Theorem 3.4 in
\cite{BZ}, it is easily seen that $\pi_1(W_H\cap\oB_*)$ is generated
by $m_1$ and $m_2$ and that $m_2=m_1$ and $m_4=m_3$.
We consider a point of $\p \delta_0$ and a $4$--ball $b$ around
it which is transverse to $\delta_0$.
The ball $b$ intersects $\alpha_1$ and $\alpha_2$ along two transverse disks and $b$ can be seen as
the product of these two disks. This product decomposition of $b$ provides a Heegaard
splitting of $\p b=S^3$ into two solid tori whose cores are $\p
b\cap\alpha_1$ and $\p b\cap\alpha_2$, and whose meridians are $m_1$ and $m_3$.
But these cores enlace as an Hopf link and since $\pi_1(\textrm{Hopf
  link})\cong \Z^2$, $m_1$ and $m_3$ commute.
Finally, using the Mayer--Vietoris exact sequence for $W_H\cap\oB_*$ seen
as $\big(\oB_*\setminus N_1\big)\cap \big(\oB_*\setminus
N_2\big)$ where $N_1$ and $N_2$ are tubular neighborhoods of,
respectively, $\alpha_1\times I$ and $\alpha_2\times I$, one can compute that the abelianization
$H_1(W_H\cap\oB_*)$ of $\pi_1(W_H\cap\oB_*)$ is $\Z^2$, so no other
relation can hold in $\pi_1(W_H\cap\oB_*)$.
By use of the Seifert--Van Kampen theorem, this proves statement (\ref{cond:trois}).

Now, since $m_1$, $m_2$, $m_3$ and $m_4$ are all conjugates of the top
meridian of $A$, they commute in the reduced group, and it follows
that $R\pi_1(\p_0
  W_H)\cong R\pi_1(\p_1
  W_H)\cong R\pi_1(W_H)\cong
  R\left(\fract{\pi_1(Z)}/{\big\{m_2=m_1;m_4=m_3\big\}}\right)$.
All the maps in the diagram (\ref{diag:TopInv}) are hence isomorphisms and since
$\varphi_{T_0}$ and $\varphi_{T_1}$ are, respectively, the upper and
lower lines of (\ref{diag:TopInv}), they are equal. 
\end{proof}

We can now give the main result of the paper. 
\begin{theo}\label{th:Iso}
  The map $\varphi$ : $
    \rTHn  \longrightarrow  \AutC(\RFn) $, sending $T$ to  $\varphi_T$ is an isomorphism.
\end{theo}
As pointed out in Section \ref{ref:CorrTTleMonde}, this is a
consequence of Theorem \ref{th:wGP=AutC}.

Some examples are given in Section \ref{sec:desexemplesdeoufdanstaface} below. 

\subsection{Classification of ribbon torus-links up to link-homotopy}
\label{sec:LinkClassification}

In \cite{HL2} a structure theorem was given for certain ``con\-cor\-dance-type'' equivalence relations on links in $3$-space.  
Here we give an analogous structure theorem in the higher dimensional
case. 
Actually, we follow the reformulation given in \cite{HMe}, which was in fact implicit in the proof of \cite{HL2}. 

We consider $n$-component \emph{ribbon torus-links}, that is, 
locally flat embeddings of $n$ disjoint tori in $S^4$ which bound
locally flat immersed solid tori whose singular set is a finite number of ribbon disks.
Denote by $\rLn$ the set of $n$-component ribbon torus-links up to isotopy. 
The tube-closure operation defined in Remark \ref{remarquefermeture}
induces a natural closure map $\hat{\ }: \rTn\rightarrow \rLn$, 
which is easily seen to be surjective. 
Indeed, given an $n$-component ribbon torus-link, it is always
possible up to isotopy to find a $3$-ball intersecting the $n$ components
transversally exactly once, along an essential circle, so that cutting
the ribbon torus-link along this ball provides a preimage for the closure operation.
We shall refer to such a ball as a ``base-ball''. 

Consider an equivalence relation $E$ on the union, for all $n\in\N^*$,
of the sets $\rTn$ and $\rLn$.   
We will denote by $E(x)$ the $E$-equivalence class of a ribbon tube or
torus-link $x$, and we also denote by $E$ the map which sends a
ribbon knotted object to its equivalence class.
We denote respectively by $E\rTn$ and $E\rLn$ the set of
$E$-equivalence classes of ribbon tubes and ribbon torus-links.  

The Habegger-Lin Classification Scheme relies on the following set of axioms: 
\begin{itemize}
\item[(1)]  The equivalence relation $E$ is \emph{local}, i.e. for all $L_1,L_2\in \rTn$ such that $E(L_1)=E(L_2)$, and for all $T_1,T_2\in \rT_{2n}$ such that $E(T_1)=E(T_2)$, we have 
  \begin{itemize}
  \item[(i)]  $E(\hat L_1)=E(\hat L_2)$,
  \item[(ii)]  $E(\mathbf{1}_n\otimes L_1)=E(\mathbf{1}_n\otimes L_2)$, where $\otimes$ denotes the horizontal juxtaposition, 
  \item[(iii)]  $E(L_1\lhd T_1)=E(L_2\lhd T_2)$ and $E(T_1\rhd L_1)=E(T_2\rhd L_2)$, 
       where the left action $\lhd$ and right action $\rhd$ of $\rT_{2n}$ on $\rTn$ are defined in Figure \ref{leftright}. 
  \end{itemize}
\begin{figure}[!h]
\[
\dessin{3cm}{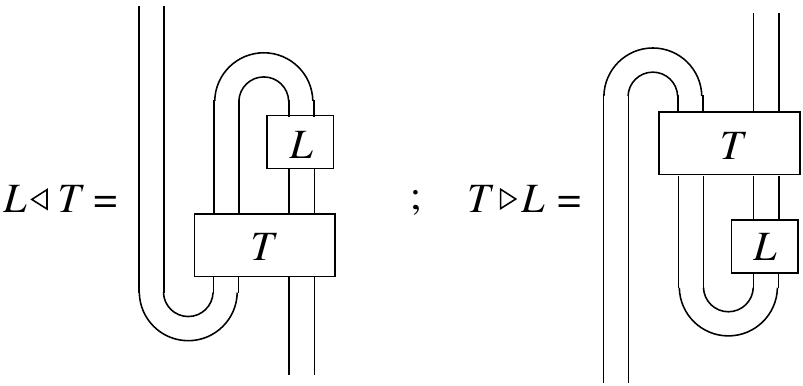}
\]
\caption{Schematical representations of the left and right actions of $T\in \rT_{2n}$ on $L\in \rTn$} 
\label{leftright}
\end{figure}
\item[(2)]  
For all $L\in \rTn$, there is a ribbon tube $L'$, such that $E(L\cdot L')=E(\mathbf{1}_n)$.
\item[($2'$)]  
For all $L\in \rTn$, $E(L\cdot \overline L)=E(\mathbf{1}_n)$, where
$\overline L$ denotes the image of $L$ under the hyperplane reflexion
about $B^3\times \{\frac{1}{2}\}$.
\item[(3)]  
The equivalence relation $E$ on ribbon torus-links is generated by isotopy of ribbon torus-links and the equivalence relation $E$ on ribbon tubes: 
If $L$ and $L'$ are two ribbon torus-links such that $E(L)=E(L')$, then there is a finite sequence $L_1,\dots, L_m$ of ribbon tubes such that $L$ is isotopic to $\hat L_1$,
$L'$ is isotopic to $\hat L_m$, and for all $i$ ($1\le i< m$), either $E(L_i)=E(L_{i+1})$ or $\hat L_i$ is isotopic to $\hat L_{i+1}$.
\end{itemize}

Let $E$ be a local equivalence relation.  
Denote respectively by  $ES^R_n$ and $ES^L_n$ the right and left
stabilizers of the trivial ribbon tube in $E\rTn$. 
One can easily check that 
$ES^R_n$ and $ES^L_n$ are both submonoid of $\rT_{2n}$.  
Furthermore, the closure operation induces a map  $\hat{\ }: E\rTn\rightarrow E\rLn$ which passes to the quotient by $ES^R_n$ (resp. $ES^L_n$).

Now, assume in addition that the equivalence relation $E$ satisfies Axiom $(2)$. 
Then clearly the monoid $E\rTn$ is a group, and both $ES^R_n$ and $ES^L_n$
are subgroups of $E\rT_{2n}$.  
If the stronger Axiom $(2')$ holds, then we actually have $ES^R_n=ES^L_n$. 

\begin{theo}[Structure Theorem for ribbon torus-links]  \hfill
\begin{itemize}
\item Let $E$ be a local equivalence relation satisfying axiom $(2)$. 
Then, for $*=R$ or $L$, the quotient map 
\[\rTn\longrightarrow \fract{E\rTn}/{ES^*_n} \]
factors through the closure map, i.e., we have a ribbon torus-link invariant
\[ \widetilde{E}: \rLn \longrightarrow \fract{E\rTn}/{ES^*_n}.  \]
such that the composite map to $E\rLn$ is $E$.

\item Furthermore, if Axiom $(3)$ also holds, then we have a bijection
$$ \fract{E\rTn}/{ES^*_n} = E\rLn.$$
\end{itemize} 
\end{theo}

This structure theorem is shown by applying verbatim the arguments of
\cite{HL2}, as reformulated in Theorem 3.2 of \cite{HMe}. Indeed,
although these papers only deal with classical knotted objects, the
proof is purely combinatorial and algebraic, and involves no
topological argument \emph{except} for \cite[Prop. 2.1]{HL2}, whose
ribbon tube analogue can actually be shown by a
straightforward adaptation of Habegger and Lin's arguments.

We have the following classification result for ribbon torus-links up to link-homotopy.
\begin{theo}\label{prop:classif}
The link-homotopy relation on ribbon tubes satisfies Axioms $(1)$, $(2')$ and $(3)$ above.
Consequently, we have a bijection
$$ \fract{\rTHn}/{S^+_n} = \rLHn,$$
where $\rLHn$ is the set of link-homotopy classes of ribbon torus-links and $S^+_n$ denotes the stabilizer of the trivial ribbon tube in $\rTHn$ with respect to the right (or left) action of $\rT_{2n}$ on $\rTn$ defined in Figure \ref{leftright}.  
\end{theo}
\begin{proof}
The fact that the link-homotopy relation is local (Axiom $(1)$) is evident.
Axiom $(2')$ holds as a consequence of Theorem \ref{thm:tube_braids}
and the group structure on monotone ribbon tubes as described
in the proof of Proposition \ref{prop:Abitbol}. 
To see that Axiom $(3)$ holds, suppose that the ribbon torus-link $L'$
is obtained from $L$ by applying circle crossing changes at a given
set $S$ of self-crossing circles. 
Let $L$ be the tube-closure of a ribbon tube $L_1$. As sketched at the
beginning of this section, $L_1$ is specified by the choice of a
base-ball intersecting each component of $L$ exactly once, and we may assume that this ball is disjoint from $S$.
Thus $L'$ is the tube-closure of a ribbon-tube $L_2$, obtained from
$L_1$ by successive circle crossing changes at each of the circles in
the set $S$.  
This shows that link-homotopy for ribbon torus-links is implied by
link-homotopy for ribbon tubes and isotopy, hence concludes the proof. 
\end{proof}

\subsection{Some examples}\label{sec:desexemplesdeoufdanstaface}

We conclude with a few simple examples of applications of the classification theorems \ref{th:Iso} and \ref{prop:classif}.  

Let $H$ be the ribbon tube represented in Figure \ref{fig:longitudemeridien}, and denote by $H^n$ the stacking product of $n$ copies of $H$. 
Using Proposition \ref{prop:WirtingerProjection}, we have that 
$\varphi(H^n)$ acts on $\RF_2$ by mapping $x_1$ to $x_2^{n}x_1x_2^{-n}$ and fixing $x_2$.  
By Theorem \ref{th:Iso}, this provides an infinite $1$-parameter family of $2$-component ribbon tubes that are pairwise non link-homotopic. 

This result is easily tranferred to the case of ribbon torus-links. 
Namely, consider the tube-closure of $H$, which is the ribbon torus-link shown on the right-hand side of Figure \ref{fig:desfiguresdeoufdanstaface}, and more generally the tube-closure of $H^n$ for all $n\ge 1$. These torus-links are pairwise non link-homotopic, which is in striking contrast with the case of ribbon $2$-links, where all objects are link-homotopically trivial \cite{BT}.
This can be proved topologically, by considering the homology class of a longitude of a torus component in the complement of the other component, or algebraically, as an application of Theorem \ref{prop:classif}.
\begin{figure}[!h]
  \[
  \begin{array}{ccc}
    \dessin{3.5cm}{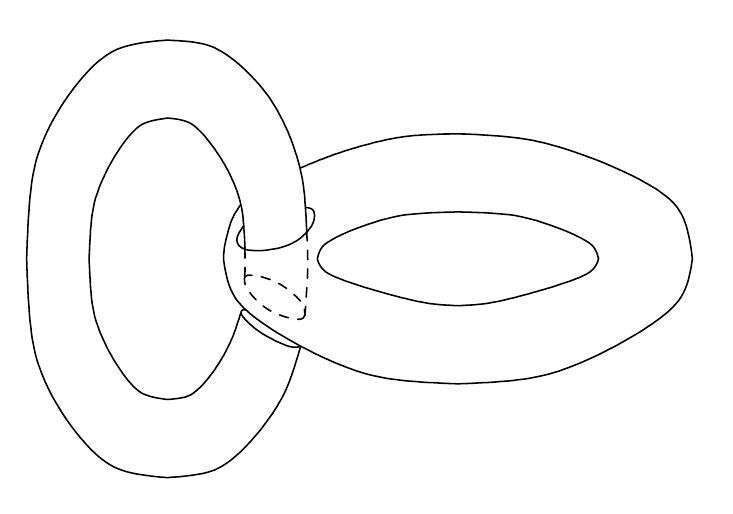}&\hspace{1cm}&\dessin{5.5cm}{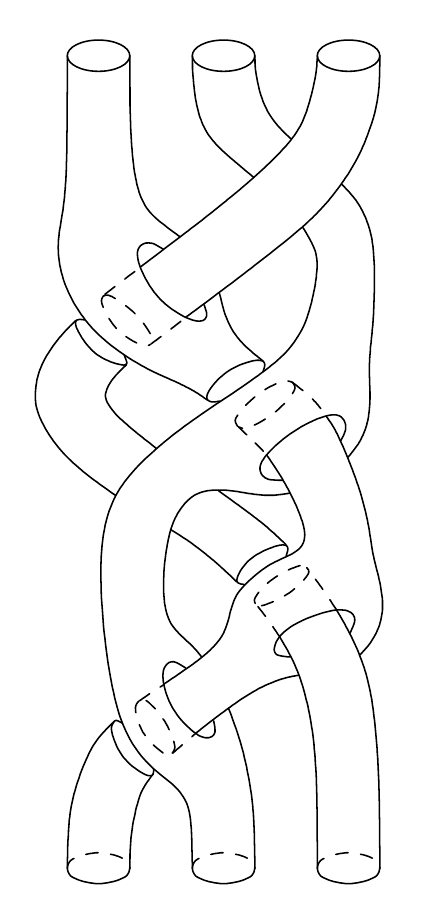}
  \end{array}
  \]
  \caption{Examples of a ribbon torus-link and a ribbon tube that are not link-homotopically trivial}
  \label{fig:desfiguresdeoufdanstaface}
\end{figure}

Further exemples involving more components can also be provided.  
For example, consider the $3$-com\-po\-nent ribbon tube $B$ shown on the right-hand side of Figure \ref{fig:desfiguresdeoufdanstaface}. 
Note that $B$ is \emph{Brunnian}, meaning that removing any component yields the trivial $2$-component ribbon tube. 
Unlike for the ribbon tube $H$, the homology class of a longitude in the complement of the other two components doesn't detect $B$.   
But using Proposition \ref{prop:WirtingerProjection}, the automorphism 
$\varphi(B)$ of $\RF_3$ fixes $x_1$ and $x_2$ and maps $x_3$ to its conjugate by $[x_2^{-1};x_1^{-1}]$. 
Theorem \ref{th:Iso} then implies that $B$ is not link-homotopic to the trivial ribbon tube.  
One can check that the tube-closure of $B$ is likewise not link-homotopic to the trivial torus-link.   
One-parameter families of pairwise non link-homotopic $3$-component ribbon tubes and torus-links can also be obtained by stacking multiples copies of $B$.


\section{Welded string links}
\label{sec:welded-diagrams}
In this section, we introduce two classes of diagrammatic objects,  welded string links and welded pure braids, and we explain how they are
related to the topological objects of the previous section. In particular, we give a classification result for welded string links up to self-virtualization.

\subsection{Definitions}
\begin{defi}\label{def:StringLink}
An $n$-component \emph{virtual string link diagram} is a locally flat
immersion $L$ of $n$ intervals $\psqcup_{i\in\UnN}I_i$ in
$\loverstar{B^2}{B}$, called \emph{strands},
such that
\begin{itemize}
\item each strand $I_i$ has boundary $\p I_i=\{p_i\}\times\{0,1\}$ and is oriented from $\{p_i\}\times\{0\}$ to $\{p_i\}\times\{1\}$ ($i\in\UnN$);
\item the singular set $\Sigma(L)$ of $L$ is a finite set of flatly transverse points.
\end{itemize}
Moreover, for each element of $\Sigma(L)$, a partial ordering is given on
the two preimages.
If the preimages are comparable, the double point is called \emph{a classical crossing}, if
not, it is called \emph{a virtual crossing}.
By convention, this ordering is specified on pictures by erasing a
small neighborhood in $\pcup_{i=1}^n\mathring{a}_i$ of the
lowest preimage of classical crossings, and by circling the
virtual one (see Figure \ref{fig:crossings}).\\
A classical crossing is said to be \emph{positive} 
if and only if the basis made of the tangent vectors of the highest and lowest preimages 
is positive. Otherwise, it is \emph{negative} (see Figure \ref{fig:crossings}).\\
Up to isotopy, the set of virtual string link diagrams is naturally endowed with a monoidal structure by the
stacking product $L\bullet L':=L\pcup_{\p_1L=\p_0L'}L'$, and with unit element the trivial diagram $\pcup_{i\in \UnN} p_i\times I$. 
\end{defi}

\begin{figure}[!h]
\[
\begin{array}{ccc}
  \dessin{2cm}{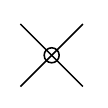}&\hspace{1cm}&\vcenter{\hbox{\xymatrix@R=-1cm{&\dessin{1.6cm}{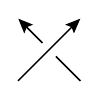}\
        \textrm{\ : positive}\\\dessin{2cm}{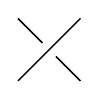}\ar[ur]\ar[dr]&\\&\dessin{1.6cm}{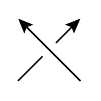}\
        \textrm{\ : negative}}}}\\
  \textrm{virtual}&&\textrm{classical}\hspace{1.5cm}
\end{array}
\]
 \caption{Virtual and classical crossings} \label{fig:crossings}
\end{figure}

\begin{defi}\label{def:MonotoneSLn}
  A virtual string link diagram is said to be \emph{monotone} if it is flatly transverse to the lamination
  $\pcup_{t\in I}I \times\{t\}$ of $B^2$.
\end{defi}

\begin{defi}
  Two virtual string link diagrams are equivalent if 
  they are related by a finite sequence of the moves, called
  \emph{generalized Reidemeister moves}, represented in
  Figure \ref{fig:AllReidMoves}. There, all lines are pieces of strands which
  may belong to the same strand or not, and can have any orientation. \\
  We denote by $\vSLn$ the quotient of $n$-component virtual string link diagrams up to
  isotopy and generalized Reidemeister moves, which is compatible with the
  stacking product. We call its elements $n$-component \emph{virtual string links}.
\end{defi}
\begin{figure}[!h]
  \[
  \begin{array}{c}
    \begin{array}{c}
       \textrm{RI}:\ \dessin{1.5cm}{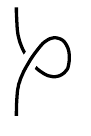}
    \leftrightarrow \dessin{1.5cm}{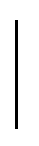} \leftrightarrow \dessin{1.5cm}{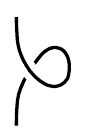}
\hspace{.7cm}
\textrm{RII}:\ \dessin{1.5cm}{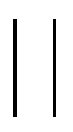} \leftrightarrow \dessin{1.5cm}{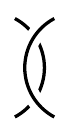}\\[1cm]
\textrm{RIIIa}:\ \dessin{1.5cm}{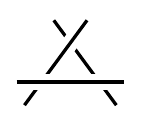} \leftrightarrow \dessin{1.5cm}{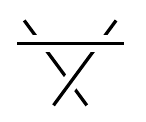}
\hspace{.7cm}
\textrm{RIIIb}:\ \dessin{1.5cm}{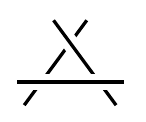} \leftrightarrow \dessin{1.5cm}{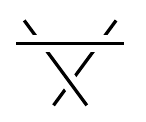}
\end{array}\\[2cm]
    \textrm{classical Reidemeister moves}
  \end{array}
  \]
\vspace{.5cm}
  \[
     \begin{array}{c}
       \textrm{vRI}:\ \dessin{1.5cm}{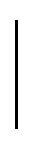}
    \leftrightarrow \dessin{1.5cm}{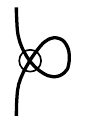}
\hspace{.7cm}
\textrm{vRII}:\ \dessin{1.5cm}{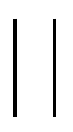} \leftrightarrow \dessin{1.5cm}{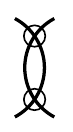}
\hspace{.7cm}
\textrm{vRIII}:\ \dessin{1.5cm}{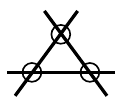} \leftrightarrow \dessin{1.5cm}{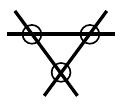}\\[1cm]
    \textrm{virtual Reidemeister moves}
  \end{array}
   \]
\vspace{.5cm}
  \[
      \begin{array}{c}
       \textrm{mRIIIa}:\ \dessin{1.5cm}{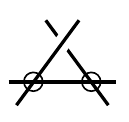}
    \leftrightarrow \dessin{1.5cm}{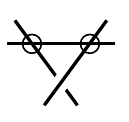}
\hspace{.7cm}
\textrm{mRIIIb}:\ \dessin{1.5cm}{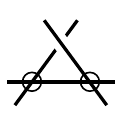} \leftrightarrow \dessin{1.5cm}{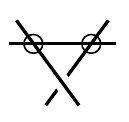}\\[1cm]
    \textrm{mixed Reidemeister moves}
  \end{array}
  \]
  \caption{Generalized Reidemeister moves on diagrams}
  \label{fig:AllReidMoves}
\end{figure}

Virtual string links are interesting objects on their own, but since
we are motivated in the first place by applications to ribbon tubes,
we will focus on the following quotient.

\begin{defi}
We define the Over Commute (OC) move as
\[
\OC:\ \dessin{1.5cm}{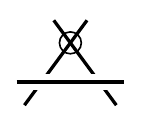}\ \longleftrightarrow \dessin{1.5cm}{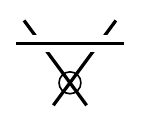}.
\]
\noi We denote by $\wSLn:=\fract{\vSLn}/{\OC}$ the quotient of $\vSLn$
  up to $\OC$ moves, which is compatible with the
  stacking product. We call its elements
  $n$-component \emph{welded string links}.\\
  We denote by $\wPn$ the subset of $\wSLn$ whose elements admit a
  monotone representative.
\end{defi}

\begin{warnings}$\ $
  \begin{itemize}
  \item The following Under Commute (UC) move
    \[
    \begin{array}{rcl}
      \UC:\ \dessin{1.5cm}{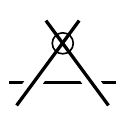}&\ \longleftrightarrow\
      &\dessin{1.5cm}{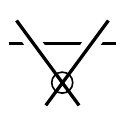}\\[-.96cm]
      &\textrm{\textcolor{red}{$\times$}}&\\[.6cm]
    \end{array},
    \]
    \noi was forbidden in the virtual context and is still forbidden
    in the welded context.
  \item Virtual and welded notions do not coincide, even for $n=1$, where we get respectively the notion of virtual and welded long knots (see \cite{WKO1}).  
  \end{itemize}
\end{warnings}

Similarly to the ribbon tube case, it is straightforwardly checked that
\begin{prop}
  The set $\wPn$ is a group for the stacking product.
\end{prop}
\begin{remarque}\label{rem:wTBW}
Note that if two monotone virtual string link diagrams are equivalent in $\wSLn$, then they are related by a
monotone transformation, that is by a sequence of isotopies, generalized Reidemeister and OC moves 
which remain within the set of monotone virtual string links.
Indeed, it is a consequence of Remark \ref{rem:star2} and Proposition \ref{prop:DiagGaussDiag} that any element of $\wPn$ 
induces an action in $\AutC(\Fn)$.\footnote{This fact can also be checked directly, 
by a straightforward adaptation of Rem. \ref{rem:star2} and Prop. \ref{prop:DiagGaussDiag} to welded string links. }
But $\AutC(\Fn)$ is isomorphic to the group $PUR_n$, according to Theorem 2.6 of \cite{WKO1}, and a presentation of the latter 
is given by monotone virtual string links up to monotone transformations \cite{BH}. 
Since two equivalent monotone welded string links induce the same action on $\Fn$, they are related by a
monotone transformation (this is the same argument as in Remark \ref{rem:star}).
It follows that $\wPn$ is isomorphic to the welded pure braid group studied, for instance, in \cite{WKO1}.  
In other words, we have that the  welded pure braid group injects into $\wSLn$. 
On that account, we will freely call welded pure braids the elements of $\wPn$. 
\end{remarque}

\begin{remarque} \label{rem:VirtualBraids}
The subset of $\vSLn$ whose elements admit a monotone representative is, of course, also defined. 
It is a group for the stacking product, 
which maps surjectively onto the  virtual pure braid group, introduced in \cite{Bardakov}.
However, the injectivity of this map remains an open question.
Indeed, unlike in the welded case addressed in Remark \ref{rem:wTBW} above, 
it is still unknown whether a sequence of isotopies and generalized Reidemeister moves can always be
modified into a monotone transformation.
\end{remarque}

\begin{defi}\label{def:selfvirtualization}
  Two virtual string link diagrams are related by a \emph{self-virtualization}
  if one can be obtained from the other by turning a classical self-crossing 
  (i.e. a classical crossing where the two preimages belong to the same component) into a virtual one.
  We call $\sv$--equivalence the equivalence relation on $\wSLn$ generated by self-virtualization. \\
  We denote  by $\wSLHn$ the quotient of $\wSLn$ under $\sv$--equivalence, which is compatible
  with the stacking product.
  We also denote by $\wPHn\subset\wSLHn$ the subset of elements having a monotone representative.
\end{defi}

As we will see in Section \ref{ref:CorrTTleMonde}, we have the following result as a straightforward consequences of  Theorem \ref{th:wGD=wGP}.
 \begin{theo} \label{th:wSLh=wPh}
    Every welded string link is monotone up to self-virtualization. 
  \end{theo}
Furthermore, Theorem \ref{th:wGP=AutC} implies immediately the following classification result for welded string links up to self-virtualization.    
  \begin{theo} \label{th:wSLh=AutC}
    $\wSLHn\cong \wPHn\cong\AutC(\RFn)$ as monoids. 
  \end{theo}

The notions of self-virtualization and self-crossing change for welded string links and pure braids, as well as for their usual and virtual counterparts, is further studied in \cite{pdm}. 


\subsection{Fundamental group for welded string link}
Let $L$ be a virtual string link diagram.

Any strand of $L$ is cut into smaller pieces by the classical crossings.
More precisely, we call \emph{overstrand}, any piece of
strand of $L$ whose boundary elements consist of either a
strand endpoint or the lowest preimage of a classical crossing, and
such that it contains no other lowest preimage of any classical
crossing in its interior.
Note that any highest preimage of a classical crossing is contained in an overstrand.
We denote by $\Over(L)$ and $\Cross(L)$ the sets of, respectively,
overstrands of $L$ and classical crossings of $L$.

We orient all strands from $\p_0 L$ to $\p_1 L$.
To any $c\in\Cross(L)$, we associate, as show in Figure \ref{fig:RulePN}, 
\begin{itemize}
\item $\e_c$ the sign of $c$;
\item $s_c^0$ the overstrand containing the highest preimage
  of $c$;
\item $s_c^-$ the overstrand whose exiting boundary component
 is the
  lowest preimage of $c$;
\item  $s_c^+$ the overstrand whose entering boundary component
 is the
  lowest preimage of $c$.
\end{itemize}

\begin{defi}
  We define the \emph{fundamental group of $L$} as $\pi_1(L):=\big\langle \Over(L)  \ |\
  s_c^+=(s_c^-)^{(s_c^0)^{\e_c}}\textrm{ for all }c\in\Cross(L) \big\rangle$. 
  See Figure \ref{fig:RulePN} for an illustration.
\end{defi}

\begin{figure}
  \[
\xymatrix@!0@R=1.5cm@C=4.5cm{
&\dessin{2.2cm}{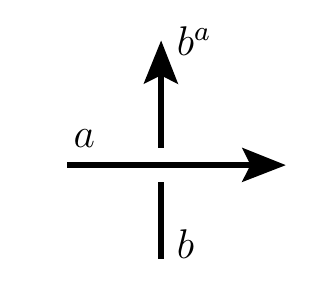}\\
\dessin{2.5cm}{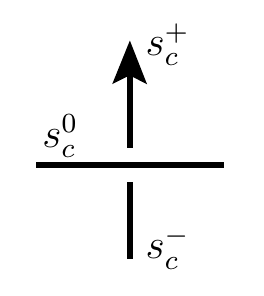}\ar[ru]^-{\e_c=1}\ar[rd]_-{\e_c=-1}\\
&\dessin{2.2cm}{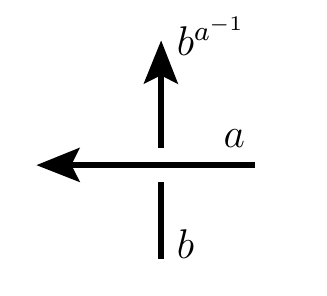}
}
\]
  \caption{Local relations for $\pi_1(L)$}
  \label{fig:RulePN}
\end{figure}

It is well-known that, up to isomorphism, the group associated to a virtual string link diagram is invariant under classical
Reidemeister moves. Kauffman proved that virtual and mixed Reidemeister moves do not change the group presentation \cite{Kauffman}.
It turns out that the ``virtual knot group'' is also invariant under Over Commute, and is thus a welded invariant, i.e.  
is well defined on $\wSLn$ \cite{Kauffman,Citare}.


\subsection{Relations with ribbon tubes}\label{sec:tubemap}

It was shown in Section \ref{sec:broken-surfaces} that 4--dimensional ribbon tubes can be described by 3--dimensional objects, namely symmetric broken surface diagrams.
Following \cite{Citare} and \cite{yajima}, it  is  also possible to  describe ribbon tubes using 2--dimensional welded string links.

Indeed, let $L$ be a welded string link diagram.
One can associate a symmetric broken surface diagram by embedding $B^2$ into
$B^3$ as $B^2\times\big\{\frac{1}{2}\big\}$ and considering 
a tubular neighborhood $N(L)$ of $L$ in $B^3$ so that $\p_\e N(L)=\psqcup_{i\in\UnN}D_i\times\{\e\}$ for $\e\in\{0,1\}$.
The boundary of $N(L)$ then decomposes as a union of $4$--punctured spheres, one for each
crossing. Then, according to the partial order on the associated crossing, we modify each
sphere as shown in Figure \ref{fig:Inflating}.
The result is a symmetric broken surface diagram, to which we can associate a ribbon tube $\Tube(L)$.

\begin{figure}[!h]
\[
\xymatrix@C=1.35cm@R=-2cm{
&&\vcenter{\hbox{\rotatebox{45}{$\dessin{2.5cm}{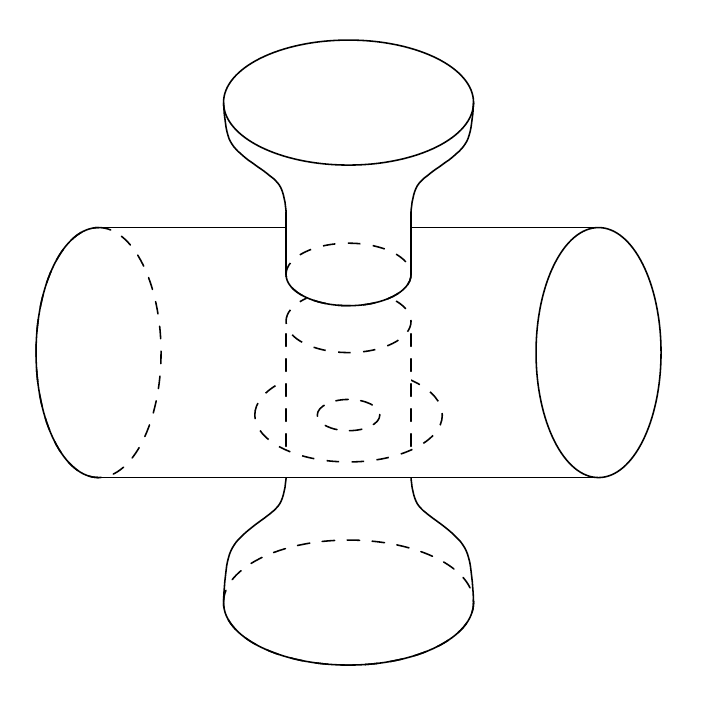}$}}}\\
\vcenter{\hbox{\rotatebox{45}{$\dessin{2.5cm}{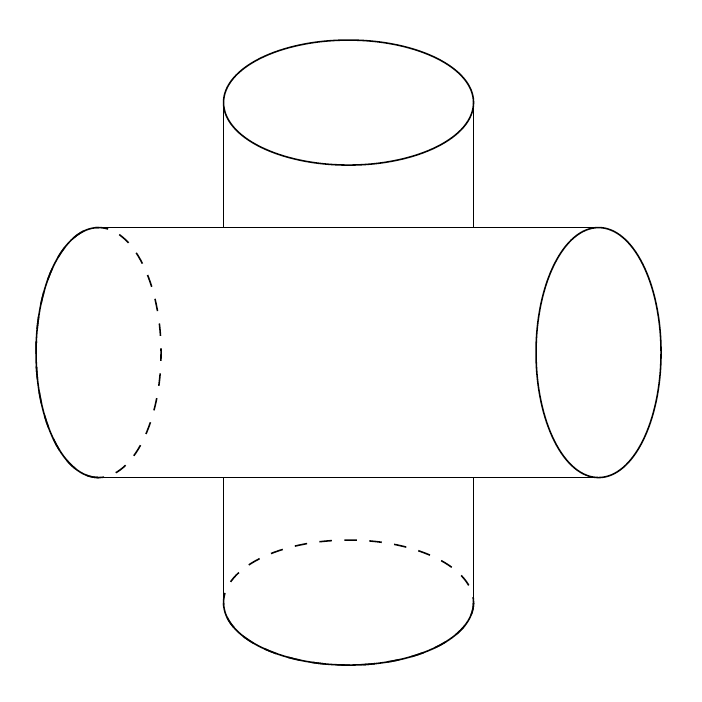}$}}}&\vcenter{\hbox{\rotatebox{45}{$\dessin{2.5cm}{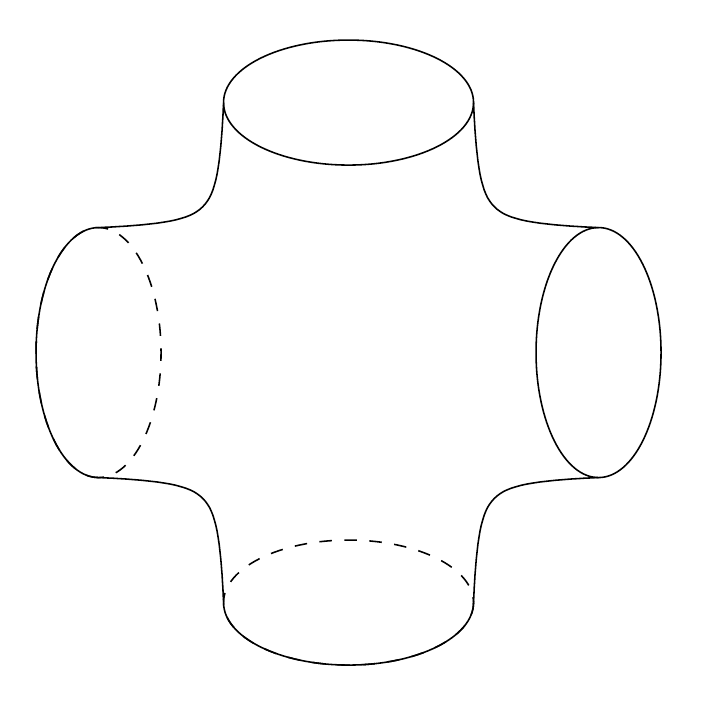}$}}}\ar[l]|(.48){\dessin{.8cm}{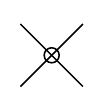}}\ar[ru]|(.48){\dessin{.8cm}{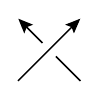}}\ar[rd]|(.48){\dessin{.8cm}{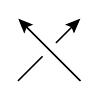}}&\\
&&\vcenter{\hbox{\rotatebox{45}{$\dessin{2.5cm}{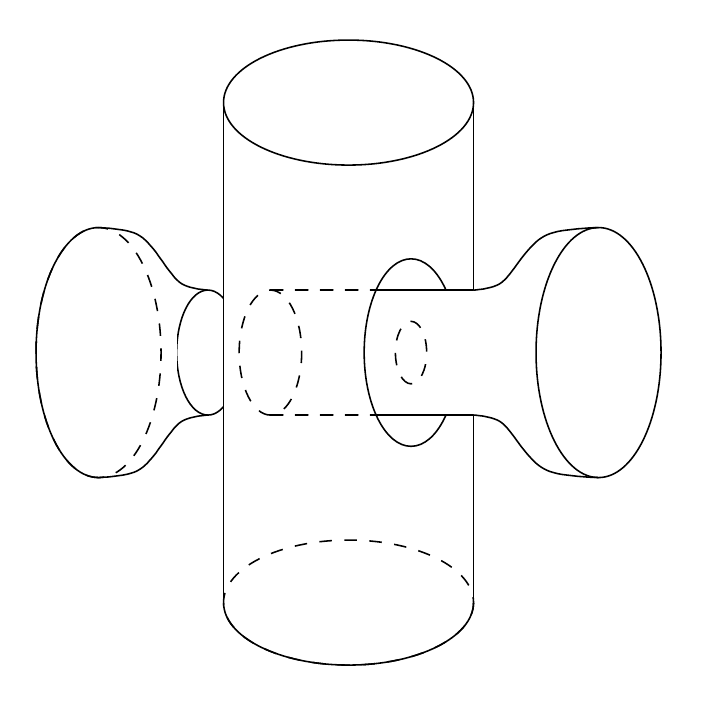}$}}}
}
\]
  \caption{Inflating classical and virtual crossings}
  \label{fig:Inflating}
\end{figure}

There is a one-to-one correspondence between
overstrands of $L$ and the outside annuli of
the associated symmetric broken surface diagram $S$, and another
one-to-one correspondence between classical crossings of $L$ and
inside annuli of $S$. Moreover, signs
associated to crossing in the definition of $\pi_1(L)$ also corresponds
to signs associated to inside annuli in the Wirtinger presentation of
$\pi_1\big(\Tube(L)\big)$ (compare Figures \ref{fig:BS_Signs} and \ref{fig:RulePN}).
Therefore we have

\begin{prop}[\cite{Citare, yajima}]\label{prop:DWirtinger}
For every welded string link diagram $L$, $\pi_1\big(\Tube(L)\big)\cong \pi_1(L)$.
\end{prop}

But the $\Tube$ map is more than just a tool to compute fundamental
groups. Indeed, it provides a way to encode
ribbon tubes.
This has been pointed out by Satoh, but some key
ideas already appeared in early works \cite{yajima} of Yajima.
\begin{prop}[\cite{Citare}]\label{prop:Satoh}
The map $ \Tube: \wSLn \rightarrow \rTn $ is well defined and surjective.
\end{prop}
\begin{figure}
  \[
\dessin{4cm}{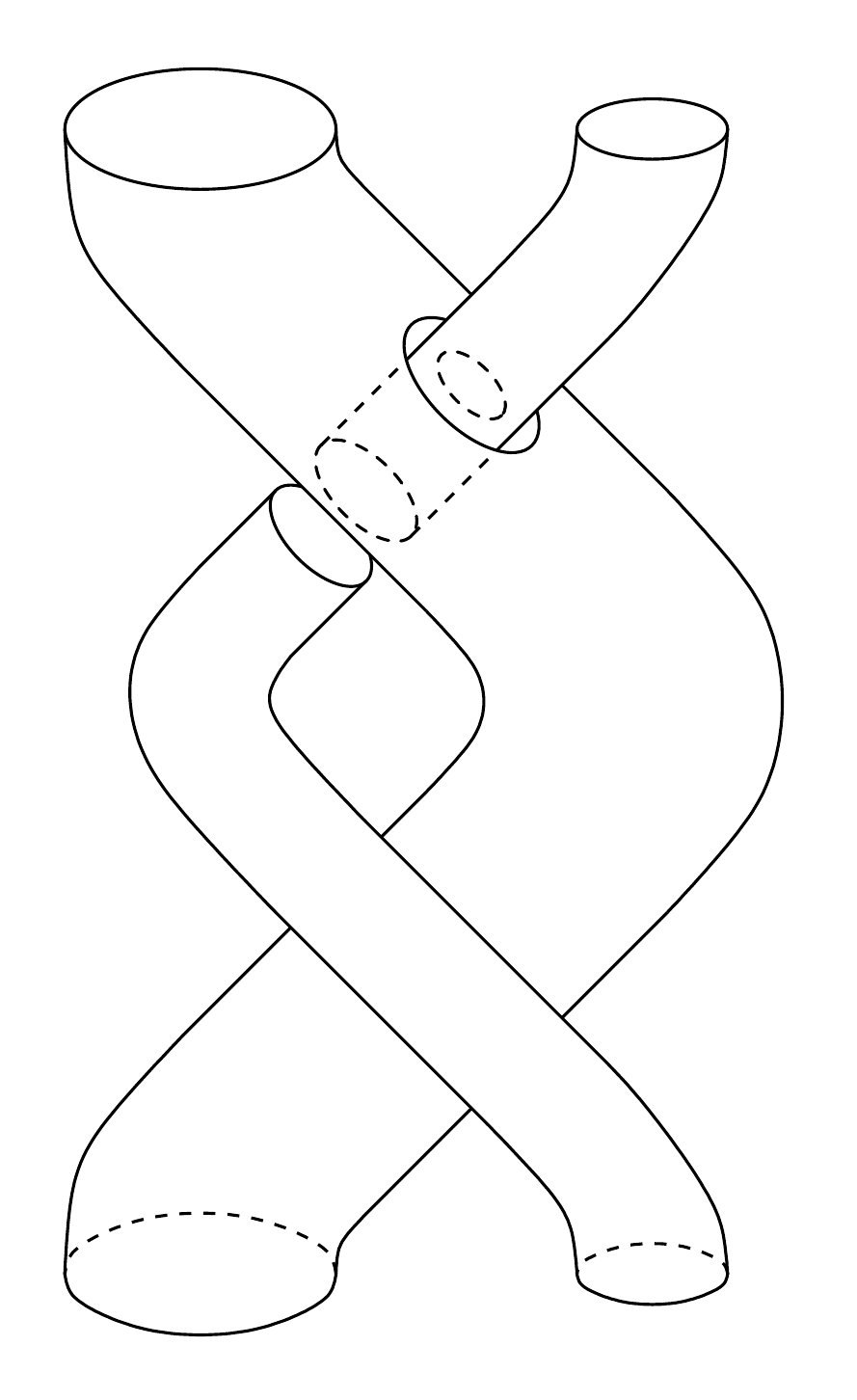}\leadsto \dessin{4cm}{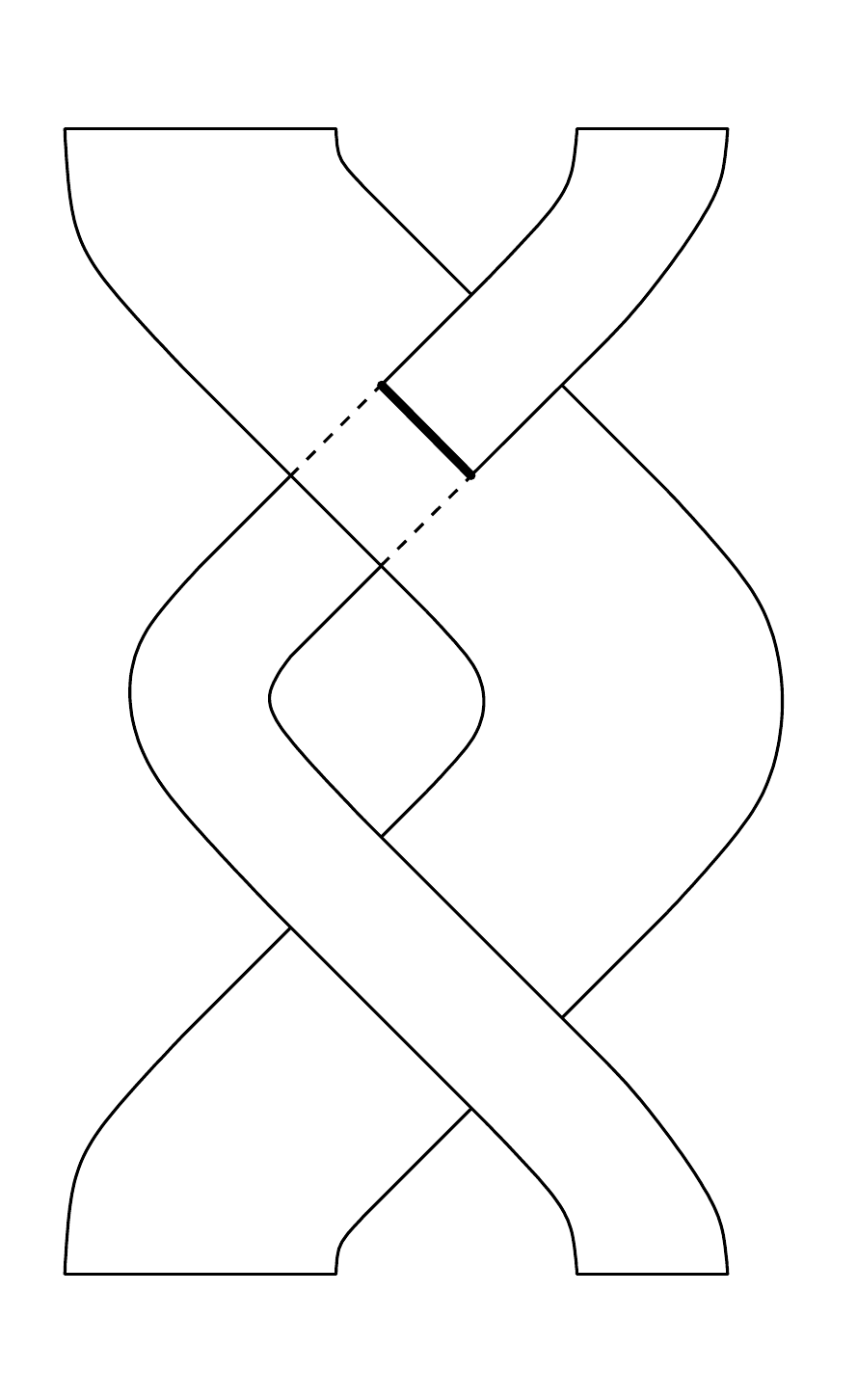}\leadsto
\dessin{3cm}{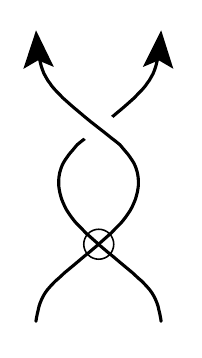}
\]
  \caption{Contracting symmetric broken surface diagrams}
  \label{fig:Contracting}
\end{figure}
\begin{proof}
  It is easily seen that classical Reidemeister moves on welded diagrams
corresponds to Reidemeister moves on symmetric broken surfaces defined in Remark \ref{rk:Reidbrok}, 
and that virtual and mixed Reidemeister moves, as well as $\OC$,
preserves the associated broken surface diagram.  
It then follows from Remark \ref{rk:Reidbrok} that $\Tube$ is well defined
on $\wSLn$.

Now, a symmetric broken surface diagram $S$ can be seen as given by a ribbon band, 
in the sense of Figure \ref{fig:RibbonBands}. 
Contracting each band $I\times I$ onto its core $\big\{\frac{1}{2}\big\}\times I$
so that, at each ribbon singularity, the cores intersect transversally, yields a singular string
link. Let $D$ be a diagram for this singular string link.
By turning the classical crossings of $D$
into virtual ones and its singular crossings into classical with signs
corresponding to the initial local ordering on $S$, we obtain a welded
string link diagram which is sent to $S$ by the above process. This
operation is illustrated in Figure \ref{fig:Contracting}.

Note however that some inside annuli of $S$ may have to be turned
around so that the sign rule given in Figure \ref{fig:Inflating} can be reversed, 
as illustrated in Figure \ref{fig:TurningAround}.  
Surjectivity of the $\Tube$ map follows then from Lemma \ref{lem:symmetric}.
\end{proof}

\begin{figure}[!h]
  \[
\dessin{2cm}{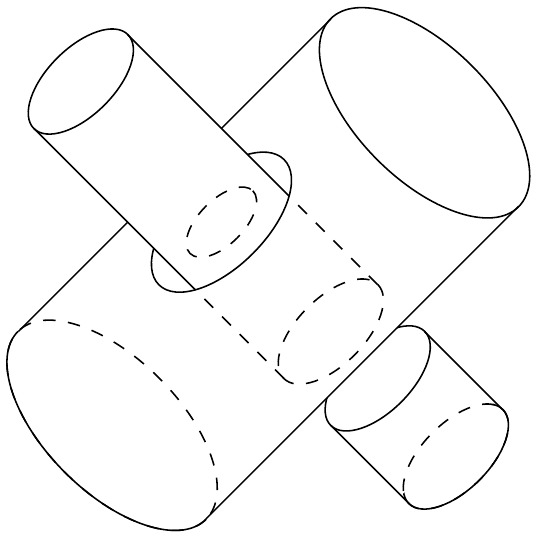}\hspace{.3cm} \simeq \hspace{.3cm} \dessin{4cm}{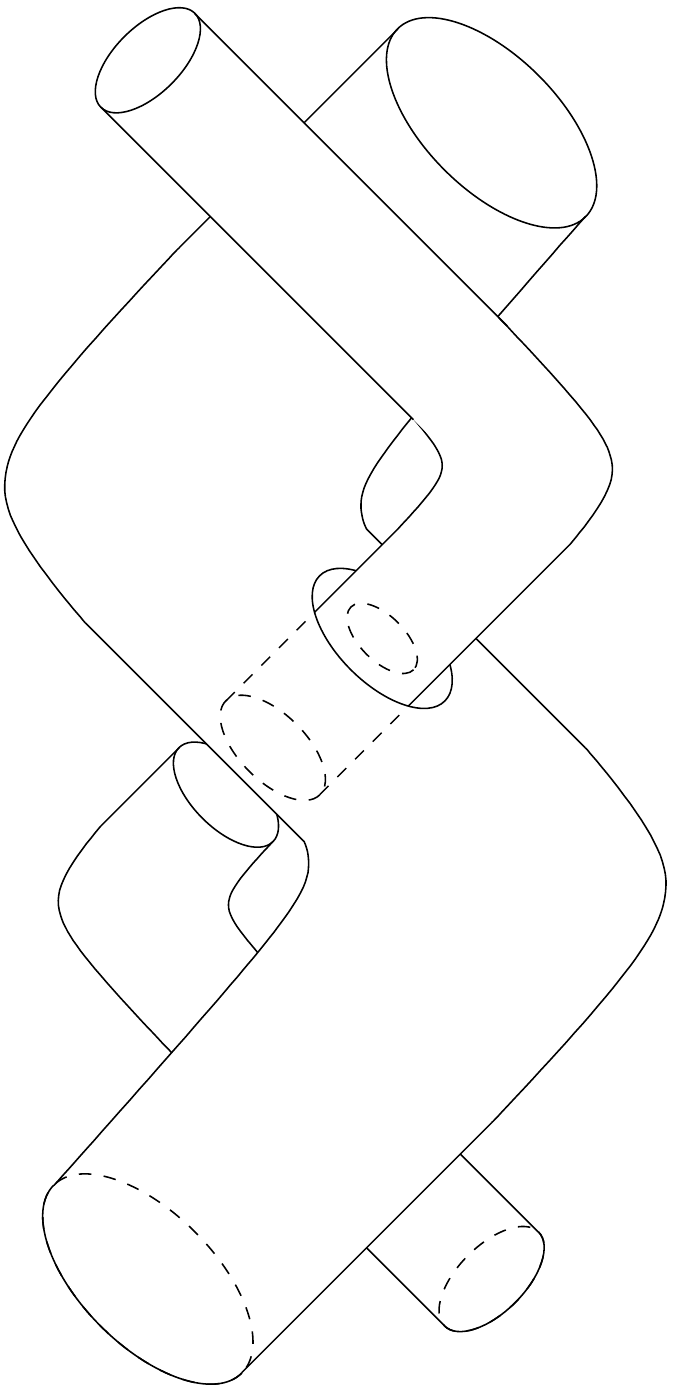}\hspace{.3cm} \longrightarrow \hspace{.3cm} \dessin{3cm}{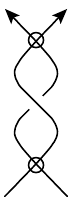}
\]
  \caption{Turning inside annuli around}
  \label{fig:TurningAround}
\end{figure}

\begin{remarque}\label{rem:injectivityTube}
  The injectivity of the  $\Tube$ map is still an open question; see \cite{WKO1,WKO2}.\footnote{In the knot case, the $\Tube$ map is known to have nontrivial kernel; see for example \cite{IK}. } 
  However, Brendle and Hatcher proved in \cite{BH} that it is an isomorphism when
  restricting to monotone objects on both sides.
  This can be compared with the following proposition,
  although it is a not a consequence of Brendle--Hatcher's result since link-homotopy is inherently a non-monotone transformation.
\end{remarque}

\begin{prop}\label{prop:DBraidIso}
  The map $ \Tube: \wSLHn \rightarrow \rTHn $ is a well defined group isomorphism.
\end{prop}

\begin{proof}
  It is easily seen that a self-virtualization corresponds to a
  self-circle crossing change, see Figure \ref{fig:SBS_LM}.
It is a surjective group homomorphism by
Proposition \ref{prop:Satoh}. Injectivity is immediate after
Corollary \ref{cor:TubeInj} and Proposition \ref{prop:DiagGaussDiag}.
\end{proof}


\section{Gauss diagrams}
\label{sec:gauss-diagrams}
Our main tool for the study of ribbon tubes and welded knotted objects is the theory of Gauss diagrams \cite{F,GPV,PV}. 

\subsection{Definitions}

\begin{defi}
A \emph{Gauss diagram} is a set of signed and oriented (thin) arrows
between points of $n$ ordered and oriented vertical (thick) strands,
up to isotopy of the underlying strands. Endpoints of arrows are
called \emph{ends} and are divided in two parts, \emph{heads} and
\emph{tails},  defined by the orientation of the arrow (which goes by
convention from the tail to the head).
\end{defi}
This definition is illustrated on the left-hand side of Figure \ref{fig:DGD_Corr}. 

For all $i\in\UnN$, we will denote the $i^\textrm{th}$ strand by $I_i$.
An arrow is said to be \emph{connected} to a strand if
it has an end on this strand.
An arrow having both ends on the same strand is called a
\emph{self-arrow}.

There is a natural stacking product operation, denoted by $\bullet$, for Gauss
diagrams defined by gluing the top endpoints of the strands of the first summand to
the bottom endpoints of the strands of the second.

\begin{defi}
  A Gauss diagram is said to be \emph{horizontal} if all of its arrows are horizontal.
\end{defi}

\begin{remarque}
An alternative definition for a Gauss diagram $G$ being horizontal
would be to ask for a global order on the set of its arrows such that
if two arrows $a_1$ and $a_2$ with $a_1\leq a_2$ have ends $e_1$ and $e_2$ on a same
strand, then $e_1$ is below $e_2$ on this strand.
Note in particular that it forbids self-arrows. In the proof of Theorem \ref{th:wGD=wGP}, we shall refer to this global order.
\end{remarque}

\begin{defi}\label{def:Rmoves}
  Two Gauss diagrams are equivalent if 
  they are related by a finite sequence of the following moves, called
  \emph{Reidemeister moves}:
  \[
   \Ru:\ \dessin{2.5cm}{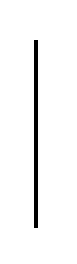}
    \leftrightarrow \dessin{2.5cm}{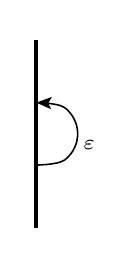}
\hspace{.7cm}
\Rd:\ \dessin{2.5cm}{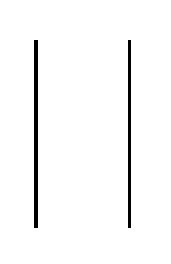} \leftrightarrow \dessin{2.5cm}{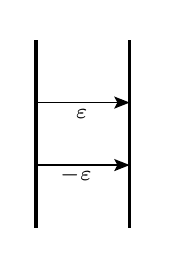}
\hspace{.7cm}
\Rt:\ \dessin{2.5cm}{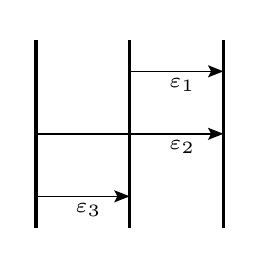} \leftrightarrow \dessin{2.5cm}{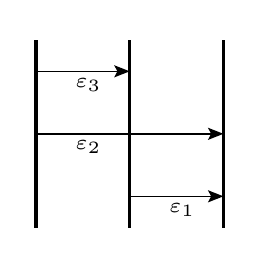}
  \]
  Here, all vertical lines are pieces of strands which
  may belong to the same strand or not, and can be oriented upward or
  downward; each $\e_*$, for $*\in\{\emptyset,1,2,3\}$, is
  either $1$ or $-1$. 
  Moreover, there is an additional condition for applying move R3: it is required that 
  $\tau_1\varepsilon_1=\tau_2\varepsilon_2=\tau_3\varepsilon_3$, where
  $\tau_i=1$ if the $i^\textrm{th}$ strand (from left to right) is
  oriented upwards, and $-1$ otherwise. 

  We denote by $\vGDn$ the quotient of Gauss diagrams up to isotopy
  and Reidemeister moves, which is compatible with the
  stacking product. 
\end{defi}

For the study of ribbon tubes and welded string links, we consider the following quotient of Gauss diagrams.

\begin{defi}
We define the Tail Commute (TC) move as
\[
\TC:\ \dessin{2.5cm}{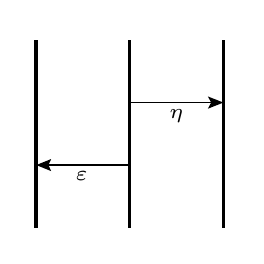}\ \longleftrightarrow\ \dessin{2.5cm}{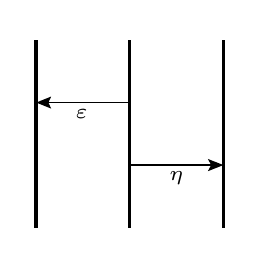},
\]
\noi where $\varepsilon,\eta\in\{\pm1\}$.  
  We denote by $\wGDn:=\fract{\vGDn}/{\TC}$ the quotient of $\vGDn$
  by relation TC, which is compatible with the
  stacking product. We call its elements
  \emph{welded Gauss diagrams}.
  We also denote by $\wGPn\subset\wGDn$ the subset of elements which
  have a horizontal representative.
\end{defi}

\begin{remarque}
  In the welded case, that is when quotiented by TC, the condition
  $\tau_1\varepsilon_1=\tau_2\varepsilon_2=\tau_3\varepsilon_3$ for 
  move R3 can be simplified to $\varepsilon_2\varepsilon_3=\tau_2\tau_3$. 
  In other words, the signs of the two arrows which have their tails on the same
  strand should agree if and only if the two others strands have parallel orientations.
\end{remarque}

The following can be easily checked.
\begin{prop}
  The set $\wGPn$ is a group for the stacking product.
\end{prop}

\begin{defi}
  Two Gauss diagrams are
  related by a \emph{self-arrow} move, denoted by $\SA$,
if one can be obtained from the other by removing a self-arrow.\\
  We denote by $\wGDHn$ the quotient of $\wGDn$ by $\SA$, which is compatible
  with the stacking product.
  We also denote by $\wGPHn\subset\wGDHn$ the subset of elements which
  have a horizontal representative.\\
Two elements of $\wGDn$ are said to be \emph{$\sa$--equivalent} if 
they are sent to the same element in $\wGDHn$.
\end{defi}

\subsection{Commutation of arrows}

In this section, we address the notion of commutation of arrows, 
which means swapping the relative position of two adjacent arrow ends on a strand. 
The Tail Commute move is a special case of such a commutation, when both ends are tails. 

More generally, and even in the non welded case, commutations of arrows can always be performed at the cost of some additional surrounding arrows.    
In the welded case, these additional arrows can be conveniently positioned, as shown below: 

\begin{defi}
  The $\textrm{C}^3$ moves are local moves on Gauss diagrams defined in three versions, shown in Figure \ref{fig:C3Moves}.
  There, $\e$ and $\eta$ are either $1$ or $-1$, all strands are assumed to be simultaneously upward or downward oriented, 
  and non oriented arrows can have either orientation.
\end{defi}

\begin{figure}
  \[
  \begin{array}{c}
    \hspace{-.6cm}
\Ctu:\ \dessin{2.5cm}{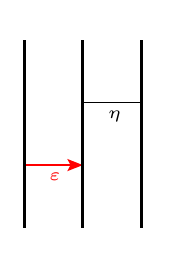}\xrightarrow[\ \Rd\
      ]{}\dessin{2.5cm}{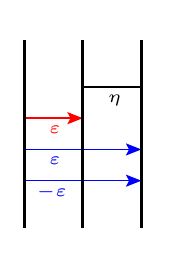}\xrightarrow[\ \Rt\
      ]{}\dessin{2.5cm}{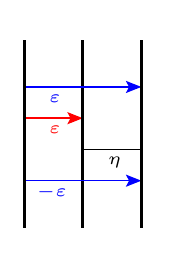}\xrightarrow[\ \TC\
      ]{}\dessin{2.5cm}{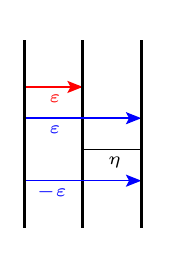}\\
    \begin{array}{ccc}
      \Ctd:\ \dessin{2.5cm}{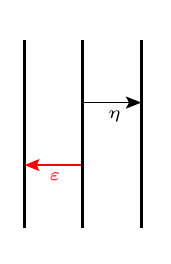}\xrightarrow[\ \TC\
      ]{}\dessin{2.5cm}{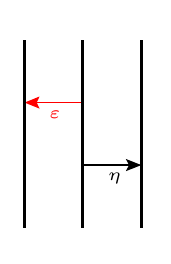}&\hspace{.5cm}&
   \Ctt:\ \dessin{2.5cm}{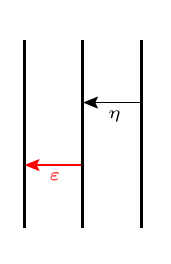}\xrightarrow[\ \Ctu\ ]{}\dessin{2.5cm}{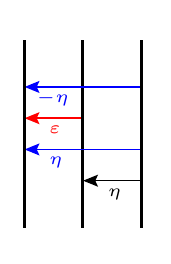}\xrightarrow[\ \Ctu\ ]{}\dessin{2.5cm}{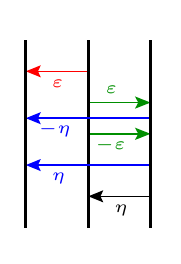}
    \end{array}
  \end{array}
  \]
\caption{The $\textrm{C}^3$ moves on welded Gauss diagrams}
\label{fig:C3Moves} 
\end{figure}

The $\textrm{C}^3$ moves can be seen as pushing one arrow accross another, at the cost of several additional arrows located below the pushed one.
Note that $\textrm{C}^3$ moves are supported by three pieces of strands, which may or may not belong to the same components.  

By iterated $\textrm{C}^3$ moves, one has the following general commutation rule: 
\begin{cor}\label{cor:Bunches}
 In $\wGDn$, commuting an arrow with a bunch of adjacent arrows all connected to a far away strand 
 is achieved at the cost of additional arrows connected to the far away strand, as shown below:
\[
\dessinH{2cm}{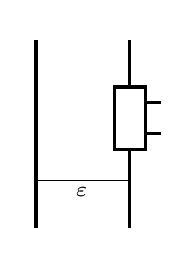}\ \longrightarrow \dessinH{2cm}{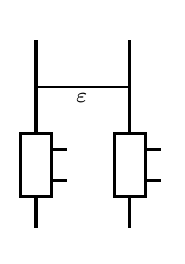}
\]
where $\dessin{1cm}{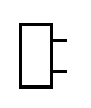}$ denotes a bunch of arrows (not necessarily with the same number of arrows), 
$\varepsilon$ is either $1$ or $-1$, and all orientations are arbitrary.
\end{cor}
\begin{proof}
  If the two strands are simultaneously upward or downward, it is a direct applications of the $\textrm{C}^3$ moves as described in Figure \ref{fig:C3Moves}, where the far away strand is the rightmost one. If not, the final step of moves $\textrm{C}^3_1$ and $\textrm{C}^3_3$ have to be changed consequently.
\end{proof}

When working up to SA, and dealing with arrows whose ends are on the same two strands, a genuine commutation result holds (without additional arrows), 
even in the non welded case. We shall refer to such a commutation as a $\textrm{C}^2$ move.
\begin{prop}
  Up to self-arrow moves, two arrows with adjacent ends can commute whenever the other two ends are on the same strand.
\end{prop}
\begin{proof}
  The strategy is to add a self arrow on the strand which does not support the adjacent ends, so that an R3 move involving the three arrows can be performed, 
  and finally to remove the self-arrow. See Figure \ref{fig:C2Move} for an example. 
  Nonetheless, the self-arrow has to be choosen so that the R3 move is valid. 
  This means that the global position of the three arrows should be as in the definition of an R3 move (Def.~\ref{def:Rmoves}), 
  and the condition $\tau_1\varepsilon_1=\tau_2\varepsilon_2=\tau_3\varepsilon_3$ should hold. 
  However, we have a complete freedom in choosing the self--arrow: orientation, sign, and relative position to the adjacent arrow ends. 
  The choice of orientation ensures that the arrows are in position of a R3 move, 
  while the sign and positions give independant control on $\tau_1\varepsilon_1$, $\tau_2\varepsilon_2$ and $\tau_3\varepsilon_3$. 
\end{proof}
\begin{figure}
  \[
\dessin{2.5cm}{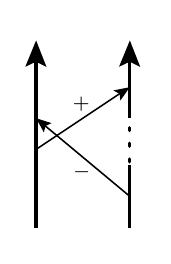}\xrightarrow[\ \SA\ ]{}\dessin{2.5cm}{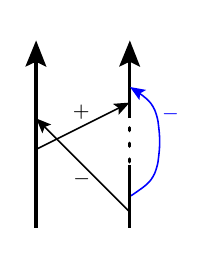}\xrightarrow[\ \Rt\ ]{}\dessin{2.5cm}{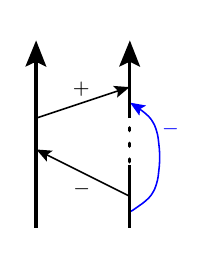}\xrightarrow[\ \SA\ ]{}\dessin{2.5cm}{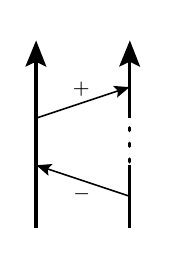}
\]
  \caption{Example of a $\textrm{C}^2$ move}
  \label{fig:C2Move}
\end{figure}

\subsection{Reduction of welded Gauss diagrams to horizontal Gauss diagrams}

 The main result of this section is the following
\begin{theo}\label{th:wGD=wGP}
 Every welded Gauss diagram is $\sa$--equivalent to a horizontal Gauss
 diagram. 
  Equivalently, the natural inclusion $\wGPHn\hookrightarrow\wGDHn$ is onto, that is 
  $\wGPHn\cong \wGDHn$.
\end{theo}
\begin{proof}
  We prove the statement by induction on $n$.
 For $n=1$ the statement is trivial since $\wGDH_1$ is reduced to one element.

  Now, we assume that the result is true for $n\in\N$ and we consider
  $G\in\wGD_{n+1}$. We choose a strand $I$ of $G$ and we call
  \emph{$I$--arrow} the arrows connected to $I$. We denote
  by $\hG$ the Gauss diagram obtained from  $G$ by removing the strand $I$, 
  that is, by removing all  $I$--arrows and forgetting $I$. 
  The welded Gauss diagram $\hG$ is then a diagram on $n$ strands, and by the 
  induction hypothesis, there is a finite sequence $S$ of moves R1, R2,
  R3, $\TC$ and $\SA$ which transforms $\hG$ into a horizontal Gauss diagram $\tB$.
  Now, this sequence cannot be directly performed on $T$, since most moves
  require that some ends of arrows are adjacent 
  on strands, and $I$--arrows may interfere in that. 
  Nevertheless, Figure \ref{fig:Imoves} illustrates how to use Corollary \ref{cor:Bunches} 
  to push away such $I$--arrows at the cost of additional $I$--arrows, 
  and perform the desired moves. There, boxes
  $\dessin{1cm}{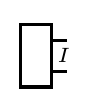}$ stands for bunches of $I$--arrows.

  \begin{figure}[!h]
    \[
    \begin{array}{l}
      \Ru:\ \dessin{2.5cm}{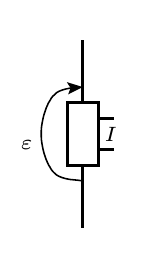} \xrightarrow[\ \SA\ ]{}
      \dessin{2.5cm}{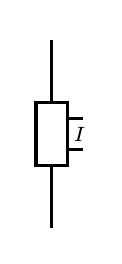}
      \hspace{.7cm}
      \Rd:\ \dessin{2.5cm}{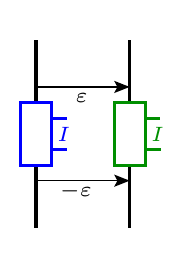} \xrightarrow[\ \Ct\textrm{'s}\ ]{\dessin{.7cm}{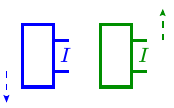}}\dessin{2.5cm}{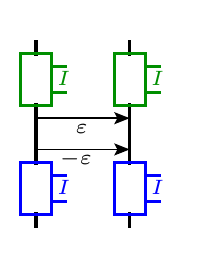}\xrightarrow[\ \Rd\ ]{}\dessin{2.5cm}{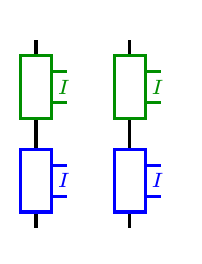}\\
      \Rt:\ \dessin{2.5cm}{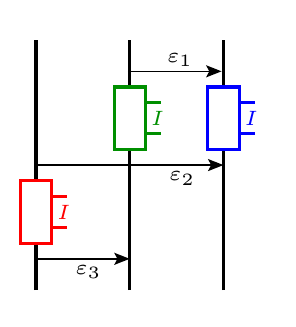} \xrightarrow[\ \Ct\textrm{'s}\ ]{\dessin{.7cm}{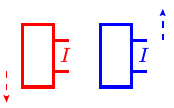}}
      \dessin{2.5cm}{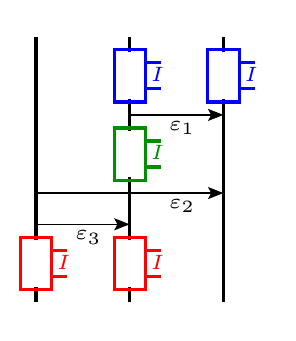}\xrightarrow[\ \Ct\textrm{'s}\ ]{\dessin{.7cm}{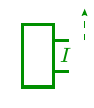}}
      \dessin{2.5cm}{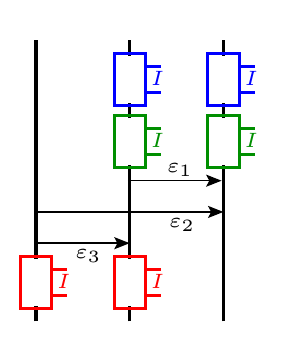}\xrightarrow[\ \Rt\ ]{}
      \dessin{2.5cm}{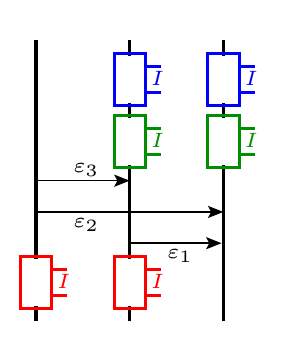}
      \\
      \TC:\ \dessin{2.5cm}{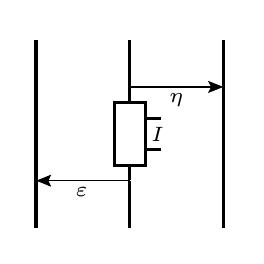} \xrightarrow[\ \Ct\textrm{'s}\ ]{\dessin{.7cm}{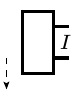}}
      \dessin{2.5cm}{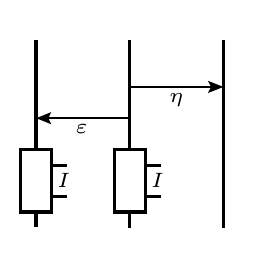}\xrightarrow[\ \TC\ ]{}
      \dessin{2.5cm}{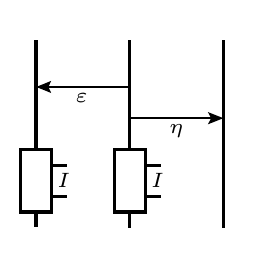}.
    \end{array}
    \]
  \caption{Performing Reidemeister and TC moves obstructed by $I$-arrows}
\label{fig:Imoves} 
\end{figure}

It follows that the whole sequence $S$ can be performed on $T$, giving a representative of $T$ obtained by adding $I$--arrows to $\tB$.

Now, since $\tB$ is horizontal, 
there is a natural total order, from bottom to top, on its arrows.  The arrows of $\tB$ can then be 
pushed up above the $I$--arrows, successively in decreasing order,  using moves $\Ct$ (for $I$ being the strand on the right). 
This leads to the following decomposition:
\[
\dessin{3.5cm}{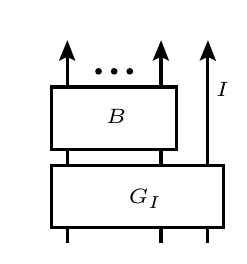}
\]
\noi where $\tB$ is horizontal and $G_I$ contains only $I$--arrows. 
Now, for any strand $J\neq I$ restricted to $G_I$, 
all arrows connected to $J$ are also connected to $I$, so that one can use moves $\Cd$ 
(for $I$ being the strand on the right), to rearrange the ends on strand $J$ so that their order corresponds
to the order induced by $I$. Applying this operation for all strands $J\neq I$, 
we eventually obtain a horizontal Gauss diagram $\widetilde{G}_I$, 
such that $G$ is  $\sa$--equivalent to the product of
$\widetilde{G}_I$ with $\tB$, which is still horizontal.
This concludes the proof.  
\end{proof}

\begin{remarque}
As noted above, $\textrm{C}^2$ moves
still hold in $\vGDHn$, i.e. without using TC moves. 
This implies that any element of $\vGDH_2$ is
$\sa$--equivalent to a horizontal Gauss diagram.
This is not likely to hold for $\vGDHn$ with $n\geq 3$.
\end{remarque}

\begin{cor}
  The set $\wGDHn$ is a group for the stacking product.
\end{cor}

Now we establish a second normalized form for welded Gauss diagrams up to $\sa$--equivalence.  
\begin{defi}\label{Defi:AscendingDiag}
  A Gauss diagram is said to be \emph{ascending}
  if all tails belong to the lowest halves of the strands, whereas all
  heads are on the highest halves, that is, if any tail is below any head.
\end{defi}

\begin{lemme}\label{lem:ascending}
  Every welded Gauss diagram is $\sa$-equivalent to an ascending Gauss diagram.
\end{lemme}
\begin{proof}
  Consider a Gauss diagram and choose an arbitrary order for its  strands. 
  Then, using move $\Ctu$ repeatedly, each strand can be successively sorted, 
  in the sense that all tails can be pushed below all heads.
  When performing such $\Ctu$ moves, the tails (resp. heads) of added arrows are in the neighborhood
  of a tail (resp. head) of a pre-existent arrow, so none of the
  performed $\Ctu$ moves will unsort the already sorted strands.
\end{proof}

\subsection{Classification of welded Gauss diagram up to $\sa$-equivalence.}

In this section, we prove the following theorem. 
 \begin{theo}\label{th:wGP=AutC}
 There is a group isomorphism $\wGDHn\cong\AutC(\RFn)$.
 \end{theo}
Recall that $\AutC(\RFn)$ is the group of conjugating
automorphisms defined in Section \ref{sec:general-setting}.
The isomorphism between $\wGDHn$ and $\AutC(\RFn)$ is given explicitely.
In Section \ref{sec:fDA} we establish the existence of a group homomorphism 
 $\func{\fDA}{\wGDHn}{\AutC(\RFn)}$, 
and in Section \ref{sec:fAD} we construct an inverse $\fAD$ for $\fDA$.

\subsubsection{Definition of $\fDA$}\label{sec:fDA}

To construct the map $\fDA$, we need some notation.

\begin{defi}
  Let $G$
  be a Gauss diagram.
  A \emph{tail interval} is a pair
  $(h_1,h_2)$ such that, either
  \begin{itemize}
  \item $h_1$ and $h_2$ are heads on a same strand, with $h_1$ lower
    than $h_2$, and there is no other head between them;
  \item $h_1$ is a head and $h_2$ is the top endpoint of the strand containing $h_1$, 
     and there is no other head on this strand above $h_1$;
  \item $h_2$ is a head and $h_1$  is the bottom endpoint of the strand containing $h_2$, 
     and there is no other head on this strand below $h_2$;
  \item $h_1$ and $h_2$ are respectively the bottom and top endpoints of 
    a strand that doesn't contain any head. 
  \end{itemize}
  Graphically, tail intervals are portions of strand comprised between
  two successive heads and/or strand endpoints. A tail interval may contain some tails, but no head.\\
We denote by $T_G$ the set of all tail intervals of $G$.
\end{defi}

\begin{nota}\label{Nota:TailsIntervals}
  Let $G$ be a Gauss diagram.
  For every head $h$, we denote by $T^+_h$
  (resp. $T^-_h$) the unique tail interval of the form $(h,\ .\ )$
  (resp. $(\ .\ ,h)$) and by $T_h^0$ the unique tail interval that
  contains the tail which is connected to $h$ by an arrow, as illustrated below.
  \[
  \dessin{2cm}{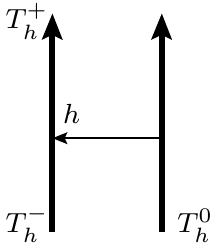}
  \]
  Finally, for every $i\in\UnN$, we denote by
  $T^+_i$ and $T^-_i$ the unique tail intervals containing respectively the top and bottom endpoint of the $i^{\textrm{th}}$ strand
  (note that $T^+_i=T^-_i$ if the $i^{\textrm{th}}$ strand contains no arrow head).\\
\end{nota}

\begin{lemme}\label{lem:wtf}
  For any Gauss diagram $G$, there is a unique
  coloring map $\func{\xi_G}{T_G}{\RFn}$ such that
  \begin{enumerate}
  \item\label{cond:pp} for every $i\in\UnN$, $\xi_G(T^-_i)=x_i$;
  \item\label{cond:ps} for every head $h$,
    $\xi_G(T^+_h)=\xi_G(T^-_h)^{\xi_G(T^0_h)^{\e_h}}$, where $\e_h$ is
      the sign of the arrow that contains $h$.
  \end{enumerate}
Moreover, for every $i\in\UnN$, $\xi_G(T_i^+)$
only depends on the $\sa$--equivalence class of $G$.
\end{lemme}
See remark \ref{rk:asc} for an example. 

Note that condition (\ref{cond:ps}), together with condition (\ref{cond:pp}), imply that, for each $i\in\UnN$, every tail interval which  belongs to $I_i$ is necessarily mapped to a conjugate of  $x_i$.

\begin{proof}
  First, we deal with the case of horizontal Gauss diagrams. 
  If $G$ contains a single $\varepsilon$--labelled non self-arrow, then the lemma is
  clear.
Indeed, the only possible coloring sends $T_1^-$ to $x_1$, $T_1^+$ to $x_1^{x_2^\varepsilon}$
and all other $T_k^+=T_k^-$ to $x_k$, as pictured below. 
  \[
\begin{array}{c}
\dessin{3cm}{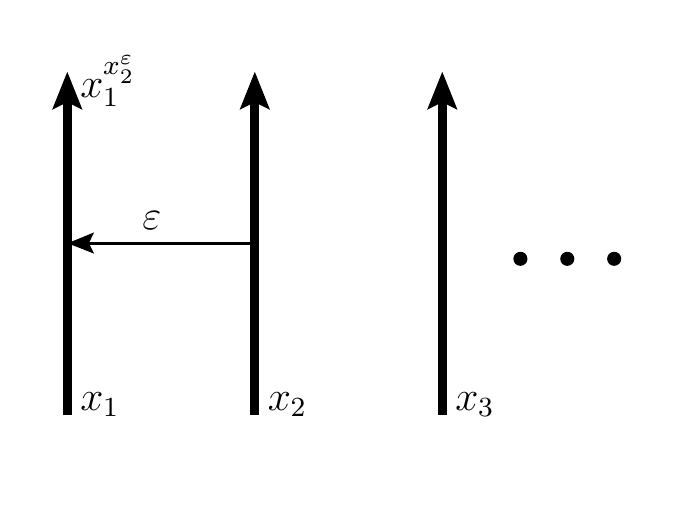}\\
\end{array}
\] 
By induction on the number of arrows the result then follows whenever $G$ is horizontal. 
  
  Now, each Reidemeister, $\TC$ or $\SA$ move between two Gauss diagrams $G_1$
  and $G_2$ induces a one-to-one correspondence between
  the coloring maps $\xi_{G_1}$ and $\xi_{G_2}$.
  These correspondences are shown in Figure \ref{fig:MapsCorr}.
  \begin{figure}[!h]
\[
\begin{array}{c}
  \dessin{3cm}{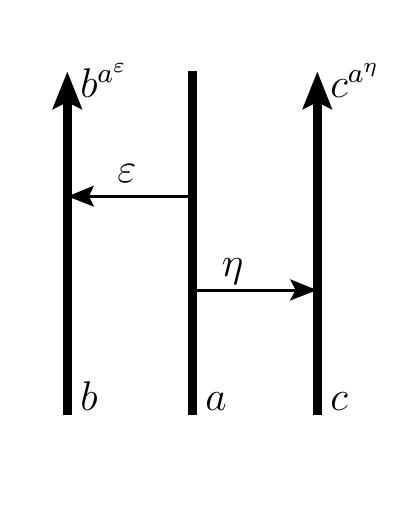} \ \leftrightarrow\ \dessin{3cm}{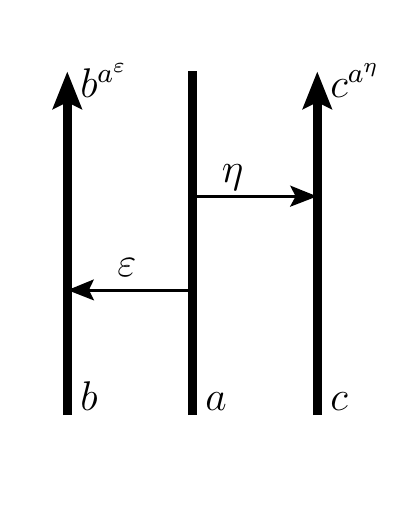}\\
  \textrm{TC}
\end{array}
\]
\[
\begin{array}{ccc}
 \dessin{3cm}{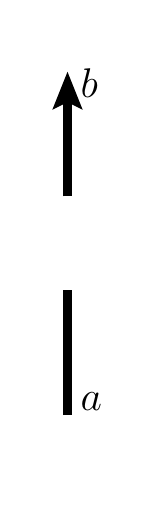} \ \leftrightarrow\ \dessin{3cm}{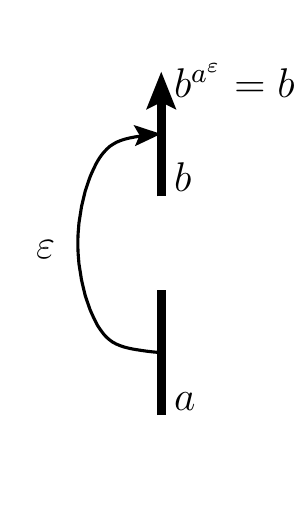}
& \hspace{1cm} &
  \dessin{3cm}{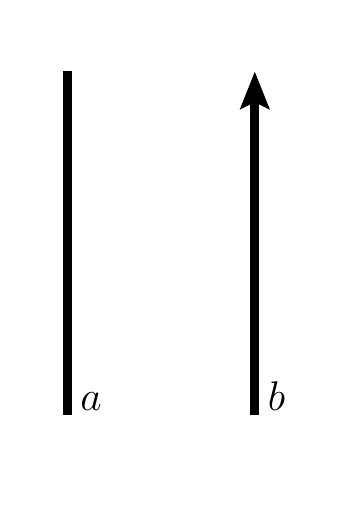} \ \leftrightarrow\ \dessin{3cm}{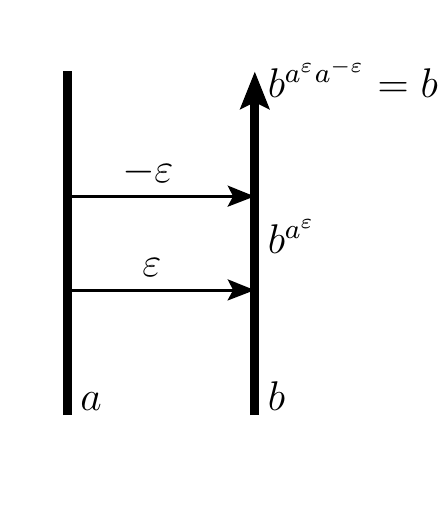} \\
\textrm{$\SA$ and R1}&&\textrm{R2}
\end{array}
\]
\[
\begin{array}{ccc}
 \dessin{3cm}{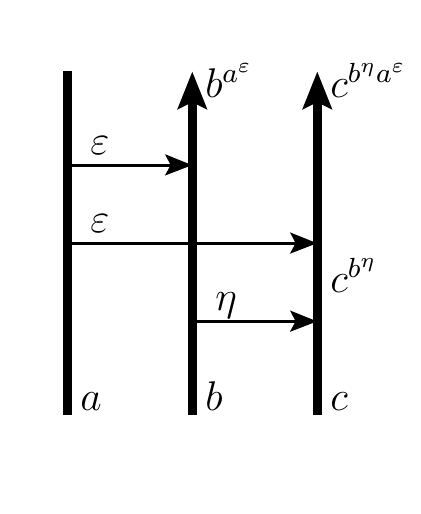} \ \leftrightarrow\ \dessin{3cm}{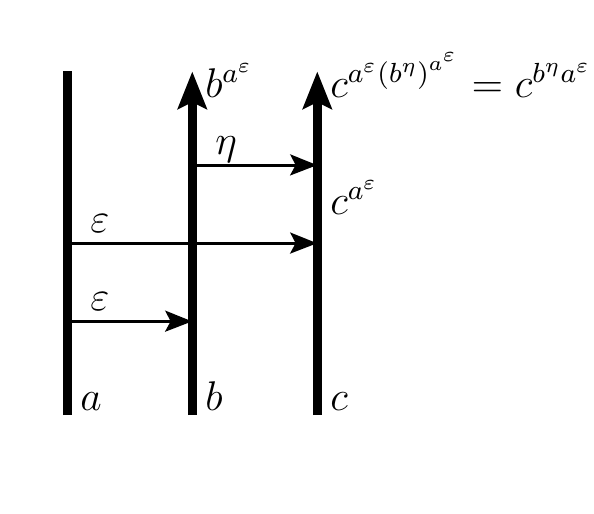}
&\hspace{1cm} &
  \dessin{3cm}{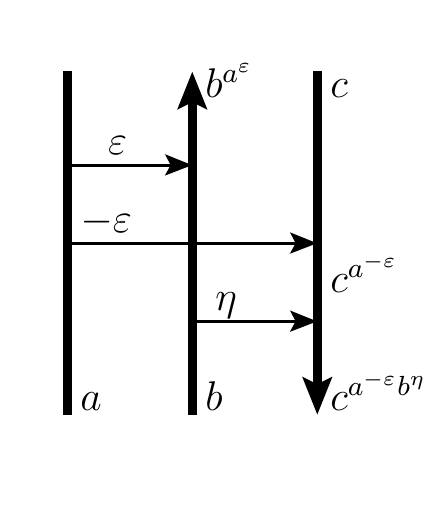} \ \leftrightarrow\ \dessin{3cm}{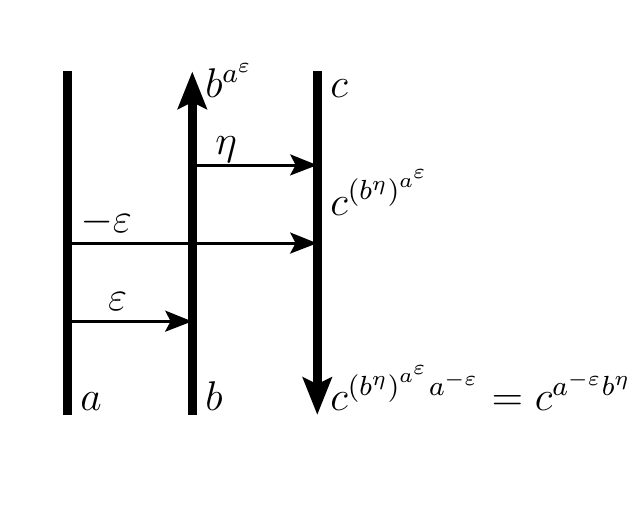}\\
  \textrm{R3, case 1} && \textrm{R3, case 2}
\end{array}
\]
    \caption{One-to-one correspondence between the coloring maps $\xi$}
    \label{fig:MapsCorr}
  \end{figure}
 As usual, for every picture, there is a non represented part which is identical on both
 side of each move. Note that this is consistent since the ends of the
 represented tail intervals are pairwise coherently labelled.
 In particular, in the case of moves $\SA$ and R1, the equality $b^a=b$ holds in
 $\RFn$ since $a$ and $b$ are both conjugate of a same generator $x_i$
 for some $i\in\UnN$. 
  By Theorem \ref{th:wGD=wGP}, this implies that any Gauss diagram has a unique coloring $\xi$ with the desired property.
Moreover, Figure \ref{fig:MapsCorr} also shows that tail intervals which are not of the form $T_i^+$ or
$T_i^-$ are possibly modified by Reidemeister, $\TC$ and $\SA$ moves. As a matter of fact,
$\xi(T_i^+)$ is kept unchanged for every $i\in\UnN$, and depends on the $\sa$--equivalency class of $G$ only.
\end{proof}

\begin{remarque}\label{rk:GenWTF}
  The above proof of Lemma \ref{lem:wtf} can be easily adapted to show that the number of $\Fn$--coloring map $\xi_G:T_G\to \Fn$ for a
  given Gauss diagram $G$ is a well defined invariant on $\GDn$.
Moreover, when such an $\Fn$--coloring map is unique, for instance when
$G\in\GPn$, then the value $\big(\xi_G(T_i)\big)_{i\in\UnN}\in\F_n^n$
is also a well defined invariant.
\end{remarque}

\begin{defi}
  For every $G\in\wGDHn$, we define $\fDA(G)\in\AutC(\RFn)$ as
  the automorphism of $\RFn$ which sends, for every $i\in\UnN$, the
  generator $x_i$ to $\xi_G(T_i^+)$.
\end{defi}

It is easily checked that $\fDA:\wGDHn\to\AutC(\RFn)$ is a group
homomorphism, so that it defines an action of $\wGDHn$ on $\RFn$.

\begin{remarque}\label{rk:asc}
  If $G$ is ascending, as defined in Definition \ref{Defi:AscendingDiag},
  then, for any $i\in\UnN$, the element $g_i\in\RFn$ by which
  $\fDA(G)$ conjugates $x_i$ can be read directly on $G$ as
  $x_{i_1}^{\e_1}\cdots x_{i_s}^{\e_s}$ where
  $T_i^-$ contains exactly the tails $t_1,\cdots,t_s$ in this
  order and for each $k\in\llbracket1,s\rrbracket$, $t_k$ is
  connected by a $\e_k$--signed arrow to a head on $I_{i_k}$.
  For example, 
\[
\xymatrix{\dessin{2.5cm}{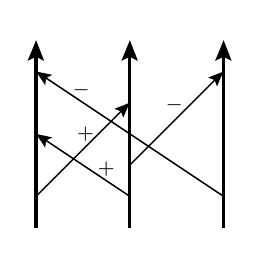} \ar@{|->}[r]^(.47){\fDA}& \left\{
{\begin{array}{ccl}
  x_1 & \mapsto & x_1^{x_2x_3^{-1}}\\
     x_2 & \mapsto & x_2^{x_1}\\
    x_3 & \mapsto & x_3^{x_2^{-1}}
  \end{array}}
\right..
}
\]
\end{remarque}

\begin{remarque}\label{rem:star2}
According to Remark \ref{rk:GenWTF}, $\fDA$ can be refined into an
action of $\GPn$ on $\Fn$.
\end{remarque} 

 \subsubsection{Definition of $\fAD$}\label{sec:fAD}

We will now define an inverse for $\fDA$. 
Let us first define, for every $i\in\UnN$, the map 
\[
 \func{\rho_i}{\AutC(\RFn)}{\RF^{(i)}_{n-1}}=R\langle x_1,\cdots,\hat{x}_i,\cdots,x_n\rangle
\]
as the unique function such that, for every $\varphi\in\AutC(\RFn)$, we have 
$\varphi(x_i)=x_i^{\rho_k(\varphi)}$.
\begin{lemme}
  For all $i\in\UnN$, the map $\rho_i$ is well defined.
\end{lemme}
\begin{proof}
  By definition of $\AutC(\RFn)$, there exists some $g\in\RFn$ such
  that $\varphi(x_i)=x_i^g$. 
  Since for every $g_1,g_2\in\RFn$ and every
  $\e\in\{\pm1\}$, we have the equality 
 \[
x_i^{g_1x_i^\e
  g_2}=g_2^{-1}x^{-\e}_ig_1^{-1}x_ig_1x_i^\e
  g_2=g_2^{-1}g_1^{-1}x_ig_1x^{-\e}_ix_i^\e
  g_2=g_2^{-1}g_1^{-1}x_ig_1g_2=x_i^{g_1g_2}
\] 
\noi  in $\RFn$, we may assume that $g$ is represented by a word which does not contain $x_i$,
  \ie that $g\in\RF^{(i)}_{n-1}$.
  Now, it remains to prove that this $g$ is uniquely defined in $RF^{(i)}_{n-1}$. 
  Suppose that there exists $g_1,g_2\in\RF^{(i)}_{n-1}$ such that
  $x_i^{g_1}=x_i^{g_2}$. Then $x_i$ commutes with
  $g_1g_2^{-1}$, that is, $[x_i,g_1g_2^{-1}]=1$. 
  Applying the Magnus expansion 
  $E$, defined in Section \ref{sec:ribbon-milnor-invariants}, to the latter equality gives that 
  $X_i.G - G.X_i=0$, 
  where $E(x_i)=1+X_i$ and $E(g_1g_2^{-1})=1+G$.  
  But since $g_1g_2^{-1}\in \RF^{(i)}_{n-1}$, the power series $G$ does not contain the variable $X_i$, 
  which implies that $X_i.G = G.X_i=0$. 
  So $G=0$, which by injectivity of the Magnus expansion (see \cite{MKS}) implies $g_1=g_2$.
\end{proof}

Now, we define a map
\[
\func{\eta}{\disp{\prod_{i=1}^n\RF^{(i)}_{n-1}}}{\wGDHn}
\]
as follows.  
Let  $(g_1,\cdots,g_n)\in \prod_{i=1}^n\RF^{(i)}_{n-1}$.  
For each $i\in\UnN$, choose a word 
$g_i=x_{j^i_1}^{\e_1^i}\cdots x_{j^i_{s_i}}^{\e_{s_i}^i}$ representing
$g_i$.
Then we define $\eta(g_1,\cdots,g_n)$ as the image in $\wGDHn$ of the
Gauss diagram obtained from the trivial one by adding successively,
according to the lexicographical order on
$\big\{(i,k)\big|i\in\UnN,k\in\llbracket1,s_i\rrbracket\big\}$,
an $\e_k^i$--signed arrow with head at the very top of $I_i$ and tail at the very bottom of $I_{j_k^i}$.
Note that this actually defines an ascending Gauss diagram. 
See Figure \ref{fig:phiAG} for an exemple.

\begin{lemme}\label{lem:mu}
  The map $\eta$ is well defined.
\end{lemme}
\begin{proof}
  The fact that $\eta$ does not depend on the order of the strands is guaranteed by relation TC, 
  so we only have to show that $\eta(g_1,\cdots,g_n)$ does not depend on the words representing the $n$ variables. 
  Fix $i\in \UnN$. Any two words representing $g_i$ differ by a finite number of the following moves
  \begin{enumerate}
  \item\label{cond:movei} $(1\leftrightarrow x_j^{-\e}x_j^{\e})$, for some $j\in\UnN$ and some $\e\in\{\pm1\}$,
  \item\label{cond:moveii} $(x_j x_j^g\leftrightarrow x_j^gx_j)$, for some $j\in\UnN$ 
  and some word $g:=x^{\e_1}_{i_1}\cdots x_{i_s}^{\e_s}\in\RF^{(i)}_{n-1}$,
  \end{enumerate}
  so we only need to check invariance of $\eta$ under these two types of moves.
  Invariance under move (\ref{cond:movei}) follows easily from move R2 on Gauss diagrams. 
  Now, let us consider move (\ref{cond:moveii}).  Using the identity
  $x_j^{g_1x_j^{\pm 1}g_2}
  =x_j^{g_1g_2}$,
  we can assume that the word $g$ involved in the move is in 
  $\RF^{(i,j)}_{n-2}:=R\langle x_1,\cdots,\hat{x_i},\hat{x_j},\cdots,x_n\rangle$.
  We need to prove that the following equality holds in $\wGDHn$ 
  \begin{equation}\label{eq:mu}
  \dessinH{3cm}{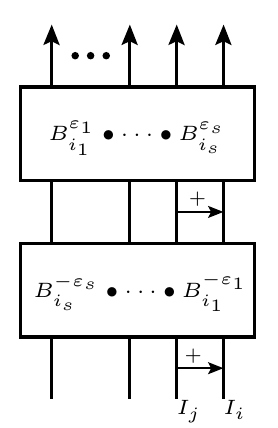}\ =\ \dessinH{3cm}{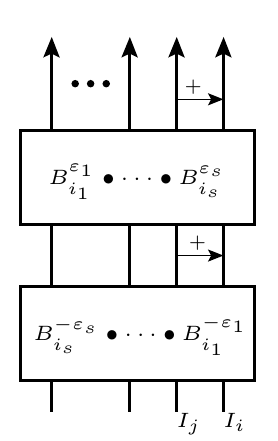},
  \end{equation}
  \noi where $B_k^\e$  stands for the Gauss diagram with only one 
  $\e$--labeled arrow 
  whose tail is on $I_k$ and head on $I_i$.
  Note that, in particular, all arrows in the boxes have their head on
  strand $I_i$.
  The proof is done by proving each side to be equal to a third Gauss diagram. 
  This is shown in two main steps, and the reader is encouraged to consult the example given in Figure \ref{fig:proofGD} while reading the following proof.  

  First, we pull the lowest ($+$--signed) arrow in the left-hand side Gauss diagram in (\ref{eq:mu}) across the lowest block 
  (i.e. across $B^{-\varepsilon_s}_{i_s}\bullet \dots  \bullet
  B^{-\varepsilon_1}_{i_1}$) by a sequence of moves TC and $\Ctu$ (with,
  in Figure \ref{fig:C3Moves}, $I_j$ being the rightmost strand, $I_i$ the
  middle one, and with all strands oriented downwards), and then use
  move $\Cdd$ to commute the heads of the two pictured $+$--signed arrows.
  This produces the following equality in $\wGDHn$ 
  \[
  \dessinH{3cm}{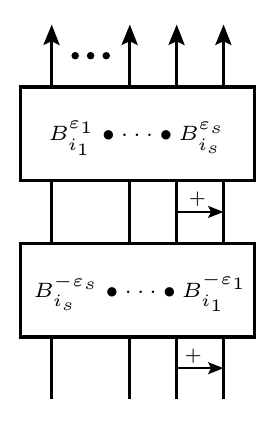}\ =\ \dessinH{2.4cm}{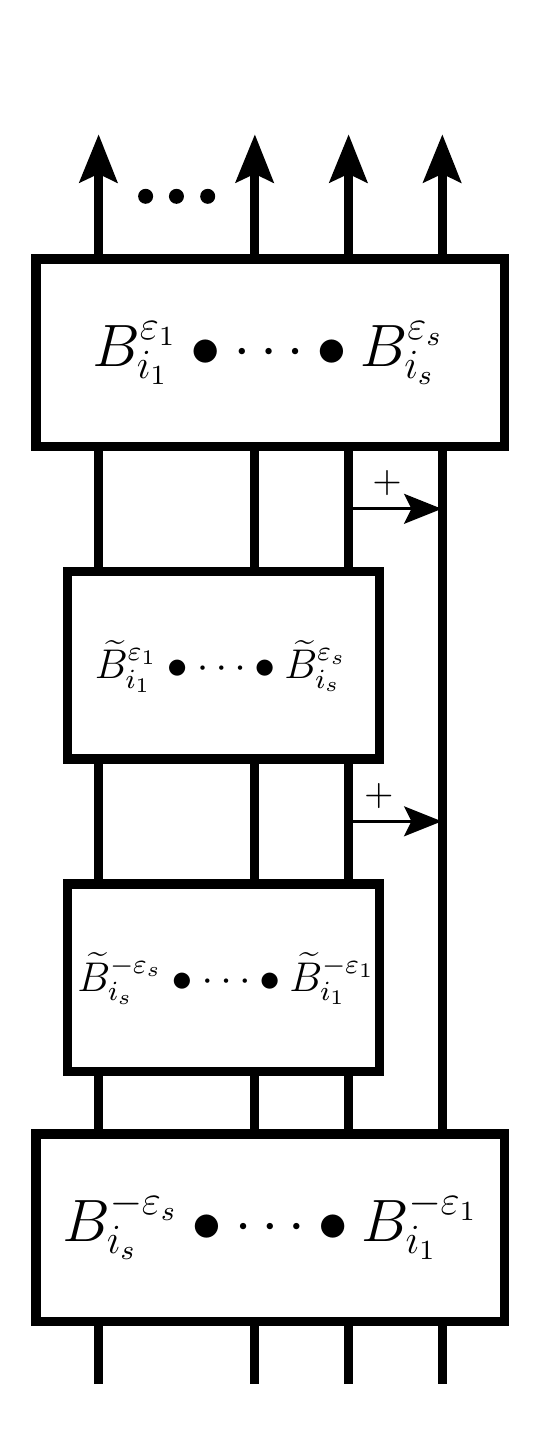}\ =\ \dessinH{2.4cm}{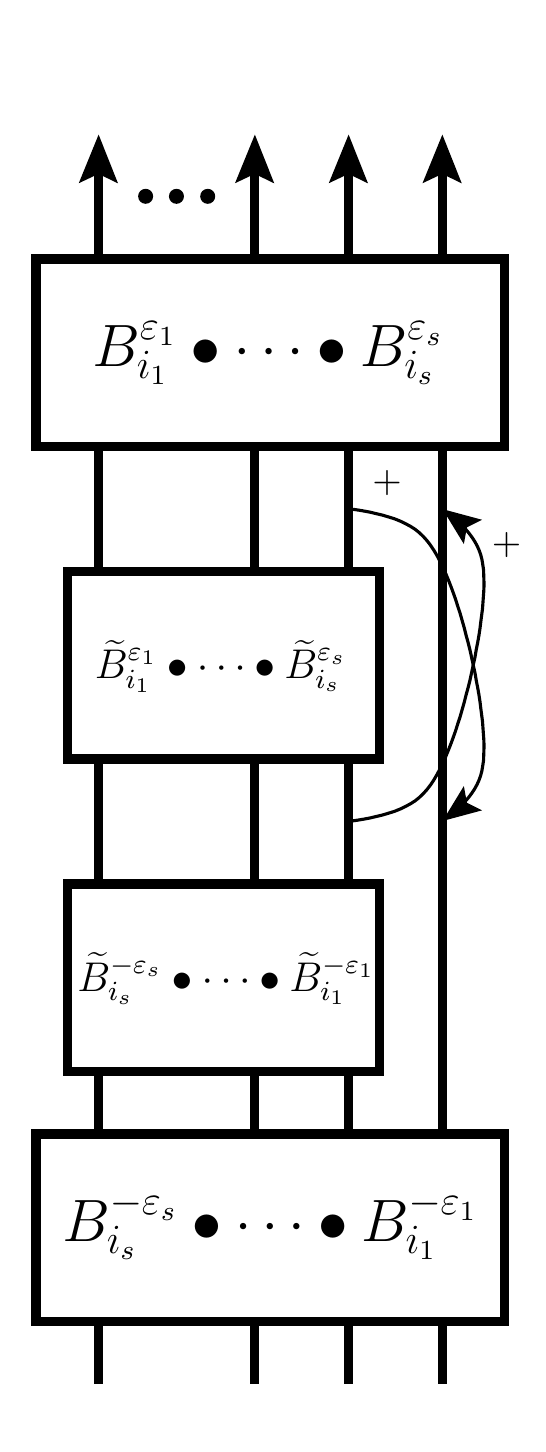},
  \]
  \noi where $\widetilde{B}_k^\e$ is obtained from $B_k^\e$ 
  by moving the head of each arrow lying on $I_i$ to the strand $I_j$ (preserving the
  horizontality and height of the arrow).

  Next, we consider the right-hand side Gauss diagram in (\ref{eq:mu}). Using TC moves, we pull the
  tail of the highest ($+$--signed) arrow just below the tail of the
  middle ($+$--signed) arrow.
Then, we pull its head across the highest block, using moves $\Ctu$ (with, in Figure \ref{fig:C3Moves}, $I_j$ the rightmost strand, $I_i$ the middle one, but with all strands oriented upwards).
This leads to
  \[
  \dessinH{3cm}{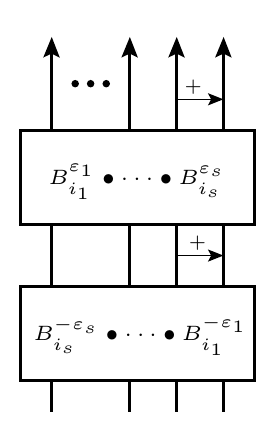}\ =\ \dessinH{3.1764cm}{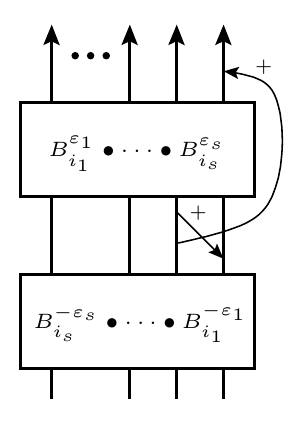}\ =\  \dessinH{2.4cm}{GDComm_4},
  \]
  \noi and proves that (\ref{eq:mu}) holds in $\wGDHn$.  
\end{proof}

The example in Figure \ref{fig:proofGD} illustrates both steps of the above proof.

\begin{figure}[!h]
  \[
\begin{array}{ccccccccc}
\dessinH{1.7cm}{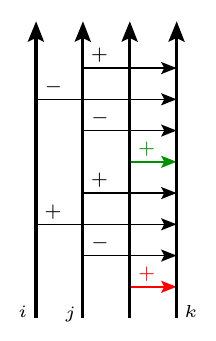}&\xrightarrow[]{\ \Ctt\
}&\dessinH{1.7cm}{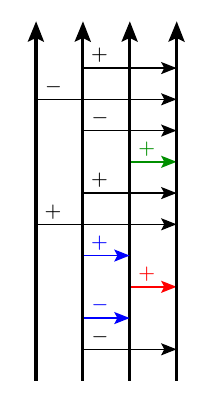}&\xrightarrow[]{\ \Ctt\
}&\dessinH{1.7cm}{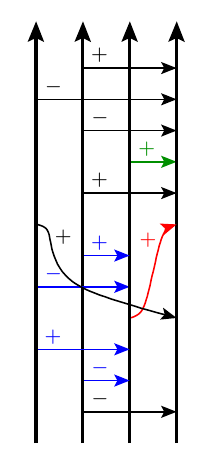}&\xrightarrow[]{\ \TC\
}&\dessinH{1.7cm}{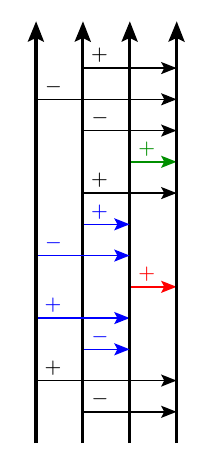}&\xrightarrow[\ \TC\ ]{\ \Ctt\ }&\dessinH{1.7cm}{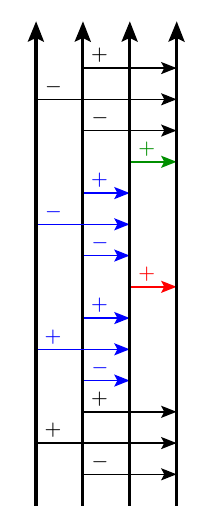}\\
&&&&&&&&\rotatebox{90}{$\ \xrightarrow[\ \rotatebox{270}{\scriptsize $\hspace{-.2cm}\Cdd$}\ ]{}\ $}\\
\dessinH{1.7cm}{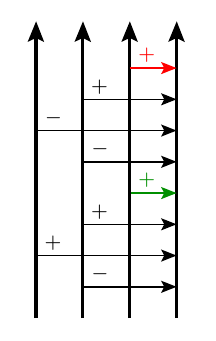}&\xrightarrow[]{\ \TC\ }&\dessinH{1.7cm}{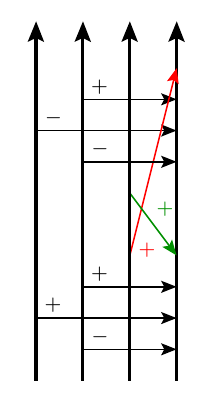}&\xrightarrow[\ \TC\ ]{\ \Ctt\ }&\dessinH{1.7cm}{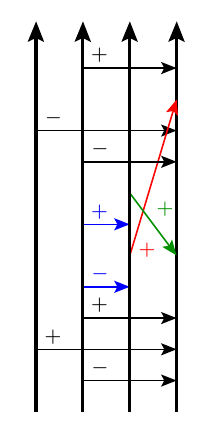}&\xrightarrow[\ \TC\ ]{\ \Ctt\ }&\dessinH{1.7cm}{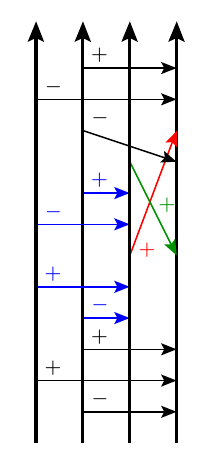}&\xrightarrow[\ \TC\ ]{\ \Ctt\ }&\dessinH{1.7cm}{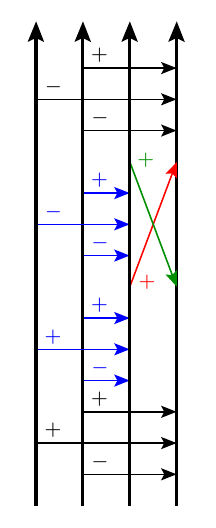}
\end{array}
\]
  \caption{Illustration of the two-step procedure in the proof of equality (\ref{eq:mu})}
  \label{fig:proofGD}
\end{figure}

\begin{defi}
The map   $\func{\fAD}{\AutC(\RFn)}{\wGDHn}$  
is defined as the composition $\eta\circ\big(\prod_{i=1}^n\rho_i\big)$.
\end{defi}

The example for $\fAD$ given in Figure \ref{fig:phiAG} is to be compared with Remark \ref{rk:asc}. 
 
\begin{figure}[h!]
\[
  \xymatrix@C=1.5cm@R=-1.2cm{
&(x_2x_3,x_3^{-1},x_3,x_3^{-1}x_1)&\\
\left\{ {\begin{array}{ccl}
          x_1 & \mapsto & x_1^{x_2x_3}\\[.15cm]
          x_2 & \mapsto & x_2^{x_3^{-1}}\\[.25cm]
          x_3 & \mapsto & x_3\\[.15cm]
          x_4 & \mapsto & x_4^{x_3^{-1}x_1}
        \end{array}}
    \right.
    \ar@<-.5cm>@{|->}[rr]_{\fAD}\ar@<.5cm>@{|->}[r]^{\prod \rho_i}&\textcolor{white}{(x_2x_3,x_3^{-1},x_3,x_3^{-1}x_1)}\ar@<.5cm>@{|->}[r]^{\eta}&
    \dessin{2.7cm}{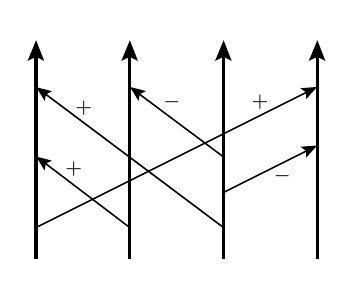}
  }
\]
\caption{An example for $\fAD$}\label{fig:phiAG}
\end{figure}

\begin{lemme}\label{lem:inverse}
  The map $\fAD$ is an inverse for $\fDA$.
\end{lemme}
\begin{proof}
  The relation $\fDA\circ\fAD=\Id_{\AutC(\RFn)}$ holds by construction. 
  The fact that $\fAD\circ\fDA=\Id_{\wGDHn}$
  is clear for ascending Gauss diagrams, which is enough, by Lemma \ref{lem:ascending}.
\end{proof}

We conclude with the following corollary which, curiously enough, is not easily checked directly. 
\begin{cor}
  The map $\fAD$ is a group homomorphism.
\end{cor}

 \subsection{Relation with welded string link and ribbon
   tubes}
\label{ref:CorrTTleMonde}

As the notation suggests, welded Gauss diagrams are merely combinatorial tools for describing and studying 
welded string links.

Indeed, to a string link $D$, one can associate a Gauss diagram
$G$ by considering the strands of $G$ as parametrization intervals for
the strands of $D$ and add arrows between the preimages of each classical crossing of $D$, oriented from the overpassing strand to the underpassing
one, and signed by the sign of the corresponding crossing.
Then Reidemeister moves on $D$ correspond to the eponymous moves on $G$,
$\OC$ corresponds to $\TC$ and a self-virtualization corresponds to
the addition or removal of a self-arrow.

Conversely, to a Gauss diagram $G$, one can associate a string link by
drawing one crossing $c_a$ for each arrow $a$ labelled by the sign of $a$,
and then, for each piece of strand of $G$ comprised between two ends
$e_1$ and $e_2$ ---with $e_1$ lower than $e_2$ and $a_1$
and $a_2$ the arrows that respectively contain $e_1$ and $e_2$--- connecting the
outgoing endpoint of the overpassing (resp. underpassing) strand of $c_{a_1}$ if $e_1$ is a
tail (resp. a head), to the ingoing endpoint of the overpassing (resp. underpassing) strand of
$c_{a_2}$ if $e_2$ is a
tail (resp. a head) ---this can always be done at the cost of adding
some virtual crossings--- and finally closing similarly the string link to its
bottom and top endpoints according to the lowest and highest
portions of strands of $G$. See Figure \ref{fig:DGD_Corr} for an example. 
This is well defined up to virtual moves and the correspondence
between moves on diagrams and moves on Gauss diagrams mentioned above
also holds.

\begin{figure}[!h]
  \[
\xymatrix{
\dessin{3.7cm}{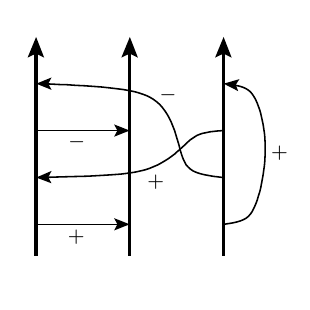}\ar@{^<-_>}[r]&\dessin{3.5cm}{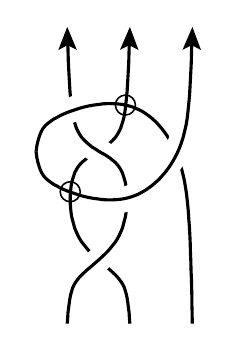}
}
  \]
  \caption{Correspondence between string link diagrams and Gauss diagrams}
  \label{fig:DGD_Corr}
\end{figure}

Moreover, it is easily seen that horizontality for Gauss diagrams corresponds to
monotony for string links.
As a consequence we have the following statement, which is considered as folklore in the literature:
\begin{prop}\label{prop:DiagGaussDiag}
$\wSLn\cong \wGDn$, $\wPn\cong \wGPn$, $\wSLHn\cong \wGDHn$ and
$\wPHn\cong \wGPHn$  as monoids.
Moreover, in these isomorphisms, there is a one-to-one correspondence
between
\begin{itemize}
\item classical crossings and arrows;
\item overstrands and tail intervals.
\end{itemize}
\end{prop}
\noindent Note that these correspondences already hold in the virtual setting.

Using these correspondences, constructions and results on
welded string link can be transferred to Gauss diagrams and {\it vice versa}. 
In particular, we obtain Theorems \ref{th:wSLh=wPh} and \ref{th:wSLh=AutC} as direct corollaries of Theorems \ref{th:wGD=wGP} and \ref{th:wGP=AutC}, respectively.  

Furthermore, Satoh's $\Tube$ map can be defined on $\wGDn$.
By abuse of notation, we will still denote it by $\Tube$.
Using Notation \ref{Nota:TailsIntervals}, it is a consequence of Proposition \ref{prop:DWirtinger}
that 
\begin{prop}\label{prop:GaussWirtinger}
  For every Gauss diagram,
\[
\pi_1\big(\Tube(G)\big)\cong\big\langle
T_G\ \big|\ T_h^+=(T_h^-)^{(T_h^0)^{\e_h}}\textrm{
for each arrow head }h\textrm{ of }G \big\rangle,
\]
\noi where $\e_h$ is the sign of the arrow that contains the head $h$.
Moreover, through this isomorphism and for every
$i\in\UnN$, the generator $T_i^-$ (resp. $T_i^+$)
corresponds to the loop that positively enlaces the $i^\textrm{th}$
bottom (resp. top) meridian of $\Tube(G)$.
\end{prop}
\begin{nota}\label{nota:WirtingerMap}
   For every Gauss diagram $G$, we denote by $W_G:T_G\to
   R\pi_1\big(\Tube(G)\big)$ the map induced by the isomorphism mentioned in the
   previous proposition.
\end{nota}

Recall that $\varphi$ denotes the isomorphism of the classification theorem \ref{th:Iso}. We have 
\begin{lemme}\label{lem:FVarphi}
  $\fDA=\varphi\circ\Tube$. 
\end{lemme}
\begin{proof}
  We freely use the notation of Section \ref{sec:Action}.  
  Let $G$ be a Gauss diagram.
  We consider $\iota_0$ and $\iota_1$ the inclusion
  maps associated to $\Tube(G)$.
  It follows from the Wirtinger presentation that the map ${\iota_0^*}^{-1}\circ
  W_G$ satisfies
  conditions (\ref{cond:pp}) and (\ref{cond:ps}) of Lemma \ref{lem:wtf}. By uniqueness, it follows
  that ${\iota^*_0}^{-1}\circ W_G$ is the coloring map $xi_G$ of this lemma.
  Moreover, by Proposition \ref{prop:GaussWirtinger} and the
  definition of $\iota_1$, we also know that
  $\iota_1^*(x_i)=W_G(T^+_i)$.
  So, finally, we get
\[
\varphi\big(\Tube(G)\big)(x_i)=\big({\iota_0^*}^{-1}\circ \iota_1^*\big)(x_i)={\iota_0^*}^{-1}\big(W_G(T^+_i)\big)=\xi_G(T^+_i).
\]
\end{proof}
Since $\fDA$ is an isomorphism, we obtain:
\begin{cor}\label{cor:TubeInj}
  The map $\Tube:\wGDHn\to\rTHn$ is injective.
\end{cor}

Since $\Tube$ was already known to be surjective, we obtain that the map $\varphi:\rTHn\to\AutC(\RFn)$ is a group isomorphism, as stated in Theorem \ref{th:Iso}.


 
 \section{Virtual Milnor invariants}
 \label{sec:milnor}

John Milnor defined in the fifties a family of link invariants, known today as Milnor's $\om$-invariants \cite{Milnor,Milnor2}.
Given an $n$--component link $L$ in the $3$--sphere, Milnor invariants $\om_I(L)$ of $L$ are defined for any finite sequence $I$ of (possibly repeating) indices in $\UnN$. 
Roughly speaking, Milnor's $\om$--invariants measure the longitudes in the lower central series of the fundamental group of the link complement.   
These invariants, however, are in general not well-defined integers, and their definition contains a rather intricate self-recurrent indeterminacy $\Delta$.
In \cite{HL2}, Habegger and Lin showed that the indeterminacy in Milnor invariants of a link is equivalent to the indeterminacy in representing this link as the closure of a string link, and that Milnor invariants are actually well defined invariants of string links; the latter integer--valued string link invariants are usually denoted $\mu_I$.  

Several authors have given tentative extensions of Milnor invariants to virtual knot theory \cite{DK,KP,kotorii}. 
They are all based on various combinatorial approaches of Milnor invariants, and are only partial extensions; in particular, they are limited to the case of link-homotopy invariants, that is, to the case of Milnor invariants indexed by sequences with no repetitions. 
In this section, we give a full extension of Milnor invariants to the virtual setting, in what seems to be the most natural way. 
Habegger and Lin's observation mentioned above suggests that it is most natural to consider the \emph{string link} case. 
Moreover, we have seen in the introduction that we are actually seeking for an invariant of \emph{welded} string links, rather than virtual, 
since Milnor invariant are extracted from the fundamental group, which is a welded invariant.
Finally, our construction is purely topological, as we shall see below, and well-behaved with respect to the Tube map, so that 
it naturally coincides with Milnor's original invariants when restricted to classical string links (see Theorem \ref{dlaballe}). 

\subsection{Milnor invariants for ribbon tubes}\label{sec:ribbon-milnor-invariants}

In this subsection, we develop a higher-dimensional analogue of Milnor invariants for ribbon tubes. 
This follows closely, and generalizes the construction given in Section \ref{sec:RedFundGr} in the case of the reduced free group. 
As a consequence, we freely borrow all notation from there.

Let $T$ be a ribbon tube with tube components $\psqcup_{i\in\UnN}A_i$.
and let $W=B^4\setminus \overstar{N}(T)$ be the complement of an open tubular neighborhood of $T$ in $B^4$.
By Proposition \ref{prop:TubeHomology}, the inclusion $\iota_\e:\p_\e
W\hookrightarrow W$ ($\e\in\{0,1\}$) induces isomorphisms both at the level of the first and second homology groups, so by Stallings theorem the maps 
\[
(\iota_\e)_k:\fract{\pi_1(\p_\e W)}/{\Gamma_k \pi_1(\p_\e W)}
\xrightarrow[]{\hspace{.3cm} \simeq\hspace{.3cm} }\fract{\pi_1(T)}/{\Gamma_k \pi_1(T)}
\]
\noi are isomorphisms for every $k\in \mathbb{N}^*$.

Recall from Remark \ref{rk:preflong} that the preferred $i^\textrm{th}$ longitude $\lambda_i$ of $T$ is longitude of $A_i$ having linking number zero 
with $A_i$.
For any $k\in\N$ and $i\in\UnN$, we define
 $\lambda^k_i:=(\iota_0)^{-1}_k(\lambda_i)\in\fract{\Fn}/{\Gamma_k\Fn}$. 
 
Denote by $\Z\langle\langle X_1,\cdots,X_n\rangle\rangle$ the ring of
formal power series in non-commutative variables $X_1,\cdots,X_n$.
The \emph{Magnus expansion} $E:\Fn\to \Z\langle\langle
X_1,\cdots,X_n\rangle\rangle$ is the injective group homomorphism
which maps $x_i$ to $1 + X_i$ and $x_i^{-1}$ to
$\psum_{k\in\N}(-1)^kX_i^k$, for each $i\in\UnN$.  
\begin{prop}
  For any $m\in\N^*$, any $(i_1,\cdots,i_m)\in\UnN^m$,
  any $j\in\UnN$ and any integer $k>m$, the coefficient $\mu^{(4)}_{i_1,\cdots,i_m j}(L)$ of
  $X_{i_1}\cdots X_{i_m}$ in the Magnus expansion $E(\lambda_j^k)$ is a well-defined invariant of $T$.
\end{prop}
\begin{defi}
The coefficient $\mu^{(4)}_{i_1. . .i_k}$ of $X_{i_1}\cdots X_{i_m}$ in  $E(\lambda_j^k)$ is called a \emph{Milnor $\mu^{(4)}$--invariant} of length $m+1$. 
\end{defi}
\noindent Here, the $(4)$-exponent refers to the $4$--dimensional nature of ribbon tubes.

For each non negative integer $k$, we define $\rTn(k)$ to be the
submonoid of $\rTn$ consisting of elements with vanishing Milnor $\mu^{(4)}$--invariants
of length lower than $k$.
We thus have a descending filtration of monoids
$\rTn = \rTn(1) \supset \rTn(2) \supset\cdots \supset \rTn(k) \supset
\cdots$, called the \emph{Milnor filtration}.

Now, we generalize the construction given in Section \ref{sec:Action}.
Indeed, for each non negative integer, a ribbon tube
$T$ induces an automorphism $(\iota_0)_k^{-1} \circ(\iota_1)_k$ of
$\fract{\Fn}/{\Gamma_k\Fn}$.
Actually, this automorphism maps $x_i$ to its conjugate by $\lambda_i^k$, for each $i\in\UnN$,  so that we have
defined a monoid homomorphism
\[
A_k : \rTn \to \AutC \big(\fract{\Fn}/{\Gamma_k\Fn}\big).
\]
One can check that $T$ is in $\rTn(k)$ if and only if $\lambda_i^k$ is
trivial modulo $\Gamma_k\Fn$ for all $i\in\UnN$.
So for all non negative integer $k$, we have $\rTn(k) = \Ker(A_k)$,
and we can consider the map
\[
\mu^{(4)}_{k+1} : \rTn(k) \to\fract{\Fn}/{\Gamma_2\Fn}\otimes \fract{\Gamma_k\Fn}/{\Gamma_{k+1}\Fn}.
\]
\noi which maps $T$ to the sum
\[
\mu^{(4)}_{k+1}(T):=\sum_{i\in\UnN} x_i\otimes \lambda^{k+1}_i,
\]
\noi We call it \emph{the universal length k + 1 Milnor invariant}.
This map $\mu^{(4)}_{k+1}$ is strictly equivalent to the collection of all Milnor $\mu^{(4)}$--invariants of length $k + 1$, via the formula
$\mu^{(4)}_{k+1}(L) :=\psum_{i\in\UnN}x_i \otimes E_k(\lambda_i^{k+1})$, where $E_k$ denotes the composition of $E$ with the
projection onto the degree $k$ part, that is
\[
E_k(\lambda_j^{k+1}):=\sum_{i_1,\cdots,i_k\in\UnN^k}\mu^{(4)}_{i_1,\cdots,i_k j}X_{i_1}\cdots
X_{i_k}.
\]
Note that $E_k(\lambda_j^{k+1})$ lives in the isomorphic image of
$\fract{\Gamma_k\Fn}/{\Gamma_{k+1}\Fn}$ in $\Z\langle\langle X_1,\cdots,X_n\rangle\rangle$.

\subsection{Milnor invariants for welded string links}\label{sec:welded-milnor-invariants}

The above $4$-dimensional version $\mu^{(4)}$ of Milnor invariants provides, through the Tube map, 
a natural and general extension $\mu^w$ of Milnor invariants to virtual/welded string links. 
\begin{defi}
For any sequence $I$ of (possibly repeating) indices in $\UnN$, 
\emph{Milnor invariant} $\mu^w_I$ of $n$-component welded string links is defined by 
 $$ \mu^w_{I} := \mu^{(4)}_{I}\circ Tube. $$
\end{defi}

The construction of Milnor invariants for ribbon tubes given in Section \ref{sec:ribbon-milnor-invariants}
is completely parallel to the definition of Milnor invariants of classical string links (as given e.g. in \cite{HL,HM}).
Moreover, the Tube map is well-behaved with respect to the fundamental group of the complement;
not only does it induce an isomorphism of $\pi_1$'s, but it also maps meridians to meridians and (preferred) longitudes to (preferred) longitudes,
so that the Wirtinger presentations are in one--to--one correspondence (see Proposition \ref{prop:DWirtinger} and the discussion preceding it).
As a consequence,  we have the following.
\begin{theo}\label{dlaballe}
For any sequence $I$ of (possibly repeating) indices in $\UnN$, the restriction of $\mu^w_{I}$ to classical string links coincides with Milnor invariant :
 $$ \mu^w_I(L) = \mu_I(L) \; \textrm{, for any classical string link $L$.} $$
\end{theo}

\subsection{Comparison with previous works}\label{sec:compare}

Let us now discuss previously existing (partial) virtual extensions of Milnor invariants.

The first virtual extension of Milnor invariants is due to Dye and Kauffman \cite{DK}. There, it is already noted that these are actually welded invariants. 
The authors, however, restricted themselves to link--homotopy invariants, and focussed on the link case.  
This means in particular that the Dye-Kauffman extensions are defined modulo a certain indeterminacy $\Delta$. 
Their construction follows very closely Milnor's original work \cite{Milnor2}, using the virtual knot group, 
and actually only slightly deviates from  \cite{Milnor2} in the definition of the indeterminacy $\Delta$: since the authors only seek a link--homotopy invariant, 
a simpler definition can indeed be used. 
However, the Dye-Kauffman extension does not always coincide with Milnor invariants for classical links, 
precisely because of this different choice of indeterminacy. 
Consider for example a $3$-component link $L$ obtained from the Borromean rings by inserting a positive clasp between the first and second components; 
according to Milnor's definition we have $\Delta(123)=1$, but in Dye--Kauffman's definition we have 
$\Delta(123)=0$, so that the invariants $\om_{123}$ of the link $L$ do not agree in both definitions. 

Recently, Kotorii gave in \cite{kotorii} another virtual extension of Milnor invariants, using Turaev's theory of nanowords \cite{turaevw} 
to describe and study virtual links. 
This extension is thus purely combinatorial and, again, addresses the link case and is only valid for link--homotopy invariants. 
Although it indeed coincides with Milnor's original invariants when restricted to classical links, 
it seems quite challenging to us to generalize directly the Kotorii extension to all Milnor invariants. 
Indeed, the fact that Milnor invariants are indexed by sequences with
no repetitions is used to ensure invariance under Reidemeister 3 move (see Propositions 6.4 \cite{kotorii}).  

The third virtual extension of Milnor invariants is due to Kravchenko and Polyak \cite{KP}. 
The authors actually provides Gauss diagram formula for $\mu$-invariants, and their extension is therefore closest to the extension given above.  
The result of \cite{KP} holds for string links (i.e. provides integer--valued virtual extensions), but as the preceding ones it is only valid for Milnor 
link--homotopy invariants. 
Indeed, identifying the Gauss diagram invariants defined in \cite{KP} with Milnor invariants relies on Polyak's skein formula \cite{polyak}, 
which is only valid for sequences without repetitions. \\
It would be interesting to generalize the Kravchenko--Polyak extension to all sequences 
(or at least, to identify the resulting Gauss diagram formulas with Milnor invariants).
This seem to be a non-obvious problem, in particular because the string link analogue of  \cite[Thm.~7]{Milnor2}
(which, roughly speaking, shows how Milnor's link--homotopy invariants suffice to generate all $\om$--invariants of links via cabling)
does not hold in full generality, as outlined in \cite[Sec.~3]{yasuharaAGT}.
Actually, we expect that one way to address this question would be by identifying  the extension of Kravchenko--Polyak's Gauss diagram formulas
with the `topologically--defined' extension $\mu^w$ of the present
paper; this identification is at least clear in the case of
non-repeated indices.


\bibliographystyle{abbrv}
\bibliography{wSL}

\end{document}